\documentclass[12pt,reqno]{amsart} 
\usepackage{amssymb,amscd,url}

\begin{document}

\def\framebox#1{\hbox{#1}} 

\allowdisplaybreaks


\title[Degree $2$ maps $f:\PP^2\to\PP^2$ with large $\Aut(f)$]
      {A classification of degree $2$ semi-stable rational maps $\PP^2\to\PP^2$
        with large finite dynamical automorphism group}
\date{\today}
\author{Michelle Manes}
\email{mmanes@math.hawaii.edu}
\address{Department of Mathematics, University of Hawaii, Honolulu, HI
  96822}
\author[Joseph H. Silverman]{Joseph H. Silverman}
\email{jhs@math.brown.edu}
\address{Mathematics Department, Box 1917
         Brown University, Providence, RI 02912 USA}
\thanks{Silverman's research supported by Simons Collaboration Grant
  \#241309}



\hyphenation{ca-non-i-cal semi-abel-ian}


\newtheorem{theorem}{Theorem}
\newtheorem{lemma}[theorem]{Lemma}
\newtheorem{sublemma}[theorem]{Sublemma}
\newtheorem{conjecture}[theorem]{Conjecture}
\newtheorem{proposition}[theorem]{Proposition}
\newtheorem{corollary}[theorem]{Corollary}
\newtheorem*{claim}{Claim}

\theoremstyle{definition}
\newtheorem*{definition}{Definition}
\newtheorem{example}[theorem]{Example}
\newtheorem{remark}[theorem]{Remark}
\newtheorem{question}[theorem]{Question}

\theoremstyle{remark}
\newtheorem*{acknowledgement}{Acknowledgements}


\newenvironment{notation}[0]{%
  \begin{list}%
    {}%
    {\setlength{\itemindent}{0pt}
     \setlength{\labelwidth}{4\parindent}
     \setlength{\labelsep}{\parindent}
     \setlength{\leftmargin}{5\parindent}
     \setlength{\itemsep}{0pt}
     }%
   }%
  {\end{list}}

\newenvironment{parts}[0]{%
  \begin{list}{}%
    {\setlength{\itemindent}{0pt}
     \setlength{\labelwidth}{1.5\parindent}
     \setlength{\labelsep}{.5\parindent}
     \setlength{\leftmargin}{2\parindent}
     \setlength{\itemsep}{0pt}
     }%
   }%
  {\end{list}}
\newcommand{\Part}[1]{\item[\upshape#1]}

\def\Case#1#2{%
\paragraph{\textbf{\boldmath Case #1: #2.}}\hfil\break\ignorespaces}

\renewcommand{\a}{\alpha}
\renewcommand{\b}{\beta}
\newcommand{\g}{\gamma}
\renewcommand{\d}{\delta}
\newcommand{\e}{\epsilon}
\newcommand{\f}{\varphi}
\newcommand{\bfphi}{{\boldsymbol{\f}}}
\renewcommand{\l}{\lambda}
\renewcommand{\k}{\kappa}
\newcommand{\lhat}{\hat\lambda}
\newcommand{\m}{\mu}
\newcommand{\bfmu}{{\boldsymbol{\mu}}}
\renewcommand{\o}{\omega}
\renewcommand{\r}{\rho}
\newcommand{\rbar}{{\bar\rho}}
\newcommand{\s}{\sigma}
\newcommand{\sbar}{{\bar\sigma}}
\renewcommand{\t}{\tau}
\newcommand{\z}{\zeta}

\newcommand{\D}{\Delta}
\newcommand{\G}{\Gamma}
\newcommand{\F}{\Phi}
\renewcommand{\L}{\Lambda}

\newcommand{\ga}{{\mathfrak{a}}}
\newcommand{\gb}{{\mathfrak{b}}}
\newcommand{\gn}{{\mathfrak{n}}}
\newcommand{\gp}{{\mathfrak{p}}}
\newcommand{\gP}{{\mathfrak{P}}}
\newcommand{\gq}{{\mathfrak{q}}}

\newcommand{\Abar}{{\bar A}}
\newcommand{\Ebar}{{\bar E}}
\newcommand{\kbar}{{\bar k}}
\newcommand{\Kbar}{{\bar K}}
\newcommand{\Pbar}{{\bar P}}
\newcommand{\Sbar}{{\bar S}}
\newcommand{\Tbar}{{\bar T}}
\newcommand{\gbar}{{\bar\gamma}}
\newcommand{\lbar}{{\bar\lambda}}
\newcommand{\ybar}{{\bar y}}
\newcommand{\phibar}{{\bar\f}}

\newcommand{\Acal}{{\mathcal A}}
\newcommand{\Bcal}{{\mathcal B}}
\newcommand{\Ccal}{{\mathcal C}}
\newcommand{\Dcal}{{\mathcal D}}
\newcommand{\Ecal}{{\mathcal E}}
\newcommand{\Fcal}{{\mathcal F}}
\newcommand{\Gcal}{{\mathcal G}}
\newcommand{\Hcal}{{\mathcal H}}
\newcommand{\Ical}{{\mathcal I}}
\newcommand{\Jcal}{{\mathcal J}}
\newcommand{\Kcal}{{\mathcal K}}
\newcommand{\Lcal}{{\mathcal L}}
\newcommand{\Mcal}{{\mathcal M}}
\newcommand{\Ncal}{{\mathcal N}}
\newcommand{\Ocal}{{\mathcal O}}
\newcommand{\Pcal}{{\mathcal P}}
\newcommand{\Qcal}{{\mathcal Q}}
\newcommand{\Rcal}{{\mathcal R}}
\newcommand{\Scal}{{\mathcal S}}
\newcommand{\Tcal}{{\mathcal T}}
\newcommand{\Ucal}{{\mathcal U}}
\newcommand{\Vcal}{{\mathcal V}}
\newcommand{\Wcal}{{\mathcal W}}
\newcommand{\Xcal}{{\mathcal X}}
\newcommand{\Ycal}{{\mathcal Y}}
\newcommand{\Zcal}{{\mathcal Z}}

\renewcommand{\AA}{\mathbb{A}}
\newcommand{\BB}{\mathbb{B}}
\newcommand{\CC}{\mathbb{C}}
\newcommand{\FF}{\mathbb{F}}
\newcommand{\GG}{\mathbb{G}}
\newcommand{\NN}{\mathbb{N}}
\newcommand{\PP}{\mathbb{P}}
\newcommand{\QQ}{\mathbb{Q}}
\newcommand{\RR}{\mathbb{R}}
\newcommand{\ZZ}{\mathbb{Z}}

\newcommand{\bfa}{{\boldsymbol a}}
\newcommand{\bfb}{{\boldsymbol b}}
\newcommand{\bfc}{{\boldsymbol c}}
\newcommand{\bfd}{{\boldsymbol d}}
\newcommand{\bfe}{{\boldsymbol e}}
\newcommand{\bff}{{\boldsymbol f}}
\newcommand{\bfg}{{\boldsymbol g}}
\newcommand{\bfi}{{\boldsymbol i}}
\newcommand{\bfj}{{\boldsymbol j}}
\newcommand{\bfp}{{\boldsymbol p}}
\newcommand{\bfr}{{\boldsymbol r}}
\newcommand{\bfs}{{\boldsymbol s}}
\newcommand{\bft}{{\boldsymbol t}}
\newcommand{\bfu}{{\boldsymbol u}}
\newcommand{\bfv}{{\boldsymbol v}}
\newcommand{\bfw}{{\boldsymbol w}}
\newcommand{\bfx}{{\boldsymbol x}}
\newcommand{\bfy}{{\boldsymbol y}}
\newcommand{\bfz}{{\boldsymbol z}}
\newcommand{\bfA}{{\boldsymbol A}}
\newcommand{\bfF}{{\boldsymbol F}}
\newcommand{\bfB}{{\boldsymbol B}}
\newcommand{\bfD}{{\boldsymbol D}}
\newcommand{\bfG}{{\boldsymbol G}}
\newcommand{\bfI}{{\boldsymbol I}}
\newcommand{\bfM}{{\boldsymbol M}}
\newcommand{\bfP}{{\boldsymbol P}}
\newcommand{\bfzero}{{\boldsymbol{0}}}
\newcommand{\bfone}{{\boldsymbol{1}}}

\newcommand{\Aut}{\operatorname{Aut}}
\newcommand{\Birat}{\operatorname{Birat}}
\newcommand{\codim}{\operatorname{codim}}
\newcommand{\Crit}{\operatorname{Crit}}
\newcommand{\diag}{\operatorname{diag}}
\newcommand{\Disc}{\operatorname{Disc}}
\newcommand{\Div}{\operatorname{Div}}
\newcommand{\Dom}{\operatorname{Dom}}
\newcommand{\End}{\operatorname{End}}
\newcommand{\Fbar}{{\bar{F}}}
\newcommand{\Fix}{\operatorname{Fix}}
\newcommand{\Gal}{\operatorname{Gal}}
\newcommand{\GL}{\operatorname{GL}}
\newcommand{\Hom}{\operatorname{Hom}}
\newcommand{\Index}{\operatorname{Index}}
\newcommand{\Image}{\operatorname{Image}}
\newcommand{\Isom}{\operatorname{Isom}}
\newcommand{\hhat}{{\hat h}}
\newcommand{\Ker}{{\operatorname{ker}}}
\newcommand{\Lift}{\operatorname{Lift}}
\newcommand{\limstar}{\lim\nolimits^*}
\newcommand{\limstarn}{\lim_{\hidewidth n\to\infty\hidewidth}{\!}^*{\,}}
\newcommand{\Mat}{\operatorname{Mat}}
\newcommand{\maxplus}{\operatornamewithlimits{\textup{max}^{\scriptscriptstyle+}}}
\newcommand{\MOD}[1]{~(\textup{mod}~#1)}
\newcommand{\Mor}{\operatorname{Mor}}
\newcommand{\Moduli}{\mathcal{M}}
\newcommand{\Norm}{{\operatorname{\mathsf{N}}}}
\newcommand{\notdivide}{\nmid}
\newcommand{\normalsubgroup}{\triangleleft}
\newcommand{\NS}{\operatorname{NS}}
\newcommand{\onto}{\twoheadrightarrow}
\newcommand{\ord}{\operatorname{ord}}
\newcommand{\Orbit}{\mathcal{O}}
\newcommand{\Per}{\operatorname{Per}}
\newcommand{\Perp}{\operatorname{Perp}}
\newcommand{\PrePer}{\operatorname{PrePer}}
\newcommand{\PGL}{\operatorname{PGL}}
\newcommand{\Pic}{\operatorname{Pic}}
\newcommand{\Prob}{\operatorname{Prob}}
\newcommand{\Proj}{\operatorname{Proj}}
\newcommand{\Qbar}{{\bar{\QQ}}}
\newcommand{\rank}{\operatorname{rank}}
\newcommand{\Rat}{\operatorname{Rat}}
\newcommand{\Resultant}{\operatorname{Res}}
\renewcommand{\setminus}{\smallsetminus}
\newcommand{\sgn}{\operatorname{sgn}} 
\newcommand{\SL}{\operatorname{SL}}
\newcommand{\Span}{\operatorname{Span}}
\newcommand{\Spec}{\operatorname{Spec}}
\renewcommand{\ss}{\textup{ss}}
\newcommand{\stab}{\textup{stab}}
\newcommand{\Stab}{\operatorname{Stab}}
\newcommand{\Support}{\operatorname{Supp}}
\newcommand{\tors}{{\textup{tors}}}
\newcommand{\Trace}{\operatorname{Trace}}
\newcommand{\trianglebin}{\mathbin{\triangle}} 
\newcommand{\tr}{{\textup{tr}}} 
\newcommand{\UHP}{{\mathfrak{h}}}    
\newcommand{\<}{\langle}
\renewcommand{\>}{\rangle}

\newcommand{\pmodintext}[1]{~\textup{(mod}~#1\textup{)}}
\newcommand{\ds}{\displaystyle}
\newcommand{\longhookrightarrow}{\lhook\joinrel\longrightarrow}
\newcommand{\longonto}{\relbar\joinrel\twoheadrightarrow}
\newcommand{\SmallMatrix}[1]{%
  \left(\begin{smallmatrix} #1 \end{smallmatrix}\right)}


\begin{abstract}
Let $K$ be an algebraically closed field of characteristic $0$.  In
this paper we classify the $\text{PGL}_3(K)$-conjugacy classes of
semi-stable dominant degree $2$ rational maps
$f:{\mathbb  P}^2_K\dashrightarrow{\mathbb P}^2_K$ whose
automorphism group
$$\text{Aut}(f):=\bigl\{\phi\in\text{PGL}_3(K): \phi^{-1}\circ f\circ\phi=f\bigr\}$$
is finite and of order at least $3$.  In particular, we prove that
$\#\text{Aut}(f)\le24$ in general, that $\#\text{Aut}(f)\le21$ for
morphisms, and that $\#\text{Aut}(f)\le6$ for all but finitely many
conjugacy classes of $f$.
\par\vspace{10pt}\noindent\textsc{Abstract}.\enspace
Soit $K$ un corps alg\'ebriquement clos de charac\-t\'er\-is\-ktique
$0$. Dans cet article nous classifions les $\text{PGL}_3(K)$-classes
de conjugaison de fonctions rationelles
$f:{\mathbb P}^2_K\dashrightarrow{\mathbb P}^2_K$ de degr\'e $2$
dominantes et semi-stables dont le groupe d'automorphismes
$$\text{Aut}(f):=\bigl\{\phi\in\text{PGL}_3(K): \phi^{-1}\circ f\circ\phi=f\bigr\}$$
est fini et d'ordre au moins $3$. En particulier, nous d\'emontrons que
$\#\text{Aut}(f)\le24$ en g\'en\'eral, que $\#\text{Aut}(f)\le21$ pour
les morphismes et que $\#\text{Aut}(f)\le6$ pour toutes except\'e un
nombre fini de classes de conjugaisons de $f$.
\end{abstract}

\maketitle

\tableofcontents

\section{Introduction}

Let $d\ge1$ and $N\ge1$ be integers and let
\[
  L = L(N,d) = \binom{N+d}{d}(N+1)-1.
\]
We identify~$\PP^L$ with the space of $(N+1)$-tuples of
homogeneous polynomials of degree~$d$ in~$N+1$ variables such that at
least one polynomial is non-zero. Thus each~$f=[f_0,\ldots,f_N]\in\PP^L$ 
defines a rational map
\[
  f : \PP^N\dashrightarrow\PP^N.
\]
Although map~$f$ need not be dominant, nor, if it is dominant, need it
have degree~$d$, we adopt the notation
\[
  \Rat_d^N:=\PP^L
\]
and call~$\Rat_d^N$ the parameter space of rational self-maps
of~$\PP^N$ of \emph{formal degree~$d$}.

The group~$\PGL_{N+1}$ acts on~$\Rat_d^N$ via conjugation, i.e., the
action of~$\f\in\PGL_{N+1}$ on $f\in\Rat_d^N$ is
\[
  f^\f := \f^{-1}\circ f\circ\f.
\]
This gives a homomorphism
\[
  \PGL_{N+1}\longrightarrow\Aut(\Rat_d^N)=\Aut(\PP^L)\cong\PGL_{L+1}.
\]
Geometric invariant theory~\cite{mumford:geometricinvarianttheory}
tells us that there are subsets~$(\Rat_d^N)^\stab$ and $(\Rat_d^N)^\ss$
of stable and semi-stable points in~$\Rat_d^N$ which admit good
quotients for the action of~$\PGL_{N+1}$.\footnote{More precisely,
  they admit good $\SL_{N+1}$-quotients;
  cf.\ \cite[\S~2.1]{MR2884382}.}  We denote these quotients by
$(\Moduli_d^N)^\stab$ and $(\Moduli_d^N)^\ss$.

In this note we are interested in the locus in~$\Rat_d^N$ of maps that
admit a non-trivial automorphism.

\begin{definition}
The \emph{automorphism group} of a map~$f\in\Rat_d^N$ is
\[
  \Aut(f) = \{ \f\in\PGL_{N+1} : f^\f=f\}.
\]
We note that $\Aut(f^\f)=\Aut(f)^\f$. In particular, the isomorphism
type of~$\Aut(f)$ is a $\PGL_{N+1}$-conjugation invariant.
\end{definition}

\begin{remark}
It is known that $(\Moduli_2^1)^\stab = (\Moduli_2^1)^\ss \cong \PP^2$
and that $\bigl\{f\in (\Moduli_2^1)^\stab : \#\Aut(f)\ge2\bigr\}$ is a
cuspidal cubic curve in~$\PP^2$. More precisely, for~$f$ on this
curve, $\Aut(f)\cong C_2$ for the non-cuspidal points and
$\Aut(f)\cong S_3$ at the cuspidal point;
see~\cite[Proposition~4.15]{MR2884382}.  We postpone to
Section~\ref{section:background} an overview of our
current knowledge of maps having non-trivial automorphism group.
\end{remark}

Our primary goal in this paper is to describe the degree~$2$ maps
in~$(\Moduli_2^2)^\ss$ having large finite automorphism group, i.e.,
we want to extend the above-mentioned classification of quadratic maps
on~$\PP^1$ to quadratic maps on~$\PP^2$. We mention that a number of
new phenomena appear, including semi-stable dominant rational
quadratic maps~$f:\PP^2\dashrightarrow\PP^2$ for which~$\Aut(f)$
contains a copy of~$\GG_m$. For reasons that we explain later, we
mostly exclude these maps from our analysis; see
Section~\ref{section:quotientmap}. We also do not study maps with
$\Aut(f)\cong C_2$, since they are too plentiful.

Before stating our main results, we need some additional notation.  In
general, for any finite subgroup $\Gcal\subseteq\PGL_{N+1}$, we
consider
\[
  \Rat_d^N(\Gcal) := \bigl\{ f\in\Rat_d^N : \Aut(f)\supseteq\Gcal \bigr\}.
\]
If~$\Gcal^\f$ is a conjugate subgroup, then
\[
  \Rat_d^N(\Gcal) \xrightarrow{\;\sim\;}\Rat_d^N(\Gcal^\f),\quad
  f\xrightarrow{\;\sim\;}f^\f,
\]
so it suffices to study~$\Rat_d^N(\Gcal)$ for each conjugacy class of finite
subgroup in~$\PGL_{N+1}$.

It is important to note that~$\PGL_{N+1}$  generally does not act on
$\Rat_d^N(\Gcal)$, since if $f\in\Rat_d^N(\Gcal)$ and
$\f\in\PGL_{N+1}$, then $\Aut(f^\f)=\Aut(f)^\f\supseteq\Gcal^\f$. Thus
in order to ensure that~$f^\f$ is in $\Rat_d^N(\Gcal)$, we need
$\Gcal^\f=\Gcal$, i.e., the map~$\f$ must be in the
normalizer~$N(\Gcal)$ of~$\Gcal$. We thus define\footnote{Again, we
  are being somewhat informal in this introduction. To be rigorous, we
  lift~$\Gcal$ to an isomorphic subgroup~$\tilde{\Gcal}$
  in~$\SL_{N+1}$ and look at maps~$f$ that
  are~$N(\tilde{\Gcal})$-semistable.}
\[
  \Rat_d^N(\Gcal)^\ss := \bigl\{f\in\Rat_d^N(\Gcal) :
     \text{$f$ is $N(\Gcal)$-semistable} \bigr\}, \\
\]
and similarly~$\Rat_d^N(\Gcal)^\stab$ denotes the set of~$N(\Gcal)$-stable maps.
It turns out that if $f\in\Rat_d^N(\Gcal)$ is $N(\Gcal)$-semistable,
then~$f$ is also $\PGL_{N+1}$-semistable when viewed as a point in~$\Rat_d^N$,
and further the natural map
\begin{equation}
  \label{eqn:RmodNRmodPGL}
  \Rat_d^N(\Gcal)^\ss/N(\Gcal) \longrightarrow  (\Rat_d^N)^\ss/\PGL_{N+1}
\end{equation}
is finite; see Proposition~\ref{proposition:XGXHNH} for a general
result. However, the map~\eqref{eqn:RmodNRmodPGL} may fail to be
injective due to the existence of~$f$'s that are~$\PGL_{N+1}$-conjugate,
but are not~$N(\Gcal)$-conjugate; see Example~\ref{example:notinj}.

Our main results give a complete descriptoin of semistable dominant
rational quadratic maps $f:\PP^2\dashrightarrow\PP^2$ satisfying
$3\le\#\Aut(f)<\infty$.

\begin{theorem}
\label{theorem:main1}
Let~$K$ be an algebraically closed field of characteristic~$0$, and
let~$\Gcal\subset\PGL_3(K)$ be a finite subgroup with $\#\Gcal\ge3$.
Suppose that there exists a map $f\in\Rat_2^2(\Gcal)^\ss(K)$ satisfying
\begin{parts}
\Part{\textbullet}
$f:\PP^2\dashrightarrow\PP^2$ is dominant.
\Part{\textbullet}
$\deg(f)=2$.
\Part{\textbullet}
$\Aut(f)$ is finite
\end{parts}
Then there is a $\PGL_3(K)$-conjugate of~$\Gcal$ that contains one of
the following groups, where~$\z_n$ is a primitive $n$'th root of
unity\textup:
\begin{equation}
\label{eqn:subgpspgl3j}
  \begin{aligned}
    \Gcal_3 &= \left\< \SmallMatrix{1&0&0\\0&\smash[t]{\z_3}&0\\0&0&\smash[t]{\z_3^2}\\} \right\>,
    & \Gcal_4 &= \left\< \SmallMatrix{1&0&0\\0&i&0\\0&0&-1\\} \right\>, 
    & \Gcal_5 &= \left\< \SmallMatrix{1&0&0\\0&\smash[t]{\z_5}&0\\0&0&\smash[t]{\z_5^3}\\} \right\>, \\
    \Gcal_7 &= \left\< \SmallMatrix{1&0&0\\0&\smash[t]{\z_7}&0\\0&0&\smash[t]{\z_7^3}\\} \right\>, 
    & \Gcal_{2,2} &= \left\< \SmallMatrix{1&0&0\\0&-1&0\\0&0&1\\},\;
                    \SmallMatrix{1&0&0\\0&1&0\\0&0&-1\\} \right\>. \hidewidth
  \end{aligned}
\end{equation}
\end{theorem}  

\begin{theorem}
\label{theorem:main2}
Let~$K$ be an algebraically closed field of characteristic~$0$, and
let~$\Gcal\subset\PGL_3(K)$ be one of the groups~\eqref{eqn:subgpspgl3j}
listed in Theorem~$\ref{theorem:main1}$. Suppose that~$f:\PP^2\dashrightarrow\PP^2$
satisfies the following\textup:
\begin{parts}
\Part{\textbullet}
$f\in\Rat_2^2(\Gcal)^\ss(K)$.  
\Part{\textbullet}
$f$ is dominant with $\deg(f)=2$.
\Part{\textbullet}
$\Aut(f)$ is finite.
\end{parts}
Then~$f$ is $N(\Gcal)$-conjugate to one of the maps
listed in Table~$\ref{table:maintable}$.  \textup(See
Table~$\ref{table:expl}$ for an explanation of the entries in
Table~$\ref{table:maintable}$.\textup)
\end{theorem}

\begin{table}[p]
{\small  
\newcommand{\BS}[1]{{\boldsymbol{#1}}}  
\[
\begin{array}{|c|c|c|c|c|c|c|c|c|} \hline
  \hspace{.5em}f\hspace{.5em}
  & \text{Coeffs} & \Aut & \dim  & \text{St?} & \#I & \Crit & \l_1 & \l_2  \\ \hline\hline
  \multicolumn{9}{|l|}{\boldsymbol{\Gcal_3}:\;f_{b,d,g} = [X^2+bYZ,Z^2+dXY,Y^2+gXZ],
       \quad f_{b,d,g}\sim f_{b,g,d} } \\ \cline{2-9}
\BS{1.1}  & (2,2\z_3,2\z_3^2) & C_7\rtimes C_3 & 0  & S & 0 & L_1{\cdot} L_2{\cdot} L_3 & 2 & 4\\
\BS{1.2}  & (-1,-1,-1) & S_4      & 0  & S & 3 & L_1{\cdot} L_2{\cdot} L_3 & 1 & 1\\
\BS{1.3}  & (0,0,0) & S_3      & 0  & S & 0 & L_1{\cdot} L_2{\cdot} L_3 & 2 & 4\\  
\BS{1.4}  & (b,b^{-1},-1) & S_3      & 1  & S & 3 & L_1{\cdot} L_2{\cdot} L_3 & *^1 & *^1\\
\BS{1.5}  & (b,d,d)      &  S_3     & 2  & S & 0 & \G & 2 & 4\\
\BS{1.6}  & bdg = -1 & C_3 & 2  & S & 3 & L_1{\cdot} L_2{\cdot} L_3 & *^2 & *^2\\
\BS{1.7}  & bdg =  8 & C_3 & 2  & S & 0 & L_1{\cdot} L_2{\cdot} L_3 & 2 & 4\\
\BS{1.8}  & \text{other} & C_3 & 3  & S & 0 & \G & 2 & 4\\  \hline
  \multicolumn{9}{|l|}{\boldsymbol{\Gcal_3}:\;f_{a,c,g} = [aX^2+YZ,cZ^2+XY,Y^2+gXZ],
    \quad f_{0,c,g}\sim f_{0,c/g,1/g} } \\
  \multicolumn{9}{|l|} {\hspace{4em} (a,c,g)\ne(0,0,0) }\\ \cline{2-9}
\BS{2.1}  & (0,0,g) & C_3 & 1  & SS & 2 & C{\cdot} L & *^3 & 2\\
\BS{2.2}  & (0,c,1) & S_3 & 1  & S & 1 & \G' & 2 & 3\\
\BS{2.3}  & (0,c,0) & C_3 & 1  & S  & 1 & 3L & 2 & 1 \\
\BS{2.4}  & (a,0,0) & C_3 & 1  & SS & 1 & 3L & 2 & 1 \\
\BS{2.5}  & (a,0,g) & C_3 & 2  & SS & 1 & \G' & 2 & 3 \\
\BS{2.6}  & (0,c,g) & C_3 & 2  & S  & 1 & \G' & 2 & 3 \\ \hline
  \multicolumn{9}{|l|}{\boldsymbol{\Gcal_4}:\;f_{a,e} = [aX^2+Z^2,XY,Y^2+eXZ] } \\ \cline{2-9} 
\BS{3.1}  & (0,0) & C_4 & 0  & S & 1 & 2L_1{\cdot} L_2 & \sqrt2 & 2 \\
\BS{3.2}  & (0,e) & C_4 & 1  & S & 1 & C{\cdot} L & 2 & 3 \\
\BS{3.3}  & (a,0) & C_4 & 1  & S & 2 & 2L_1{\cdot} L_2 & 2 & 2\\
\BS{3.4}  & (a,e) & C_4 & 2  & S & 0 & \G & 2 & 4 \\ \hline
  \multicolumn{9}{|l|}{\boldsymbol{\Gcal_4}:\;f_{c} = [YZ,X^2+cZ^2,XY],
    \quad f_c\sim f_{1/c} } \\ \cline{2-9}
\BS{4.1}  & (-1) & \GG_m\rtimes C_2 & 0  & S & 3 & L_1{\cdot} L_2{\cdot} L_3 & 1 & 1\\
\BS{4.2}  & (1) & S_4 & 0  & S & 3 & L_1{\cdot} L_2{\cdot} L_3 & 1 & 1\\
\BS{4.3}  & (0) & C_4 & 0  & S & 2 &  2L_1{\cdot} L_2 & 1 & 1\\
\BS{4.4}  & (c) & C_4 & 1  & S & 3 & L_1{\cdot} L_2{\cdot} L_3 & 1 & 1\\   \hline
  \multicolumn{9}{|l|}{\boldsymbol{\Gcal_{2,2}}:\;f_{a,e} = [aX^2+Y^2-Z^2,XY,eXZ],
    \quad f_{0,e} \sim f_{0,1/e} } \\ \cline{2-9}
\BS{5.1}  & (0,1)  & \GG_m\rtimes C_2 & 0  & S & 3 & L_1{\cdot} L_2{\cdot} L_3 & 1 & 1\\
\BS{5.2}  & (0,-1) & S_4 & 0  & S & 3 & L_1{\cdot} L_2{\cdot} L_3 & 1 & 1\\
\BS{5.3}  & (0,e)  & C_2^2 & 1  & S & 3 & L_1{\cdot} L_2{\cdot} L_3 & 1 & 1\\
\BS{5.4}  & (a,1)  & \GG_m\rtimes C_2 & 1  & S & 2 & C{\cdot} L & 2 & 2\\
\BS{5.5}  & (a,e) & C_2^2 & 2  & S & 2 & C{\cdot} L & 2 & 2\\ \hline
  \multicolumn{9}{|l|}{\boldsymbol{\Gcal_{2,2}}:\;f = [YZ,XZ,XY]} \\ \cline{2-9}
\BS{6.1}  &  &  S_4 & 0  & S & 3 & L_1{\cdot} L_2{\cdot} L_3 & 1 & 1 \\
\hline  
  \multicolumn{9}{|l|}{\boldsymbol{\Gcal_5}:\;f = [YZ,X^2,Y^2]} \\ \cline{2-9}
\BS{7.1}  &  &  C_5 & 0  & SS & 1 & 2L_1{\cdot} L_2 & \sqrt2 & 2 \\
\hline  
  \multicolumn{9}{|l|}{\boldsymbol{\Gcal_7}:\;f = [Z^2,X^2,Y^2]} \\ \cline{2-9}
\BS{8.1}  &  &  C_7\rtimes C_3 & 0  & S & 0 & L_1{\cdot} L_2{\cdot} L_3 & 2 & 4 \\
\hline  
\end{array}
\]
}
\caption{Dominant semistable degree 2 maps $\PP^2\dashrightarrow\PP^2$ with large automorphism group}
\label{table:maintable}
\end{table}

\begin{table}
{\tiny
\setlength{\parindent}{1em}    
\begin{parts}
\Part{\textbullet}
For each family of Type~$N.M$, Table~\ref{table:maintable} first gives
a formula for the maps in the family~$N.*$ and indicates by the
notation~$f\sim f'$ the $N(\Gcal)$-conjugacy equivalences between
maps. It then lists subfamilies~$M=1,2,\ldots$.  The columns in
Table~\ref{table:maintable} contain the following information\textup:
\begin{center}
  \begin{tabular}{|c|l|} \hline
    \multicolumn{2}{|l|}{\textbf{Key for Columns in Table \ref{table:maintable}}} \\  \hline \hline
    Coeffs & restrictions on the coefficients of $f$ \\ \hline
    Aut & the full automorphism group of $f$ \\ \hline
    dim & dimension of the familiy in $\Moduli_2^2$ \\ \hline
    St? & stability, with $S=\text{stable}$ and $SS=\text{semistable}$ \\ \hline
    $\#I$ & number of points in the indeterminacy locus of $f$ \\ \hline
    Crit & geometry of the critical locus of $f$ (see below for key) \\ \hline
    $\l_1$ & dynamical degree of $f$ (see Remark \ref{remark:l1l2def}) \\ \hline
    $\l_2$ & topological degree of $f$ \\ \hline
\end{tabular}
\end{center}
\Part{\textbullet}
Within each type, the maps in a given line are understood to exclude
the maps in all previous lines. So for example
maps of Type~1.4 exclude the case $b=-1$, which is covered by Type~1.2,
while Type~1.8 excludes maps satisfying $bdg\ne-1$ and $bdg\ne8$
Further, each line includes the indicated $\PGL_3$-equivalences,
so for example Type~1.4 includes both~$(b,b^{-1},-1)$ and~$(b,-1,b^{-1})$.
\Part{\textbullet}
Table~\ref{table:maintable} includes a few cases (Types~4.1,~5.1,~5.4) with
$\Aut(f)\supset\GG_m$. These help fill in the indicated family.
\Part{\textbullet}
The geometry of the critical locus is described by:
\[
\begin{array}{|r@{}l|r@{}l|} \hline
\multicolumn{4}{|l|}{\textbf{Key for $\boldsymbol{\Crit(f)}$}} \\  \hline \hline
  \G &= \text{smooth cubic curve} &
  L_1{\cdot} L_2{\cdot} L_3 &= \text{3 distinct lines} \\ \hline
  \G' &= \text{nodal cubic curve} &
  2L_1{\cdot} L_2 &= \text{double line}\cup\text{line} \\ \hline
  C{\cdot} L &= \text{conic}\cup\text{line} &
  3L &= \text{triple line} \\  \hline
\end{array}
\]
\Part{$*^1$}
We expect that maps of Type~1.4 satisfy $\l_1(f)=2$ and $\l_2(f)=4$.
\Part{$*^2$}
For generic values of~$b,d,g$ satisfying $bdg=-1$, we expect that maps
of Type~1.6 satisfy $\l_1(f)=2$ and $\l_2(f)=4$, but it seems likely
that there is a countable collection of~$(b,d,g)$ triples
satisfying~$\l_1(f)<2$ and~$\l_2(f)<4$.
\Part{$*^3$}
Experiments suggest that $\deg(f^n)$ is the $(n+2)$'nd Fibonacci number, which
would imply that $\l_1(f)=\frac12(1+\sqrt5)$.
\end{parts}  
}
\caption{Notes for Table~$\ref{table:maintable}$}
\label{table:expl}
\end{table}

The next corollary catalogs the complete list of finite groups that appear as automorphism
groups of semi-stable degee~$2$ maps of~$\PP^2$, as well as other information related to
the maps in Table~\ref{table:maintable}.

\begin{corollary}
\label{corollary:listofaut}
Let $K$ be an algebraically closed field of characteristic~$0$, and
let $f\in\Rat_2^2(K)$ be a semi-stable dominant rational map of
degree~$2$ with finite automorphism group.
\begin{parts}
\Part{(a)}
$\Aut(f)$ is isomorphic to one of the following nine groups\textup:
\begin{gather*}
  C_1,\quad C_2,\quad C_3,\quad C_4,\quad C_5,\quad
  C_2^2,\quad S_3,\quad S_4,\quad C_7\rtimes C_3.
\end{gather*}
\Part{(b)}
For  each group~$G$ in~\textup{(a)},
there exists a group $\Gcal\subset\PGL_3$ with $\Gcal\cong G$
and a map $f\in\Rat_2^2(\Gcal)^\ss$ such
that~$f$ is a dominant map of degree~$2$ satisfying $\Aut(f)=\Gcal$.
\Part{(c)}
Let $\Gcal\subset\PGL_3$ be a finite group that is \emph{not} isomorphic to one
of the following groups\textup:
\[
  C_1,\quad C_2,\quad C_3,\quad C_4,\quad C_2^2,\quad\text{or}\quad S_3.
\]
Then $\Moduli_2^2(\Gcal)^\ss$ contains only finitely many dominant degree~$2$ maps.  
\Part{(d)}
Let $f:\PP^2\dashrightarrow\PP^2$ be a  dominant degree~$2$ rational map
such that $\Aut(f)$ contains a copy of $C_2^2$ or $C_5$. Then~$f$ is not a morphism.
\end{parts}  
\end{corollary}

We briefly explain the strategy that we employ to classify maps with
large automorphism group:
\begin{parts}
\Part{\textbullet}
Assume that $\Aut(f)$ contains a subgroup isomorphic to some
finite group~$G$ satisfying~$\#G\ge3$.
\Part{\textbullet}
Classify the conjugacy classes~$\Gcal_1,\ldots,\Gcal_k$ of finite
subgroups of $\PGL_3(K)$ that are isomorphic to~$G$ and choose a
(nice) representative group $G_i\subset\PGL_3(K)$ in~$\Gcal_i$ for
each $1\le i\le k$.\footnote{The complete classification of finite subgroups
  of~$\PGL_3(K)$ is classical, but for completeness we include the
  short proof of the part that we need.}
\Part{\textbullet}
For each~$G_i$, decompose the set of~$f\in\Rat_2^2$ satisfying
$G_i\subseteq\Aut(f)$ into a disjoint union of irreducible
families~$\Fcal_{i,1},\Fcal_{i,2},\ldots\subset\Rat_2^2$.
For example,  in the (typical) case that~$G_i$ is a group of diagonal
matrices, the various~$\Fcal_{i,j}$ are characterized by the
eigenvalues of the generators of~$G_i$ acting on the monomials in the
coordinates of~$f$.
\Part{\textbullet}
By inspection, determine which $f\in\Fcal_{i,j}$ are dominant.
\Part{\textbullet}
Use the numerical criterion of Mumford--Hilbert to determine the set
of semi-stable $f\in\Fcal_{i,j}$.\footnote{We remark that in many
  cases it turns out that~$\Fcal_{i,j}$ contains no dominant
  semi-stable maps. Indeed there are conjugacy classes
  with~$\Gcal\cong C_3$ and~$\Gcal\cong C_4$ that contain no dominant
  semi-stable maps.}
\Part{\textbullet}
It remains to determine the full automorphism group for dominant
semi-stable maps in $f\in\Fcal_{i,j}$, or more generally, to determine
\[
  \Hom(f,f') := \bigl\{\f\in\PGL_3(K) : f'=f^\f\bigr\}
  \quad\text{for $f,f'\in\Fcal_{i,j}$.}
\]
(Taking~$f'=f$ gives~$\Aut(f)$.) A key tool in this endeavor is to
exploit the fact that every~$\f\in\Hom(f,f')$ induces an isomorphism
of the associated indeterminacy and critical loci,
\[
  I(f) \xrightarrow[^\sim]{\;\;\f\;\;} I(f') \quad\text{and}\quad
  \Crit(f)\xrightarrow[^\sim]{\;\;\f\;\;}\Crit(f').
\]
These isomorphisms impose restrictions on~$\f$ which can be used as
the starting point of a case-by-case determination of~$\Hom(f,f')$
and~$\Aut(f)$.
\end{parts}

\begin{remark}
We offer some further brief comments on the final step.
If~$I(f)=I(f')$ is a finite set of points, then~$\f\in\Hom(f,f')$
induces a permutation of these points, and similarly if
$\Crit(f)=\Crit(f')$ is a union of (three) lines, or the union of a
conic and a line, etc., then~$\f$ induces a permutation of these
geometric configurations.  However, there are three cases,
Type~1.5,~1.8, and~3.4 in Table~\ref{table:maintable}, for
which~$I(f)=\emptyset$ and~$\Crit(f)$ is a smooth cubic.  For Types~1.5
and~1.8 we exploit the fact that every~$\f\in\Aut(f)$ permutes the~$9$
flex points of the smooth cubic curve~$\Crit(f)$. This leads to
several hundred cases, which we check by computer. For Type~3.4 we
take a slightly different approach by first showing that~$\Crit(f)$ is
an elliptic curve with CM by~$\ZZ[i]$, and that if~$\f\in\Aut(f)$,
then $\f:\Crit(f)\to\Crit(f)$ is translation by a $3$-torsion
point~$P_0$. We next prove that if~$P_0\ne0$, then~$\Aut(f)$ would
contain a copy of~$C_3^2$, contradicting an earlier calculation. This
allows us to conclude that~$\Aut(f)\cong \ZZ[i]^*\cong C_4$.
\end{remark}


\begin{remark}
We take a moment to record some additional interesting properties of some
of the maps in Table~$\ref{table:maintable}$.
\begin{parts}
\Part{(a)}
The maps $f_{c} = [YZ,X^2+cZ^2,XY]$ of Types~4.1--4.4 satisfy
\begin{align*}
  \deg(f_c^n)=n+1 &\quad\text{if $c\ne\pm1$,} \\
  f_c^{2k}=[X,c^kY,Z]&\quad\text{if $c=\pm1$.}
\end{align*}
In all cases, the second iterate satisfies
$\Aut(f_c^2)\supseteq\GG_m$.  This gives a family of examples of maps
with~$\Aut(f)$ finite and~$\Aut(f^2)$ infinite.  See
Proposition~\ref{proposition:c4f2autgm}.
\Part{(b)}
The maps $f_{0,e}=[Y^2-Z^2,XY,eXZ]$ of Types~5.1--5.3 satisfy
\begin{align*}
  \deg(f_{0,e}^n)=n+1 &\quad\text{if~$e^2$ is not an odd-order root of unity,} \\
  f_{0,e}^{4k+2}=[X,Y,Z]&\quad\text{if $e^{4k+2}=1$.}
\end{align*}  
See Proposition~\ref{proposition:fY2Z2XYeXZ}.
\Part{(c)}
The map~$f=[YZ,X^2,Y^2]$ of Type~$7.1$ with $C_5\cong\Aut(f)$
has the property that $f^8=[X^{16},Y^{16},Z^{16}]$.
\Part{(d)}
The map~$f=[Z^2,X^2,Y^2]$ of Type~$8.1$ with $C_7\subset\Aut(f)$
has the property that $f^3=[X^{8},Y^{8},Z^{8}]$.
\Part{(e)}
The maps of Type~1.2,~4.2,~5.2, and~6.1 are $\PGL_3(K)$-conjugate to one another and
have the property that $f^2=[X,Y,Z]$.  See
Example~\ref{example:notinj}.
\end{parts}   
\end{remark}

\begin{example}
\label{example:notinj}
Consider the following maps from Table~\ref{table:maintable}:
\begin{align*}
  f_{1.2} &:= [X^2-YZ,Z^2-XY,Y^2-XZ], &&\text{Type 1.1} \\
  f_{4.2} &:= [YZ,X^2+Z^2,XY],  &&\text{Type 4.2} \\
  f_{5.2} &:= [Y^2-Z^2,XY,-XZ], &&\text{Type 5.2} \\
  f_{6.1} &:= [YZ,XZ,XY], &&\text{Type 6.1} .
\end{align*}  
One easily checks that
\[
  f_{1.2}\in\Rat_2^2(\Gcal_3),\quad
  f_{4.2}\in\Rat_2^2(\Gcal_4),\quad
  f_{5.2},f_{6.1}\in\Rat_2^2(\Gcal_{2,2}).
\]
Further, we find that all four maps are~$\PGL_3$-conjugate. Explicitly
\[
  f_{1.2}^\a = f_{4.2}^\b = f_{5.2}^\g = f_{6.1}
\]
for
\[
  \a = \SmallMatrix{1&1&1\\\z_3^2&\z_2&1\\\z_3&\z_3^2&1\\},\quad
  \b = \SmallMatrix{0&\z_8^2&1\\ -2\z_8&0&0\\ 0&1&\z_8^2\\},\quad
  \g = \SmallMatrix{2&0&0\\ 0&1&-1\\ 0&1&1\\},
\]
where~$\z_n$ denotes a primitive $n$'th root of unity.

Theorem~\ref{theorem:main2} says that~$f_{5.2}$ and~$f_{6.1}$ both
satisfy $\Aut(f)=\Scal_3\Gcal_{2,2}$, where~$\Scal_3\subset\PGL_3$ is
the group of permutation matrices. In particular,~$f_{5.2}$
and~$f_{6.1}$ are $N(\Scal_3\Gcal_{2,2})$-conjugate, since~$\g$
normalizes~$\Scal_3\Gcal_{2,2}$, but they are
not~$N(\Gcal_{2,2})$-conjugate, since~$\g$ does not
normalize~$\Gcal_{2,2}$. Thus~$f_{5.2}$ and~$f_{6.1}$ represent
different points in $\Rat_2^2(\Gcal_{2,2})^\ss$, but they define the
same point in $\Rat_2^2(\Scal_3\Gcal_{2,2})^\ss$.
\end{example}

\begin{remark}
\label{remark:cremonagp}
A referee has pointed out that for rational maps~$f$ that are not
morphisms, it would be interesting, and possibly more natural, to
compute the group of birational automorphisms of~$\PP^2$ that commute
with~$f$. Writing $\operatorname{BiRat}(\PP^2)$ for the \emph{Cremona
  group}, one might try to classify dominant, semi-stable degree~$2$
maps~$f$ for which the group
\[
  \operatorname{BiAut}(f) :=
  \bigl\{\f\in\operatorname{BiRat}(\PP^2) : f^\f = f \bigr\}
\]
is finite and, say, has order at least~$3$.  A starting point would be
the known classification of the finite subgroups of the Cremona
group~\cite{MR2641179,MR2648675}.  For example, a number of our maps
with large~$\Aut(f)$ are themselves elements of order~$2$ in the
Cremona group, a typical example being the map $f=[YZ,XZ,XY]$
labeled~(6.1) in Table~\ref{table:maintable}.  In these
cases~$\operatorname{BiRat}(f)$ is at least as large as $\Aut(f)\times C_2$,
with the extra~$C_2$ being generated by~$f$.  The analysis of maps
with large finite $\operatorname{BiAut}(f)$ seems like an interesting
problem that deserves further study, but in view of the length of the
present paper, we will not address it at this time.
\end{remark}

\begin{remark}
A referee has pointed out that the present paper is close in spirit to
the classification by Forn\ae ss and Wu~\cite{MR1628170} of degree~$2$
polynomial automorphisms of~$\CC^3$, and work of Cerveau and
D\'eserti~\cite{MR3155973} describing birational maps, especially of
degrees~$2$ and~$3$, of~$\PP^2$, although we note that the latter paper
studies the two-sided action of $\PGL_3\times\PGL_3$, which leads to a
different, albeit also very interesting, classification problem.
\end{remark}

\section{Background}
\label{section:background}
We briefly summarize some of the existing literature on the study
of~$(\Moduli_d^N)^\stab$ and~$(\Moduli_d^N)^\ss$.  A fair amount is
known in the case that $N=1$. For example, it is known
that~$(\Moduli_d^1)^\stab$ and~$(\Moduli_d^1)^\ss$ are rational
varieties~\cite{MR2741188}. And for~$N=1$ and~$d=2$ there are natural
isomorphisms~$(\Moduli_2^1)^\stab=(\Moduli_2^1)^\ss\cong\PP^2$, with
the set of maps $\bigl\{f\in (\Moduli_d^1)^\stab:\deg f=2\bigr\}$
corresponding to~$\AA^2$. See~\cite{milnor:quadraticmaps} for the
proof over~$\CC$ and~\cite{silverman:modulirationalmaps} for the proof
over~$\Spec\ZZ$. For degree~$2$ morphisms~$f:\PP^1\to\PP^1$, the
group~$\Aut(f)\subset\PGL_2$ is isomorphic to either~$C_1$,~$C_2$,
or~$S_3$.  The locus of~$f\in\AA^2$ with~$C_2\subset\Aut(f)$ is a
cuspidal cubic curve, with the cusp corresponding to the only~$f$
having~$\Aut(f)\cong S_3$; see~\cite[Proposition~4.15]{MR2884382}.
For similar results on~$\Moduli_3^1$, see~\cite{arxiv1408.3247}.  More
generally, for~$N=1$ and~$d\ge3$, the singular locus
of~$(\Moduli_d^1)^\stab$ is exactly the set of~$f$ with~$\Aut(f)\ne1$;
see~\cite{arxiv1408.5655} for this result and for 
a calculation of the dimension
and Picard and class groups of
\[
  \Moduli_d^1(\Gcal)^\stab
  := \bigl\{ f \in (\Moduli_d^1)^\stab : \Gcal \subseteq \Aut(f) \bigr\}
  \quad\text{for $\Gcal\subset\PGL_2$.}
\]
(Note that here~$f$ represents a conjugacy class of maps, so~$\Aut(f)$ is a
conjugacy class of subgroups of~$\PGL_2$.)

For all~$N\ge1$ it is known that
$\bigl\{f\in\Moduli_d^N:\Aut(f)=1\bigr\}$ is a non-empty Zariski open
subset of $\Moduli_d^N$; see~\cite{MR2741188}.  Thus ``most'' maps~$f$
have no automorphisms. On the other hand, those~$f$ with~$\Aut(f)\ne1$
are of particular arithmetic interest, since they tend to have
non-trivial $\Kbar/K$-twists, i.e.,  families of maps
that are~$\PGL_{N+1}(\Kbar)$-conjugate, but
not~$\PGL_{N+1}(K)$-conjugate.  There has been a
considerable amount of work studying dynamical twist families and
related problems having to do with fields of definition and fields of
moduli; see for example~\cite{MR3223364,MR2407816,
  silverman:fieldofdef,MR3273498},
\cite[Chapter~7]{silverman:modulirationalmaps},
\cite[Sections~4.7--4.10]{MR2316407}.

For $N\ge2$, there has been some progress.  It is known that if
$f:\PP^N\to\PP^N$ is a morphism of degree at least~$2$, then~$\Aut(f)$
is finite; see \cite{MR2567424}. For~$N=1$ and~$2$, there are explicit
bounds. For example, a morphism $f:\PP^2\to\PP^2$ satisfies
$\#\Aut(f)\le6d^6$; see~\cite[Theorem~6.2]{arxiv1509.06670}. It is
also known that for every finite subgroup $\Gcal\subset\PGL_{N+1}(\Kbar)$,
there are infinitely many morphisms $f:\PP_\Kbar^N\to\PP_\Kbar^N$ of
degree~$\ge2$ such that~$\Aut(f)\supseteq \Gcal$;
see~\cite[Theorem~4.7]{arxiv1509.06670}. However, the situation is
more delicate if one or more of the following natural
conditions is imposed:
\begin{parts}
\Part{\textbullet}
$\Aut(f)$ is exactly equal~$\Gcal$.
\Part{\textbullet}
The degree of~$f$ is specified.
\Part{\textbullet}
The map~$f$ is defined over a non-algebraically closed field~$K$.
\end{parts}
For example, every subgroup of~$\PGL_2$ except the tetrahedral group
can be realized by a map $f:\PP^1\to\PP^1$ defined over~$\QQ$;
see~\cite[Theorem~4.9]{arxiv1509.06670}.  And
\cite[Section~8.1]{arxiv1509.06670} gives examples of morphisms
$f:\PP^2\to\PP^2$ defined over~$\QQ$ with very large automorphism
groups. We also mention that \cite[Section~5]{arxiv1509.06670}
contains a nice summary, in modern notation and with explicit
generators, of the classical classification of finite subgroups of
$\PGL_3(\CC)$.

\begin{remark}
\label{remark:l1l2def}
Two important invariants associated to dominant rational maps
$f:\PP^2\dashrightarrow\PP^2$ are the \emph{dynamical degree}

\[
  \l_1(f) := \lim_{k\to\infty} (\deg f^{\circ k})^{1/k}
\]
and the \emph{topological degree} 
\[
  \l_2(f) := \#f^{-1}(P)\quad\text{for a generic point $P\in\PP^n(K)$.}
\]
The map~$f$ is said to be \emph{algebraically stable}
if~$\l_1(f)=\deg(f)$.  There is a large literature studying dynamical
degrees, algebraic stability, and the existence of invariant measures,
of which we mention two articles. The first~\cite{MR2358970} gives a
precise formula for the dynamical degree of a monomial map.  The
second~\cite{MR2097402} classifies degree two polynomial
maps~$f:\AA^2_\CC\to\AA^2_\CC$ and shows that, up to affine conjugacy,
there are~$13$ families of such maps, all of which extend to an
algebraically stable map on either~$\PP^2$ or~$\PP^1\times\PP^1$.
\end{remark}

\section{Scope of This Paper and Further Questions}
\label{section:questions}
The original goal of this paper was to completely describe the moduli
spaces $\Moduli_2^2(\Gcal)^\ss$ and~$\Moduli_2^2(\Gcal)^\stab$ over an
arbitrary field~$K$, and more generally over~$\Spec R$ for an
appropriately chosen ring~$R$.  This analysis would have included
giving normal forms for~$N(\Gcal)$-conjugacy classes of maps, and it
would have included classifying semi-stable maps that are not dominant
or have degree~$1$.  This turned out to be overambitious, as we
realized when the analysis of the case $\Gcal\cong C_2^2$
approached~$50$ pages and it became clear that the cases~$\Gcal\cong
C_4$ and~$\Gcal\cong C_3$ were going to be even more
complicated. Further, if~$K$ has positive characteristic, then one
must also deal with finite cyclic subgroups of~$\PGL_3(K)$ that are
not diagonalizable, adding another level of complication.

We thus decided to restrict attention to algebraically closed fields
of characteristic~$0$ and to restrict attention to maps that are
dominant and have degree~$2$, since these are the maps whose iterates
potentially have interesting dynamics.  This curtailed goal ended up
being sufficiently challenging, as the length of the present paper
attests. However, we propose the following problems as deserving 
study in future papers and/or a monograph.

\begin{parts}
\Part{(1)}
Describe the geometry of the moduli spaces~$\Moduli_2^2(\Gcal)^\ss$
and $\Moduli_2^2(\Gcal)^\stab$, and the geometry of the natural map
$\Moduli_2^2(\Gcal)^\ss\to \Moduli_2^2(\Gcal')^\ss$ for
subgroups $\Gcal'\subset\Gcal$.
\Part{(2)}
Determine the field of moduli and minimal fields of definition for
points in $\Moduli_2^2(\Gcal)^\ss$, where $\Gcal\subset\PGL_3(K)$ is a
finite subgroup and~$K$ is minimal for the conjugacy class of~$\Gcal$.
\Part{(3)}
Let $f:\PP^2\dashrightarrow\PP^2$ be defined over~$K$.
We recall that
the set of~$\Kbar/K$-twists of~$f$ is the set of maps~$f'$ defined over~$K$
that are~$\PGL_3(\Kbar)$-conjugate to~$f$, modulo~$f'$ and~$f$ being considered equivalent
if they are~$\PGL_3(K)$-conjugate. The set of twists is classified by the kernel
of the inflation map
\[
  H^1\bigl(\Gal(\Kbar/K),\Aut(\f)\bigr) \longrightarrow
  H^1\bigl(\Gal(\Kbar/K),\PGL_3(\Kbar)\bigr);
\]
see~\cite[Section~7.1]{MR2884382}.  Find normal forms for the twists
of the maps in Table~\ref{table:maintable}.
\Part{(4)}
Classify rational maps~$f$ having large finite birational automorphism
groups, as described in Remark~\ref{remark:cremonagp}.
\end{parts}

\begin{example}
We illustrate twisting with an example.  Consider the map
$f=[YZ,X^2,Y^2]$ of Type~7.1. Up to~$\PGL_3(\Qbar)$-conjugacy, this is
the only map in~$(\Moduli_2^2)^\ss$ whose automorphism group is finite
and contains an element of order~$5$.  The isomorphism
\[
  \bfmu_5\longrightarrow\Aut(f),\quad \z\longmapsto \f_\z=[X,\z Y,\z^3 Z],
\]
is~$\Gal(\Qbar/\QQ)$-invariant, and it turns out that every element of
\[
  H^1\bigl(\Gal(\Qbar/\QQ),\bfmu_5)\cong\QQ^*/(\QQ^*)^5
\]
gives a twist of~$f$. Precisely, let~$b\in\QQ^*$,
let~$\b=b^{1/5}\in\Qbar$, and let $\psi=[X,\b Y,\b^3 Z]$. Then the
twist of~$f$ associated to~$b$ is
\begin{gather*}
  f_b := f^{\psi}(X,Y,Z) = \psi^{-1}\circ f\circ\psi(X,Y,Z)
  = \psi^{-1}\circ f(X,\b Y,\b^3 Z) \\
  = \psi^{-1}(\b^4 Y Z, X^2, \b^2 Y^2)
  = [\b^4 Y Z, \b^{-1} X^2, \b^{-1} Y^2]
  = [b Y Z, X^2, Y^2].
\end{gather*}
Note that~$f_b$ is defined over~$\QQ$, but that the map~$\psi$ conjugating~$f$ to~$f_b$
is only defined over~$\QQ(b^{1/5})$.

Similarly, the~$\bfmu_7$-twist of the map~$f=[Z^2,X^2,Y^2]$ associated
to $b\in \QQ^*/(\QQ^*)^7 \cong H^1\bigl(\Gal(\Qbar/\QQ),\bfmu_7)$ is
$[bZ^2,X^2,Y^2]$.  We leave the details of the computation to the
reader.
\end{example}  

\section{Some Finite Subgroups of $\PGL_3$}
\label{section:finitesubgpspgl3}

In this section we prove some elementary results concerning 
finite subgroups of~$\PGL_3$. This information may be gleaned from
classical descriptions of all finite subgroups of~$\PGL_3$, but for
completeness we shall prove what we need.

\begin{lemma}
\label{lemma:GinPGL3}  
Let~$K$ be an algebraically closed field of characteristic~$0$,
and let $G\subset\PGL_3(K)$ be a finite subgroup.
\begin{parts}
\Part{(a)}
Suppose that $G\cong C_q$ with $q$ a prime power.
Then there is a~$\f\in\PGL_3(K)$, a primitive $q$'th
root of unity $\z\in K$, and an integer~$m$ such that
\begin{equation}
  \label{eqn:Gf100z}
  G^\f  = \left\< \SmallMatrix{1&0&0\\0&\z&0\\0&0&\smash[t]{\z^m}\\} \right\>.
\end{equation}
\Part{(b)}
Suppose that $G\cong C_q$ with $q\in\{4,5,7\}$.  Then there is
a~$\f\in\PGL_3(K)$ and a primitive $q$'th root of unity~$\z$ such that
\begin{align*}
  G\cong C_4 &\quad\Longrightarrow\quad
    G^\f = \left\< \SmallMatrix{1&0&0\\0&\z&0\\0&0&1\\} \right\>
    \;\text{or}\;
    \left\< \SmallMatrix{1&0&0\\0&\z&0\\0&0&-1\\} \right\>, \\
  G\cong C_5 &\quad\Longrightarrow\quad
    G^\f = \left\< \SmallMatrix{1&0&0\\0&\z&0\\0&0&1\\} \right\>
    \;\text{or}\;
    \left\< \SmallMatrix{1&0&0\\0&\z&0\\0&0&\z^3\\} \right\>, \\
  G\cong C_7 &\quad\Longrightarrow\quad
    G^\f = \left\< \SmallMatrix{1&0&0\\0&\z&0\\0&0&1\\} \right\>
    \;\text{or}\;
    \left\< \SmallMatrix{1&0&0\\0&\z&0\\0&0&\z^2\\} \right\>
    \;\text{or}\;
    \left\< \SmallMatrix{1&0&0\\0&\z&0\\0&0&\z^3\\} \right\>.
\end{align*}
\Part{(c)}
Suppose that $G\cong C_p\times C_p$ with $p$ prime, and
let~$\z$ be a primitive $p$'th root of unity.
Then there is a~$\f\in\PGL_3(K)$ such that one of the following
is valid\textup:
\begin{align}
  G^\f  &= \left\< \SmallMatrix{1&0&0\\0&\z&0\\0&0&1\\},\;
  \SmallMatrix{1&0&0\\0&1&0\\0&0&\z\\} \right\>,
  &&\text{$p$ arbitrary.}
  \label{eqn:Lemma10c1}  \\
  G^\f  &= \left\< \SmallMatrix{1&0&0\\0&\z&0\\0&0&\z^2\\},\;
  \SmallMatrix{0&1&0\\0&0&1\\1&0&0\\} \right\>, 
  &&\text{$p=3$ only.}
  \label{eqn:Lemma10c3}  
\end{align}
Further, the group~\eqref{eqn:Lemma10c1} with $p=3$ is
not~$\GL_3(K)$-conjugate to the group~\eqref{eqn:Lemma10c3}.
\end{parts}  
\end{lemma}
\begin{proof}
We remark that if $\a\in\PGL_3(K)$ has finite order~$n$, then we can
lift it to an element~$A\in\GL_3(K)$ having the same order.  To see
this, we start with an arbitrary lift~$A$. Then~$A^n=c I$ for some
$c\in K^*$, so we can take $c^{-1/n}A$ as our lift of~$\a$. We also
remark that since we have assumed that $\operatorname{char}(K)=0$,
every element in~$\GL_3(K)$ of finite order is diagonalizable.
\par\noindent(a)\enspace  
Let $G=\<\a\>$ with~$\a\in\PGL_3(K)$ having order~$q$.  We lift~$\a$
to an~$A\in\GL_3(K)$ of order~$q$.  Conjugating~$A$ to put it into
Jordan normal form, the fact that~$A^q=1$ implies that~$A$ is
diagonal and its diagonal entries are $q$'th roots of unity.
Replacing~$A$ by a scalar multiple, which we may do since we are
really only interested in the image of~$A$ in~$\PGL_3(K$), we may
assume that the upper left entry of~$A$ is~$1$. (Note that we still
have~$A^q=I$.)  The fact that~$\a$ has
exact order~$q$, where~$q$ is a prime power, implies that one of the
other diagonal entries is a primitive $q$'th root of unity, which we
denote~$\z$. Possibly after reversing the~$Y$ and~$Z$
coordinates,~$\a$ is diagonal with entries~$1,\z,\eta$, where~$\eta$,
being a $q$'th root of unity, is a power of~$\z$.
\par\noindent(b)\enspace
From~(a), we can find~$\f$ so that~$G^\f$ is given
by~\eqref{eqn:Gf100z} with $0\le m<q$. For notational convenience, we let
\[
  \t(m) :=    \SmallMatrix{\smash[t]{1}&0&0\\0&\smash[t]{\z}&0\\0&0&\smash[t]{\z^m}\\} .
\]
We note that conjugation by a permutation matrix in~$\PGL_3$ has
the effect of permuting the entries of a diagonal matrix. Writing~$\sim$
to denote~$\PGL_3$-conjugation equivalence and using the fact that we are
working in~$\PGL_3$, we have
\begin{align*}
  \t(m)
  &\sim 
  \left\< \SmallMatrix{\smash[t]{\z}&0&0\\0&\smash[t]{1}&0\\0&0&\smash[t]{\z^m}\\} \right\>
  =
  \left\< \SmallMatrix{\smash[t]{1}&0&0\\0&\smash[t]{\z^{-1}}&0\\0&0&\smash[t]{\z^{m-1}}\\} \right\>
  =
  \t(1-m\bmod q), \\
  \t(m)
  &\sim 
  \left\< \SmallMatrix{\smash[t]{1}&0&0\\0&\smash[t]{\z^m}&0\\0&0&\smash[t]{\z}\\} \right\>
  =
  \t(m^{-1}\bmod q) \quad\text{if $\gcd(m,q)=1$.}
\end{align*}
(We remark that the other three permutations do not give results that are
useful for our purposes.)  Hence after a further conjugation by a
permutation matrix, we may take~$G^\f$ to be generated by any one of
the following three matrices,
\[
  \t(m),\quad \t(1-m\bmod q),\quad \t(m^{-1}\bmod q),
\]
subject to $\gcd(m,q)=1$ for the last one.

Suppose first that~$q=4$. Taking $m=0$ and $m=2$, we find that
\[
  \bigl\<\t(0)\bigr\>\sim \bigl\<\t(1)\bigr\>
  \quad\text{and}\quad
  \bigl\<\t(2)\bigr\>\sim \bigl\<\t(3)\bigr\>.
\]
Hence we can find a~$\f$ so that~$G^\f$ is generated by either~$\t(0)$ or~$\t(2)$.

Next let $q=5$. Then
\[
  \bigl\<\t(0)\bigr\>\sim \bigl\<\t(1)\bigr\>
  \quad\text{and}\quad
  \bigl\<\t(2)\bigr\>\sim \bigl\<\t(4)\bigr\>\sim \bigl\<\t(3)\bigr\>.
\]
Hence we can find a~$\f$ so that~$G^\f$ is generated by either~$\t(0)$ or~$\t(3)$.

Finally let $q=7$. Then
\[
  \bigl\<\t(0)\bigr\>\sim \bigl\<\t(1)\bigr\>,\quad
  \bigl\<\t(2)\bigr\>\sim \bigl\<\t(6)\bigr\>\sim \bigl\<\t(4)\bigr\>
  \quad\text{and}\quad
  \bigl\<\t(3)\bigr\>\sim \bigl\<\t(5)\bigr\>.
\]
Hence we can find a~$\f$ so that~$G^\f$ is generated by either~$\t(0)$ or~$\t(2)$
or~$\t(3)$.
\par\noindent(c)\enspace
Let $\a,\b\in G$ be generators of~$G$. We lift~$\a$ and~$\b$,
respectively, to matrices~$A,B\in\GL_3(K)$ satisfying $A^p=B^p=I$.
The fact that $\a\b=\b\a$ in~$\PGL_3(K)$ tells us that there is an
$\e\in K^*$ such that $AB=\e BA$ in~$\GL_3(K)$.

We start with the case that~$\e=1$, so we have diagonalizable
matrices~$A,B\in\GL_3(K)$ satisfying $AB=BA$.  Standard linear algebra
says that they can be simultaneously diagonalized, so after
conjugation, we may assume that~$A$ and~$B$ are both diagonal. And
since $A^p=B^p=I$ and the image of~$\<A,B\>$ in~$\PGL_3(K)$ is of
type~$C_p^2$, we see that the group
\[
  \<\z I,A,B\> = \{\z^i A^j B^k : 0\le i,j,k\le p-1\} \subset \GL_3(K)
\]
contains~$p^3$ distinct diagonal elements of~$\GL_3(K)$ 
of order dividing~$p$. But~$\GL_3(K)$ contains exactly~$p^3$ diagonal
matrices of order dividing~$p$, namely the diagonal matrices with entries
that are arbitrary $p$'th roots of unity. It follows
that~$G=\<\a,\b\>$, which is the image of~$\<\z I,A,B\>$ in~$\PGL_3(K)$,
is the group described in~\eqref{eqn:Lemma10c1}.

We next suppose that~$\e\ne1$.  Applying~(a) to~$\<\a\>$, we can
conjugate so that~$\a$ lifts to a matrix of the form
$A=\SmallMatrix{1&0&0\\ 0&\z&0\\ 0&0&\smash[t]{\z^m}\\}\in\GL_3(K)$,
where~$\z$ is a primitive $p$'th root of unity and $0\le m<p$.
Writing the lift~$B$ of~$\b$ with generic entries, the relation $AB=\e
BA$ becomes
\[
  \SmallMatrix{ a & b & c \\ d & e & f \\ g & h & i \\ }
  = B
  = \e A^{-1}BA
  = \e \SmallMatrix{ a & \z b & \smash[t]{\z^m} c \\
    \smash[t]{\z^{-1}}d & e & \smash[t]{\z^{m-1}}f \\
    \smash[t]{\z^{-m}}g & \smash[t]{\z^{1-m}}h & i \\ }.
\]
The assumption that~$\e\ne1$ forces~$a=e=i=0$. Since~$B$ is invertible,
we see that either~$b\ne0$ or~$c\ne0$. For the former, we find that
\begin{equation}
  \label{eqn:bne0ez1}
  b\ne0
  \quad\Longrightarrow\quad
  \e=\z^{-1}
  \quad\Longrightarrow\quad
  \left\{
  \begin{aligned}
    (1-\z^{m-1})c &= 0,\\
    (1-\z^{-2})d  &= 0,\\
    (1-\z^{m-2})f &= 0,\\
    (1-\z^{-m-1})g &= 0,\\
    (1-\z^{-m})h &= 0.\\
  \end{aligned}
  \right.
\end{equation}
The invertibility of~$B$ also tells us that~$d$ and~$f$ are not
both~$0$, and that~$g$ and~$h$ are not both~$0$.  Therefore
\[
  (p=2~\text{or}~m=2) \quad\text{and}\quad (m=p-1~\text{or}~m=0).
\]
Hence either $p=2$, or else $m=2$ and $p=3$. We consider these cases
in turn.

If $p=2$, then $m\in\{0,1\}$, and~\eqref{eqn:bne0ez1} tells us that
one of the following holds:
\[
  \begin{aligned}
      (p,m)=(2,0) &\Longrightarrow c=g=0
        \Longrightarrow B = \SmallMatrix{0&b&0\\ d&0&f\\ 0&h&0\\}, \\
      (p,m)=(2,1) &\Longrightarrow f=h=0
        \Longrightarrow B = \SmallMatrix{0&b&c\\ d&0&0\\ g&0&0\\}. \\
  \end{aligned}
\]
Thus both~$m$ values with~$p=2$ lead to a matrix~$B$ that is not invertible.

If $m=2$ and $p=3$, then $c=d=h=0$, so~$B$ has the form
\[
  B = \SmallMatrix{ 0 & b & 0 \\ 0 & 0 & f \\ g & 0 & 0 \\ }.
\]  
We know that~$B^3=I$, so~$bfg=1$. Conjugating~$B$ by the
matrix~$\f:=\SmallMatrix{gf&0&0\\0&f&0\\0&0&1\\}$ yields
$B^\f=\SmallMatrix{ 0 & 1 & 0 \\ 0 & 0 & 1 \\ 1 & 0 & 0 \\ }$,
while~$A^\f=A$. This proves that~$G$ is conjugate to the
group~\eqref{eqn:Lemma10c3}.

Next we assume that $b=0$ and $c\ne0$, which leads to
\begin{equation}
  \label{eqn:b0cne0ez1}
  b=0~\text{and}~c\ne0
  \quad\Longrightarrow\quad
  \e=\z^{-m}
  \quad\Longrightarrow\quad
  \left\{
  \begin{aligned}
    (1-\z^{-1-m})d &= 0,\\
    (1-\z^{-1})f  &= 0,\\
    (1-\z^{-2m})g &= 0,\\
    (1-\z^{1-2m})h &= 0.\\
  \end{aligned}
  \right.
\end{equation}
Thus $f=0$, and then the fact that $\det B=cdh\ne0$ tells us
that~$d\ne0$ and~$h\ne0$. Then~\eqref{eqn:b0cne0ez1} gives
\[
  d\ne0 \quad\Longrightarrow\quad m=p-1
  \quad\text{and}\quad
  h\ne0 \quad\Longrightarrow\quad m=\frac{p+1}{2}.
\]
Equating the values of~$m$, we conclude that $p=3$ and $m=2$,
and~\eqref{eqn:b0cne0ez1} forces~$g=0$. This shows that
$B=\SmallMatrix{0&0&c\\d&0&0\\0&h&0\\}$, and using $B^3=1$ shows that
$cdh=1$. So if we conjugate by
$\phi:=\SmallMatrix{c&0&0\\0&cd&0\\0&0&1\\}$, we find that $A^\phi=A$
and $B^\phi=\SmallMatrix{0&0&1\\1&0&0\\0&1&0\\}$.  Hence the subgroup
of~$\PGL_3$ generated by~$A^\phi$ and~$B^\phi$ is the
group~\eqref{eqn:Lemma10c3}.

Finally, to prove that the groups in~\eqref{eqn:Lemma10c1} with $p=3$
and~\eqref{eqn:Lemma10c3} are not $\GL_3(K)$-conjugate, we observe
that~\eqref{eqn:Lemma10c1} fixes three points in~$\PP^2$,
whiile~\eqref{eqn:Lemma10c3} has no fixed points.
\end{proof}  

\begin{lemma}
\label{lemma:subgpnlzr}
Let~$K$ be an algebraically closed field of characteristic~$0$, and
let $\Gcal\subset\PGL_3(K)$ be one of the
groups~\eqref{eqn:subgpspgl3j} listed in Theorem~$\ref{theorem:main1}$.
Then the identity component of the normalizer~$\Gcal$ is the group of
diagonal matrices,
\[
  N(\Gcal)^\circ = \Dcal := \left\{ \SmallMatrix{\a&0&0\\0&\b&0\\0&0&\g\\} \in \PGL_3(K)\right\}.
\]
More precisely, we have
\begin{align*}
  N(\Gcal_3)&=N(\Gcal_{2,2})=\Scal_3\Dcal, &
  N(\Gcal_4) &= \left\< \SmallMatrix{0&0&1\\0&1&0\\1&0&0\\} \right\> \Dcal,\\
  N(\Gcal_5) &= \left\< \SmallMatrix{0&1&0\\1&0&0\\0&0&1\\} \right\> \Dcal, &
  N(\Gcal_7) &= \left\< \SmallMatrix{0&1&0\\0&0&1\\1&0&0\\} \right\> \Dcal,
\end{align*}  
where~$\Scal_3\subset\PGL_3(K)$ is the group of permutation matrices.
\end{lemma}
\begin{proof}
We recall that for any group~$G$ and subgroup~$H\subseteq G$, the
kernel of the standard homomorphism
\[
  N_G(H) \longrightarrow \Aut(H),\quad g\longmapsto (h\mapsto g^{-1}hg)
\]
is the centralizer~$C_G(H)$.  We use this to simplify our calculations.

Let $\z$ be a primitive $n$'th root of unity with $n\ge3$, let $2\le m<n$, and let
$T=T_{n,m}:=\SmallMatrix{1&0&0\\0&\z&0\\0&0&\smash{\z^m}\\}$.
Then $A=\SmallMatrix{a&b&c\\ d&e&f\\ g&h&i\\}\in C(T)$ if and only
if $A=T^{-1}AT$, so if and only if
\[
  \SmallMatrix{a&b&c\\ d&e&f\\ g&h&i\\}
  = \SmallMatrix{1&0&0\\0&\z&0\\0&0&\smash{\z^m}\\}^{-1}
  \SmallMatrix{a&b&c\\ d&e&f\\ g&h&i\\}
  \SmallMatrix{1&0&0\\0&\z&0\\0&0&\smash{\z^m}\\}
  = \SmallMatrix{
    a        & \z b      & \z^m c \\
    \z^{-1} d & e         & \z^{m-1} f \\
    \z^{-m} g & \z^{1-m} h & i\\}.
\]  
Keeping in mind that we are working in~$\PGL_3$, we first note that if
any of~$a,e,i$ is non-zero, then~$A$ is diagonal.  Suppose that
$a=e=i=0$.  If $b\ne0$, then the fact that~$\z^m,\z^{-1},\z^{1-m}$
are distinct from~$\z$
gives~$c=d=h=0$, and then the nonsingularity of~$A$ tells
us that~$fg\ne0$, and~$\z=\z^{m-1}=\z^{-m}$.  Hence~$b\ne0$ is allowed
only if~$n=3$ and~$m=2$, in which case~$C(T)$ contains the scaled
cyclic permutation $\SmallMatrix{0&b&0\\ 0&0&f\\ g&0&0\\}$. A similar
analysis for~$c\ne0$ yields the inverse scaled permutation for~$n=3$
and a contradiction for~$n\ge4$.  This completes the proof that
\[
  C(T_{n,m}) = \begin{cases} 
    \<\pi\>\Dcal&\text{if $n=3$ and $m=2$,}\\
    \Dcal&\text{if $n\ge4$ and $2\le m\le n-1$,}\\
  \end{cases}
\]
where~$\pi\in\PGL_3$ is a cyclic permutation.

Suppose now that $(n,m)=(3,2)$. Then the transposition
$\a=\SmallMatrix{1&0&0\\0&0&1\\0&1&0\\}$ satisfies
$\a^{-1}T_{3,2}\a=T_{3,2}^2$, so~$\a\in N(\Gcal_3)\setminus
C(\Gcal_3)$. Using the inclusion
\[
  N(\Gcal_3)/C(\Gcal_3)\hookrightarrow\Aut(\Gcal_3)\cong(\ZZ/3\ZZ)^*\cong\ZZ/2\ZZ,
\]
we conclude that
$N(\Gcal_3)=\<\a\>C(\Gcal_3)=\<\a\>\<\pi\>\Dcal=\Scal_3\Dcal$.

Next let $(n,m)=(4,2)$. Then the transposition
$\b=\SmallMatrix{0&0&1\\0&1&0\\1&0&0\\}$ satisfies
$\b^{-1}T_{4,2}\b=T_{4,2}^3$, so~$\b\in N(\Gcal_4)\setminus
C(\Gcal_4)$.
Using the inclusion
\[
  N(\Gcal_4)/C(\Gcal_4)\hookrightarrow\Aut(\Gcal_4) \cong (\ZZ/4\ZZ)^* \cong \ZZ/2\ZZ,
\]
we conclude that $N(\Gcal_4)=\<\b\> C(\Gcal_4)=\<\b\>\Dcal$.

Next we consider~$T_{m,n}$ with $m=3$ and $n\ge5$ prime. An element
$A\in N(T)\setminus C(T)$ needs to satisfy
\[
  \SmallMatrix{a&b&c\\ d&e&f\\ g&h&i\\}
  = \SmallMatrix{1&0&0\\0&\z&0\\0&0&\smash{\z^3}\\}^{-1}
  \SmallMatrix{a&b&c\\ d&e&f\\ g&h&i\\}
  \SmallMatrix{1&0&0\\0&\z&0\\0&0&\smash{\z^3}\\}^j
  = \SmallMatrix{
    a        & \z^j b      & \z^{3j} c \\
    \z^{-1} d & \z^{j-1} e   & \z^{3j-1} f \\
    \z^{-3} g & \z^{j-3} h & \z^{3j-3} i\\}
\]  
for some $2\le j<n$ with $\gcd(j,n)=1$. We have
\begin{align*}
  a\ne 0 &\Longrightarrow b=d=e=g=0 \Longrightarrow fh\ne0 \\
  &\Longrightarrow j\equiv 3\pmodintext{n}~\text{and}~3j\equiv1\pmodintext{n} 
   \Longrightarrow n\mid 8. \quad {\rightarrow}{\leftarrow} \\  
  e\ne 0 &\Longrightarrow a=b=d=h=0 \Longrightarrow cg\ne0 \\
  &\Longrightarrow 3j\equiv j-1\pmodintext{n}~\text{and}~ -3\equiv j-1\pmodintext{n} \\
  &\Longrightarrow n\mid 3. \quad {\rightarrow}{\leftarrow} \\  
  i\ne 0 &\Longrightarrow c=e=f=g=h \Longrightarrow bd\ne0 \\
  &\Longrightarrow j\equiv 3j-3\pmodintext{n}~\text{and}~ -1\equiv 3j-3\pmodintext{n} \\
  &\Longrightarrow n\mid 5.
\end{align*}
So we find that the normalizer of~$\Gcal_5$ contains
$\SmallMatrix{0&1&0\\ 1&0&0\\ 0&0&1\\}$.

We next look for maps with $a=e=i=0$, so
\[
  \SmallMatrix{0&b&c\\ d&0&f\\ g&h&0\\}
  = \SmallMatrix{1&0&0\\0&\z&0\\0&0&\smash{\z^3}\\}^{-1}
  \SmallMatrix{0&b&c\\ d&0&f\\ g&h&0\\}
  \SmallMatrix{1&0&0\\0&\z&0\\0&0&\smash{\z^3}\\}^j
  = \SmallMatrix{
    0        & \z^j b      & \z^{3j} c \\
    \z^{-1} d & 0   & \z^{3j-1} f \\
    \z^{-3} g & \z^{j-3} h & 0 \\}
\]  
This gives
\begin{align*}
  b\ne 0 &\Longrightarrow c=h=0 \Longrightarrow g\ne0 \\
  &\Longrightarrow j\equiv -3\pmodintext{n} \Longrightarrow d=0 \Longrightarrow f\ne0 \\
  &\Longrightarrow j\equiv-3\equiv 3j-1\pmodintext{n} \Longrightarrow n\mid 7. \\
  c\ne 0 &\Longrightarrow b=f=0 \Longrightarrow d\ne0 \\
  &\Longrightarrow -1\equiv 3j\pmodintext{n} \Longrightarrow g=0 \Longrightarrow h\ne0 \\
  &\Longrightarrow -1\equiv 3j \equiv j-3 \pmodintext{n} \Longrightarrow n\mid 7.   
\end{align*}
So we find that the normalizer of~$\Gcal_7$ contains
$\SmallMatrix{0&1&0\\ 0&0&1\\ 1&0&0\\}$ and
$\SmallMatrix{0&0&1\\ 1&0&0\\ 0&1&0\\}$.

This completes the computation of~$N(\Gcal_n)$ with~$n=3,4,5,7$.

Finally, consider an element $A\in C(\Gcal_{2,2})$ of the centralizer
of~$\Gcal_{2,2}$. It satisfies the two equations
\begin{align*}
  \SmallMatrix{a&b&c\\ d&e&f\\ g&h&i\\}
  &= \SmallMatrix{1&0&0\\0&-1&0\\0&0&1\\}^{-1}
  \SmallMatrix{a&b&c\\ d&e&f\\ g&h&i\\}
  \SmallMatrix{1&0&0\\0&-1&0\\0&0&1\\}
  =
  \SmallMatrix{a&-b&c\\ -d&e&-f\\ g&-h&i\\} \\
  \SmallMatrix{a&b&c\\ d&e&f\\ g&h&i\\}
  &= \SmallMatrix{1&0&0\\0&1&0\\0&0&-1\\}^{-1}
  \SmallMatrix{a&b&c\\ d&e&f\\ g&h&i\\}
  \SmallMatrix{1&0&0\\0&1&0\\0&0&-1\\}
  =
  \SmallMatrix{a&b&-c\\ d&e&-f\\ -g&-h&i\\}.
\end{align*}
If any of~$a,e,i$ is non-zero, then these two equations combine to
tell us that~$A$ is diagonal. On the other hand, if~$a=e=i=0$, then
$b\ne0$ forces $g=h=0$, contradicting the non-singularlity of~$A$, and
similarly~$c\ne0$ forces $b=h=0$, giving the same
contradiction. Hence~$C(\Gcal_{2,2})=\Dcal$.
Next we observe that every permutation in~$\Scal_3$ is in~$N(\Gcal_{2,2})$,
so
\[
  \Scal_3 \longrightarrow N(\Gcal_{2,2})/C(\Gcal_{2,2})
  \longhookrightarrow \Aut(\Gcal_{2,2}) \cong \GL_2(\FF_2) \cong \Scal_3.
\]
A quick calculation shows that the map $\Scal_3\to\Aut(\Gcal_{2,2})$ is an isomorphism,
and hence $N(\Gcal_{2,2})=\Scal_3 C(\Gcal_{2,2})$.
\end{proof}

\section{Diagonal Stability and Maps of Finite Order}
\label{section:diagstability}
In this section we set notation that is used throughout the rest of
this paper, we remind the reader of the Hilbert--Mumford criterion for
GIT stability, and we create two tables that we will use to determine
the stability of elements of~$\Rat_2^2$.

For a fixed root of unity~$\z$ and integer~$m$, we define
\[
  \t_m = \t_{\z,m} \in\GL_3(K),\quad \t_m(X,Y,Z)=(X,\z Y,\z^m Z).
\]
Further, for each pair of integers~$(k,\ell)$ we define a 1-parameter
subgroup of~$\SL_3$ by
\[
  L_{k,\ell}:\GG_m\to\SL_3,\quad L_{k,\ell}(t) =
  \SmallMatrix{t^k&0&0\\ 0&t^\ell&0\\ 0&0&t^{-k-\ell}\\}.
\]
We now compute the effect of applying~$\t_m$ and~$L_{k,\ell}$ to each of the
quadratic monomials in a degree~$2$ map of~$\AA^3$. 

Table~\ref{table:effectoftm} gives the effect of applying the map
$\t_m=(X,\z Y,\z^m Z)$ to each quadratic monomial, where an integer
entry~$\e$ in Table~\ref{table:effectoftm} means that the monomial is
multiplied by~$\z^\e$.  Similarly, Table~\ref{table:effectofL} gives
the effect of applying~$L_{k,\ell}(t)$ to each quadratic monomial,
where an integer entry~$\d$ in Table~\ref{table:effectofL} means that
the monomial is multiplied by~$t^\d$. 

\begin{table}[ht]
{\small
\[
\begin{array}{|c||c|c|c|c|c|c|} \hline
  & X^2 & Y^2 & Z^2 & XY & XZ & YZ \\ \hline\hline
  {\text{$X$-coord}}
  & 0 & 2 & 2m & 1 & m & m+1 \\ \hline
  {\text{$Y$-coord}}
  & -1 & 1 & 2m-1 & 0 & m-1 & m \\ \hline
  {\text{$Z$-coord}}
  & -m & 2-m & m & 1-m & 0 & 1 \\ \hline
\end{array}
\]
}
\caption{Effect of $\t_m=(X,\z Y,\z^m Z)$ on monomials}
\label{table:effectoftm}
\end{table}

\begin{table}[ht]
{\small
\[
\begin{array}{|c||c|c|c|c|c|c|} \hline
  & X^2 & Y^2 & Z^2 & XY & XZ & YZ \\ \hline\hline
  {\text{$X$-coord}}
  & k & -k+2\ell & -3k-2\ell & \ell & -k-\ell & -2k \\ \hline
  {\text{$Y$-coord}}
  & 2k-\ell & \ell & -2k-3\ell & k & -2\ell & -k-\ell \\ \hline
  {\text{$Z$-coord}}
  & 3k+\ell & k+3\ell & -k-\ell & 2k+2\ell & k & \ell \\ \hline
\end{array}
\]
}
\caption{Effect of $L_{k,\ell}(t)=(t^k X,t^\ell Y,t^{-k-\ell} Z)$ on monomials}
\label{table:effectofL}
\end{table}

We are going to use the numerical criterion of
Hilbert--Mumford~\cite[Chapter~2,
  Theorem~2.1]{mumford:geometricinvarianttheory} to determine the
stability of maps.  We recall the general setup.
(See~\cite[Section~2.2]{MR2884382} or~\cite{MR2741188} for similar
calculations.)  Let $\Gcal\subseteq\SL_{n+1}$ be an algebraic subgroup
of~$\SL_{n+1}$ and let $\bfa\in\PP^n$.  For any given one-parameter
subgroup $L:\GG_m\to\Gcal$, choose coordinates on~$\PP^{n+1}$ so that
the image of~$L$ is contained in the group of diagonal matrices.
Write~$\bfa=[a_1,\ldots,a_n]\in\PP^n$ in these coordinates,
let~$\hat\bfa=(\hat a_1,\ldots,\hat a_n)$ be a lift of~$\bfa$
to~$\AA^n$, and write the action of~$L$ on the lift~$\hat\bfa$ as
\[
  L(t)\cdot\hat\bfa=(t^{r_1}\hat a_1, t^{r_2}\hat a_2,\ldots,
  t^{r_n}\hat a_n),
\]
where~$r_1,\ldots,r_n\in\ZZ$.  The numerical factor associated
to~$L$ at~$\bfa$ is the quantity
\[
  \mu^{\Ocal(1)}(\bfa,L) 
     = \max\{ -r_i : \text{$i$ satisfies $\hat a_i\ne0$}\}. 
\]
Then the Hilbert--Mumford numerical criterion says that
\begin{align*}
  \text{$\bfa$ is $\Gcal$-unstable}
  &\quad\Longleftrightarrow\quad
  \text{$\mu^{\Ocal(1)}(L,\bfa)<0$ for some $L$,} \\
  \text{$\bfa$ is $\Gcal$-not stable}
  &\quad\Longleftrightarrow\quad
  \text{$\mu^{\Ocal(1)}(L,\bfa)\le0$ for some $L$.} 
\end{align*}
Equivalently,~$\bfa$ is $\Gcal$-stable if $\mu^{\Ocal(1)}(L,\bfa)>0$
for all~$L$, and it is~$\Gcal$-semistable if
$\mu^{\Ocal(1)}(L,\bfa)\ge0$ for all~$L$.

We write~$f_{m,\e}$ to denote a generic element of~$\Rat_2^2$
whose affine lift~${\hat f}_{m,\e}:\AA^3\to\AA^3$ satisfies
\[
  {\hat f}_{m,\e}^{\t_m} = \z^\e {\hat f}_{m,\e}.
\]
Since the one-parameter subgroup~$L_{k,\ell}$ is already diagonalized,
the following two-step procedure computes~$\mu^{\Ocal(1)}(f_{m,\e},L_{k,\ell})$.
\begin{parts}
\Part{\textbullet}
Look at Table~\ref{table:effectoftm} and check off all of the boxes
whose entry is congruent to~$\e\bmod p$.
\Part{\textbullet}
Then $\mu^{\Ocal(1)}(f_{m,\e},L_{k,\ell})$ is equal to the maximum of
the \emph{negatives} of the corresponding entries in
Table~\ref{table:effectofL}.
\end{parts}

We note that every diagonalized one-parameter subgroup of~$\SL_3$
is conjugate to~$L_{k,\ell}$ for some~$(k,\ell)\ne(0,0)$. We set the
notation
\[
  \Dcal = \left\{ \SmallMatrix{\a&0&0\\0&\b&0\\0&0&\g\\} \in \SL_3(K)\right\}
\]
for the group of diagonal matrices. By abuse of notation, we may
sometimes also write~$\Dcal$ for the diagonal subgroup
of~$\PGL_3$. Similarly, we set the notation
\[
   \Scal_3 := \left\<\SmallMatrix{1&0&0\\ 0&1&0\\ 0&0&1\\},
   \SmallMatrix{0&0&1\\ 1&0&0\\ 0&1&0\\},
   \SmallMatrix{0&1&0\\ 0&0&1\\ 1&0&0\\},
   \SmallMatrix{0&1&0\\ 1&0&0\\ 0&0&1\\},
   \SmallMatrix{0&0&1\\ 0&1&0\\ 1&0&0\\},
   \SmallMatrix{1&0&0\\ 0&0&1\\ 0&1&0\\}\right\>
\]
for the group of permutation matrices in~$\PGL_3$ or~$\SL_3$.

\paragraph{\textbf{Numerical Criterion for $\boldsymbol\Dcal$-Stability}}
Let $\Gcal\subset\PGL_3$ be a finite subgroup such
that~$N(\Gcal)^\circ=\Dcal$, and let $f\in\Rat_2^2(\Gcal)$.
\begin{align*}
  \text{$f$ is $N(\Gcal)$-unstable}
  &\;\Longleftrightarrow\;
  \text{$\mu^{\Ocal(1)}(f,L_{k,\ell})<0$ for some $(k,\ell)$.} \\
  \text{$f$ is $N(\Gcal)$-semistable}
  &\;\Longleftrightarrow\;
  \text{$\mu^{\Ocal(1)}(f,L_{k,\ell})\ge 0$ for all $(k,\ell)$.} \\
  \text{$f$ is $N(\Gcal)$-stable}
  &\;\Longleftrightarrow\;
  \text{$\mu^{\Ocal(1)}(f,L_{k,\ell}) > 0$ for all $(k,\ell)\ne(0,0)$.} 
\end{align*}

\section{Maps with an Automorphism of Prime Order $p\ge5$}
\label{section:autpge5}
In this section we analyze maps having an automorphism of prime order~$p\ge5$.

\begin{proposition}
\label{proposition:disteps}
Let~$K$ be an algebraically closed field of characteristic~$0$, let
$p\ge5$ be prime, let $\z$ be a primitive $p$'th root of unity,
let~$m\in\ZZ/p\ZZ$, and let $f\in\Rat_2^2$ be a dominant rational map
such that $\t_m\in\Aut(f)$.  Choose $0\le \e<p$ so that the lifts
$\hat f:\AA^3\to\AA^3$ of~$f$ satisfy $\hat f^{\t_m}=\z^\e \hat f$.
Then one of the following is true\textup:
\begin{parts}
\Part{(a)}
$f$ is $\Dcal$-unstable.
\Part{(b)}
$f$ is $\Dcal$-semistable, but not $\Dcal$-stable, and is
$\PGL_3(K)$-conjugate to a rational map of the form
$[aX^2+YZ,bXY,cXZ]$ with~$bc\ne0$. Further
$(m,\e)=(-1,0)$.
\Part{(c)}
$p=5$, and $f$ is $\Dcal$-stable and $\PGL_3(K)$-conjugate
to the rational map $[YZ,X^2,Y^2]$ with $(m,\e)=(3,4)$.
\Part{(d)}
$p=7$, and $f$ is $\PGL_3$-stable and $\PGL_3(K)$-conjugate to the
morphism $[Z^2,X^2,Y^2]$ with $(m,\e)=(3,6)$.
\end{parts}
\end{proposition}  
\begin{proof}
We write~$f=f_{m,\e}$ to help keep track of the dependence on~$m$
and~$\e$.  The assumption that~$f$ is dominant implies that its
coordinate functions are non-zero, so it necessarily includes at least
three monomials.  Looking at Table~\ref{table:effectoftm}, we see that
each of the quantities~$0$,~$1$, and~$m$ appears three times, while
each of the quantities
\begin{equation}
  \label{eqn:mlist}
  -m,\,1-m,\,2-m,\,-1,\,2,\,m-1,\,m+1,\,2m-1,\,2m
\end{equation}
appears exactly once.

We suppose first that $\e\notin\{0,1,m\}$. Then the only way
for~$f_{m,\e}$ to have at least three monomials is for at least three
of the quantities in the list~\eqref{eqn:mlist} to be equal. Our assumption
that $p\ge5$ means that the elements in each of the subsets
\[
  \{-m,1-m,2-m\},\quad
  \{-1,2\},\quad \{m-1,m+1\}, \quad \{2m-1,2m\},
\]
remain distinct when reduced modulo~$p$, so in order to obtain three
equal values modulo~$p$, we first choose three of these four sets,
then choose an element from each set, then equate the three quantities
and solve for~$m$ modulo~$p$.  This gives a total of~$44$
possibilities, although many of them give no value for~$m$, since each
choice yields two equations for the one quantity~$m$.  Further, some
choices give $\e\in\{0,1,m\}$, which we are not presently considering.
We do not know a clever way to do this computation, but working
through the complete set of possibilities, we find that exactly~$10$
choices yield values of~$m$, and all but two of these require
either~$p=5$ or~$p=7$. The data and resulting maps are listed in
Table~\ref{table:mcollisionall}.

\begin{table}[ht]
{\tiny\[
\begin{array}{|c|c|c||c|c|c||c|} \hline
\multicolumn{3}{|c||}{\textup{Values from Table
    \ref{table:effectoftm}}} & m & \e & p & f_{m,\e}\\ \hline\hline
  -1 & m - 1 & 2 m - 1 & 0 & -1 & \text{all $p$}  & [0,aX^2+bZ^2+cXZ,0]\\ \hline
 2 & m + 1 & 2 m & 1 & 2 & \text{all $p$}  & [aY^2+bZ^2+cYZ,0,0]\\ \hline\hline
-1 & m + 1 & -m + 2 & 3 & -1 & 5  & [a YZ, b X^2, c Y^2 ]\\ \hline
-1 & -m + 1 & 2 m & 2 & -1 & 5    & [a Z^2, b X^2, c XY ] \\ \hline
 2 & m - 1 & -m & 3 & 2 & 5       & [a Y^2, b XZ, c X^2 ]\\ \hline
2 & -m + 1 & 2 m - 1 & 4 & 2 & 5  & [a Y^2, b Z^2, c XY ]\\ \hline
m + 1 & -m & 2 m - 1 & 2 & 3 & 5  & [a YZ, b Z^2, c X^2 ]\\ \hline
m - 1 & -m + 2 & 2 m & 4 & 3 & 5  & [a Z^2, b XZ, c Y^2 ]\\ \hline\hline
-1 & -m + 2 & 2 m & 3 & -1 & 7    & [a Z^2, b X^2, c Y^2]\\ \hline
2 & -m & 2 m - 1 & 5 & 2 & 7      & [a Y^2, b Z^2, c X^2] \\ \hline
\end{array}
\]
}
\caption{Values of $m\bmod p$ and $\e\notin\{0,1,m\}$ such that $f_{m,\e}$ has at least $3$ monomials}
\label{table:mcollisionall}
\end{table} 

In Table~\ref{table:mcollisionall}, we require~$abc\ne0$, since we
need at least three monomials.  We start by noting that the maps
$f_{0,-1}$ and $f_{1,2}$, which work for all~$p$, are clearly
non-dominant (indeed, they are constant maps), so they may be
discarded. (It is also easy to check that they are $\Dcal$-unstable.)

We next note that the six families of maps for~$p=5$
are~$\Scal_3$-con\-ju\-gates, i.e., they may be obtained from one another
by permuting the variables.  It thus suffices to
consider~$f_{3,-1}=[aYZ,bX^2,cY^2]$, which is a dominant rational map
having a single point~$[0,0,1]$ of indeterminacy.  Using
Table~\ref{table:effectofL}, we find that
\[
  \mu^{\Ocal(1)}(f_{3,-1},L_{k,\ell}) =\max\{2k,-2k+\ell,-k-3\ell\}.
\]
The identity
\[
  7\cdot(2k) + 6(-2k+\ell) + 2(-k-3\ell) = 0 
\]
shows that at least one of the quantities in parentheses is non-negative,
and indeed unless $k=\ell=0$, one of them is positive. Hence
\[
  \inf_{(k,\ell)\ne(0,0)} \mu^{\Ocal(1)}(f_{3,-1},L_{k,\ell})
  = \inf_{(k,\ell)\ne(0,0)} \max\{2k,-2k+\ell,-k-3\ell\} > 0,
\]
which proves that~$f_{3,-1}$ is $\Dcal$-stable.

In order to obtain the map in~(c), we observe that the~$a,b,c$
coefficients of~$f_{3,-1}$ are twist parameters. To see this,
let~$\s(X,Y,Z)=[uX,vY,wZ]$. Then
\[
  f_{3,-1}^\s = [v^2w^2 aYZ, u^3w bX^2, uv^3 cY^2],
\]
so setting $u^{20}=a^3b^{-6}c^{-2}$, $v^{20}=a^{-1}b^2c^{-6}$, and
$w^{20}=a^{-9}b^{-2}c^6$ (with an appropriate choice of $20$'th
roots) yields $f_{3,-1}^\s=[YZ,Z^2,X^2]$. Thus the family~$f_{3,-1}$ for $p=5$
is a single~$\Dcal$-orbit.

Similarly, the two families of maps for~$p=7$ are conjugate via a
cyclic permutation of the variables. It is also clear that they are
morphisms, so in particular they are stable~\cite{MR2741188}.
Further, just as in~(c), the coefficients are twist parameters. Thus
for~$p=7$ and~$\s(X,Y,Z)=[uX,vY,wZ]$ we have
\[
  f_{3,-1}^\s =  [vw^3 a Z^2, u^3w b X^2, uv^3 c Y^2],
\]
so setting $u^{28}=a^3b^{-9}c^{-1}$, $v^{28}=a^{-1}b^3c^{-9}$, and
$w^{28}=a^{-9}b^{-1}c^3$ (with an appropriate choice of $28$'th roots)
yields $f_{3,-1}^\s=[Z^2,X^2,Y^2]$.  Thus the family~$f_{3,-1}$ for
$p=7$ is also a single~$\Dcal$-orbit.

This completes the classification of dominant semistable
maps~$f_{m,\e}$ with $\e\notin\{0,1,m\}$. We next observe that if
$\e\in\{0,1,m\}$ and~$f_{m,\e}$ has exactly three monomials, then
Tables~\ref{table:effectoftm} and~\ref{table:effectofL} give the
following three maps and their numerical invariants:
\[\begin{array}{|c|c|c|} \hline
\e & f_{m,\e} & \mu^{\Ocal(1)}(f_{m,0},L_{k,\ell}) \\ \hline\hline
0 & [aX^2,bXY,cXZ] & -k \\ \hline
1 & [aXY,bY^2,cYZ] & -\ell \\ \hline
m & [aXZ,bYZ,cZ^2] & k+\ell \\ \hline
\end{array}
\]
Thus in all cases~$f_{m,\e}$ induces the linear map~$[aX,bY,cZ]$,
and we can find a~$(k,\ell)\ne(0,0)$ making $\mu^{\Ocal(1)}(f_{m,0},L_{k,\ell})<0$,
so all of these maps are~$\Dcal$-unstable.

We now assume that $\e\in\{0,1,m\}$ and that~$f_{m,\e}$ has at least
four monomials.

\paragraph{\framebox{$\boldsymbol{\e=0}$}}
Since $f_{m,0}$ has four or more monomials,
Table~\ref{table:effectoftm} tells us that~$p$ must divide one of the
quantities in the set
\[
   \{m-2,m-1,m,m+1,2m-1\}.
\]
Since~$\t_m$ depends only on~$m$ modulo~$p$, this gives five
possibilities:
\[\begin{array}{|c|c|} \hline
m\bmod p & f_{m,0}(X,Y,Z) \\ \hline\hline
2 & [aX^2,bXY,cY^2+dXZ] \\ \hline
1 & [aX^2,bXY+cXZ,dXY+eXZ] \\ \hline
0 & [aX^2+bZ^2+cXZ,dXY+eYZ,fX^2+gZ^2+hXZ]\\ \hline
-1 & [aX^2+bYZ,cXY,dXZ] \\ \hline
2^{-1} & [aX^2,bZ^2+cXY,dXZ] \\ \hline
\end{array}
\]
For each of these families we use Table~\ref{table:effectofL} to
compute
\begin{align*}
  \mu^{\Ocal(1)}(f_{2,0},L_{k,\ell})
  &\le \max\{-k,-k-3\ell\} \xrightarrow{(k,\ell)=(1,0)} -1,
  \\
  \mu^{\Ocal(1)}(f_{1,0},L_{k,\ell})
  &\le \max\{-k,2\ell,-2k-2\ell\} \xrightarrow{(k,\ell)=(2,-1)} -2,
  \\
  \mu^{\Ocal(1)}(f_{0,0},L_{k,\ell})
  &\le \max\{-k,3k+2\ell,k+\ell,-3k-\ell\} \xrightarrow{(k,\ell)=(1,-2)} -1,
  \\
  \mu^{\Ocal(1)}(f_{-1,0},L_{k,\ell})
  &= \max\{-k,2k\} \ge 0~\text{for all $(k,\ell)\ne(0,0)$,} 
  \\
  \mu^{\Ocal(1)}(f_{1/2,0},L_{k,\ell})
  &\le \max\{-k,2k+3\ell\} \xrightarrow{(k,\ell)=(1,-1)} -1.
\end{align*}
Thus $f_{-1,0}$ is $\Dcal$-semi-stable provided $b\ne0$ and $a,c,d$
are not all~$0$, while the maps in the other four families are
$\Dcal$-unstable. Making a change of variables
$[X,Y,Z]\to[X,b^{-1}Y,Z]$, we can make the coefficient of~$YZ$
in~$f_{-1,0}$ equal to~$1$. It is also clear that if~$c$ or~$d$
is~$0$, then~$f_{-1,0}$ is not dominant. This gives the family of maps
in~(b).

\paragraph{\framebox{$\boldsymbol{\e=1}$}}
The straightforward approach is to use the assumption that $f_{m,1}$
has four or more monomials and Table~\ref{table:effectoftm} to deduce
that $m \in \{2^{-1},1,0,2,-1\} \bmod p$, which leads to the  five
families of maps:
\[
  \begin{array}{|c|c|} \hline
    m\bmod p & f_{m,1}(X,Y,Z) \\ \hline\hline
    2 & [aXY,bY^2+cXZ,dYZ] \\ \hline
    1 & [aXY+bXZ,cY^2+dZ^2+eYZ,eY^2+fZ^2+gYZ] \\ \hline
    0 & [aXY+bYZ,cY^2,dXY+eYZ]\\ \hline
    -1 & [aXY,bY^2,cX^2+dYZ] \\ \hline
    2^{-1} & [aZ^2+bXY,cY^2,cYZ] \\ \hline
  \end{array}
\]
Up to~$\Scal_3$-conjugation, these are exactly the five families that
we found for~$\e=0$, so we obtain nothing new.

An alternative is to let $\s(X,Y,Z)=(Y,X,Z)$ and to observe that
\[
  (\hat{f}_{m,1}^\s)^{\t_{1-m}} = \hat{f}_{m,1}^\s.
\]
Thus the map~$f_{m,1}^\s$ is in the family of maps~$f_{1-m,0}$, and
since~$f_{m,1}^\s$ and~$f_{m,1}$ have the same number of non-zero
monomials, the set of maps with~$\e=1$ is equal to the set
of~$\s$-conjugates of the maps with~$\e=0$.  (It's also amusing to
note that the set of~$m$ values that we obtained for~$\e=0$ is
invariant under $m\to1-m$.)

\paragraph{\framebox{$\boldsymbol{\e=m}$}}
Again using the assumption that $f_{m,1}$ has four or more monomials,
Table~\ref{table:effectoftm} tells us that
\[
  m \in \{2^{-1},1,0,2,-1\} \bmod p.
\]
We are currently dealing with the case $\e=m$, and we have already
analyzed the cases $\e=0$ and $\e=1$, so it remains to consider
$m\in\{2^{-1},2,-1\}$.  This gives three families of maps:
\[\begin{array}{|c|c|} \hline
  m\bmod p & f_{m,1}(X,Y,Z) \\ \hline\hline
  2 & [aY^2+bXZ,cYZ,dZ^2] \\ \hline
  -1 & [aXZ,bX^2+cYZ,dZ^2] \\ \hline
  2^{-1} & [aXZ,bYZ,cZ^2+dXY] \\ \hline
\end{array}
\]
These three families are~$\Scal_3$-conjugate to three of the families
that we found for~$\e=0$. Hence up to~$\PGL_3$ equivalence, we again
obtain nothing new.
\end{proof}

The next step is to compute the full automorphism groups of the maps
appearing in Proposiiton~\ref{proposition:disteps}.

\begin{proposition}
\label{proposition:C5C7etc}
\begin{parts}
\Part{(a)}
Let $f=[YZ,X^2,Y^2]$, and let $\z$ be a primitive $5$'th root of unity. Then
\[
  \Aut(f) = \left\< \SmallMatrix{1&0&0\\0&\z&0\\0&0&\smash[t]{\z^3}\\} \right\> \cong C_5.
\]
\Part{(b)}
Let $f=[Z^2,X^2,Y^2]$, and let $\z$ be a primitive $7$'th root of unity. Then
\[
  \Aut(f) = \left\< \SmallMatrix{1&0&0\\0&\z&0\\0&0&\smash[t]{\z^3}\\},
  \SmallMatrix{0&1&0\\0&0&1\\1&0&0\\}\right\> \cong C_7\rtimes C_3.
\]
\Part{(c)}
Let $(a,b,c)\in K^3$ with $bc\ne0$, and let $f=[aX^2+YZ,bXY,cXZ]$.
Then
\[
  \Aut(f) \supset \left\{
  \SmallMatrix{1&0&0\\0&t&0\\0&0&\smash[t]{t^{-1}}\\} : t\in\GG_m
  \right\} \cong \GG_m.
\]
\end{parts}
\end{proposition}
\begin{proof}
(a)\enspace
We see by inspection that $\Aut(f)\supseteq\<\t_3\>\cong C_5$.
We claim that this is the full automorphism group.  The
indeterminacy and critical loci of~$f$ are
\[
  I(f)=\bigl\{[0,0,1]\bigr\}
  \quad\text{and}\quad
  \Crit(f)=\{4XY^2=0\}.
\]
Since any~$\f\in\Aut(f)$ preserves both of these sets, with
their multiplicities, we see that~$\f$ leaves both of the lines~$X=0$
and~$Y=0$ invariant. Hence~$\f$ necessarily has the form
$\f=\SmallMatrix{\a&0&0\\ 0&\b&0\\ \g&\d&\e\\}$. Comparing the first
coordinates of~$f\circ\f$ and~$\f\circ f$ (note $\a\b\ne0$),
\begin{align*}
  f\circ\f(X,Y,Z)&=[\b Y(\g X+\d Y+\e Z),\ldots],\\
  \f\circ f(X,Y,Z)&=[\a YZ,\ldots],
\end{align*}
we see that~$\g=\d=0$, i.e., the map~$\f$ is diagonal. Without loss of
generality, we write~$\f\in\PGL_3$ as $\f(X,Y,Z)=[X,vY,wZ]$, and then
\[
  f^\f = [v^2 w^2 YZ,  w X^2, v^3 Z^2].
\]
Hence $f^\f=f$ if and only if $v^2w^2=w=v^3$.  Substituting $w=v^3$
into $v^2w^2=v^3$ gives $v^8=v^3$, so $v^5=1$.  Therefore~$v$ is a
$5$'th root of unity and $w=v^3$, so~$\f\in\<\t_3\>$.
\par\noindent  (b)\enspace 
We see by inspection that $\Aut(f)\supseteq\<\t_3\>\cong C_7$, but it
turns out that $\Aut(f)$ is strictly larger than this. Precisely, if
we let $\pi(X,Y,Z)=[Z,X,Y]$, then it is easy to check that $f^\pi=f$,
so $\pi\in\Aut(f)$.  Also, we compute $\pi^{-1}\t_3\pi=\t_3^2$,
so~$\Aut(f)$ contains the semi-direct product
$\<\t_3\>\rtimes\<\pi\>\cong C_7\rtimes C_3$.  We claim that this is
the full automorphism group of~$f$.

The critical locus of~$f$ is
\[
  \Crit(f) = \{8XYZ=0\},
\]
so $\Crit(f)$ consists of the three lines $XYZ=0$. Any
$\s\in\Aut(f)$ must permute these lines and their intersection
points. The map~$\pi$ is a cyclic permutation of the intersection
points, so replacing~$\s$ by~$\pi^{\pm1}\s$ if necessary, we may assume
that~$\s$ fixes~$[1,0,0]$ and either fixes or swaps~$[0,1,0]$
and~$[0,0,1]$. If~$\s$ fixes all three points, then~$\s$ is a diagonal
map, say $\s(X,Y,Z)=[X,vY,wZ]$, and we have
\[
  f^\s = [vw^3 Z^2, w X^2, v^3 Y^2 ].
\]
Hence $f^\s=f$ if and only if $vw^3=w=v^3$.  Substituing $w=v^3$ into
$vw^3=v^3$ gives $v^{10}=v^3$, so $v\in\bfmu_7$ and $w=v^3$.  Hence
$\s\in\<\t_3\>$.

It remains to deal with the case that~$\s$ fixes $[1,0,0]$ and
permutes $[0,1,0]$ and~$[0,0,1]$. But then we would have
$f^\s(1,0,0)=[0,0,1]$, while $f(1,0,0)=[0,1,0]$,
so~$f^\s$ cannot equal~$f$. (More precisely, the
maps~$f$ and~$f^\s$ are inverses in their action on the
three points.)  This completes the proof that~$\Aut(f)$ is
generated by~$\t_3$ and~$\pi$.
\par\noindent  (c)\enspace
It is trivial to check that the indicated copy of~$\GG_m$ is contained in~$\Aut(f)$.
A more detailed analysis, which we leave to the interested reader,
can be used to show that~$\Aut(f)\cong\GG_m$.
\end{proof}

\section{Maps with Automorphism Group Containing $C_p\times C_p$ with $p\ge3$}
\label{section:autCpCp}
Our goal in this section is essentially a non-existence result.
Somewhat surprisingly, the case~$p=3$ will be crucial to our analysis
of maps whose automorphism group contains a copy of~$C_4$. We also
note that the proposition is wildly incorrect for~$p=2$, and indeed we
devote a long section (Section~\ref{section:autC2C2}) to classifying maps
whose automorphism group contains a copy of~$C_2^2$.

\begin{proposition}
\label{proposition:fCpCp}
Let $K$ be an algebraically closed field of characteristic~$0$.
Let~$p\ge3$ be prime, and let $f\in\Rat_2^2(K)$ have the property that
$\Aut(f)$ contains a copy of~$C_p^2$.  Then
either~$f$ is a linear map or else~$f$ is not dominant.
\end{proposition}
\begin{proof}
Let $G\subset\PGL_3(K)$ be a subgroup of type~$C_p^2$ that is
contained in~$\Aut(f)$. Lemma~\ref{lemma:GinPGL3} tells us
that after an appropriate conjugation, we may assume that
\[
  G = \<\a,\b\> \quad\text{with}\quad \a=
  \SmallMatrix{1&0&0\\0&\z&0\\0&0&1\\} \quad\text{and}\quad \b=
  \SmallMatrix{1&0&0\\0&1&0\\0&0&\z\\},
\]
where~$\z$ is a primitive $p$'th root of unity. The following table
describes the action of~$\a$ and~$\b$ on quadratic monomials that
might appear in~$f$.  An entry~$(i,j)$ in the table means that~$\a$
multiplies the monomial by~$\z^i$ and that~$\b$ multiplies the
monomial by~$\z^j$.
{\small
\[
  \begin{array}{|c||c|c|c|c|c|c|} \hline
    & X^2 & Y^2 & Z^2 & XY & XZ & YZ \\ \hline\hline
  \text{$X$-coordinate} & (0,0) & (2,0) & (0,2) & (1,0) & (0,1) & (1,1) \\ \hline
  \text{$Y$-coordinate} & (-1,0) & (1,0) & (-1,2) & (0,0) & (-1,1) & (0,1) \\ \hline
  \text{$Z$-coordinate} & (0,-1) & (2,-1) & (0,1) & (1,-1) & (0,0) & (1,0) \\ \hline
  \end{array}
\]
}%
There are a number of possible families of maps invariant for~$G$, indexed
by the pairs~$(i,j)$ modulo~$p$. The most interesting case is~$p=3$, so~$-1\equiv2$,
in which case there are~$9$ families of maps as given in the following table:
{\small
\[
  \begin{array}{|c||c|c|c|} \hline  
  (i,j)       & 0 & 1 & 2 \\ \hline\hline
  0      & [aX^2,bXY,cXZ] & [aXY,bY^2,cYZ] & [aY^2,bX^2,0] \\ \hline
  1   & [aXZ,bYZ,cZ^2] & [aYZ,0,0] & [0,aXZ,0] \\ \hline
  2 & [aZ^2,0,bX^2] & [0,0,aXY] & [0,aZ^2,bY^2] \\ \hline
  \end{array}
\]  
}%
Three of these families coincide with the linear map~$[aX,bY,cZ]$,
while the other six families clearly give non-dominant maps. And
if~$p\ge5$, then we obtain the same three linear maps, plus nine
additional maps defined by a single monomial.

According to Lemma~\ref{lemma:GinPGL3}, it remains to deal with the
case that $p=3$ and, after appropriate conjugation, 
\[
  G = \<\t_2,\pi\> \quad\text{with}\quad
  \t_2 =  \SmallMatrix{1&0&0\\0&\z&0\\0&0&\smash[t]{\z^2}\\} \quad\text{and}\quad
  \pi=  \SmallMatrix{0&1&0\\0&0&1\\1&0&0\\},
\]
where~$\z$ is a primitive cube root of unity and~$\t_2$ is as in
Section~\ref{section:diagstability}.  Suppose that $f\in\Rat_2^2$ with
$\t_2\in\Aut(f)$.  Using Table~\ref{table:effectoftm} with $m=2$
and entries reduced modulo~$3$, we find that~$f$ has one of the
following forms:
\begin{align*}
  F &:= [aX^2+bYZ,cZ^2+dXY,eY^2+gXZ], \\
  G &:= [aZ^2+bXY,cY^2+dXZ,eX^2+gYZ], \\
  H &:= [aY^2+bXZ,cX^2+dXZ,eZ^2+gXY].
\end{align*}
Conjugating by the cyclic permutation~$\pi(X,Y,Z)=[Y,Z,X]$ yields
\begin{align*}
  F^\pi &= [eZ^2+gXY,aY^2+bXZ,cX^2+dYZ],\\
  G^\pi &= [eY^2+gXZ,aX^2+bYZ,cZ^2+dXY],\\
  H^\pi &= [eX^2+gYZ,aZ^2+bXY,cY^2+dXZ].
\end{align*}
Since~$F$ and~$F^\pi$ have no non-zero monomials in common,
it follows that~$\pi\notin\Aut(F)$, and similarly for~$G$ and~$H$.
This completes the proof that there are no maps~$f\in\Rat_2^2$
with $G\subseteq\Aut(f)$.
\end{proof}

\section{Maps with Automorphism Group Containing $C_2\times C_2$}
\label{section:autC2C2}

In this section we classify maps with $\Gcal_{2,2}\subseteq\Aut(f)$.

\begin{proposition}
\label{proposition:classifyC2C2}
Let $K$ be an algebraically closed field of characteristic~$0$,
let $\Gcal_{2,2}$ be the group described in
Theorem~$\ref{theorem:main1}$, and let $f\in\Rat_2^2(\Gcal_{2,2})^\ss$ be a
dominant map of degree~$2$.  Then~$f$ is~$N(\Gcal_{2,2})$-stable and
one of the following holds\textup:
\begin{parts}
\Part{(a)}
$f$ is $N(\Gcal_{2,2})$-conjugate to a map of the form
\[
  [X^2+Y^2-Z^2,dXY,eXZ]\quad\text{with $d,e\in K^*$.} 
\]
The $N(\Gcal_{2,2})$-conjugacy class of the map~$f$ is uniquely
determined by the unordered pair~$\{d,e\}$. The automorphism group
of~$f$ is
\[
  \Aut(f) = \begin{cases}
    \Gcal_{2,2} &\text{if $d\ne e$,} \\
    \GG_m\rtimes C_2&\text{if $d=e$.} \\
    \end{cases}
\]
Further, $\deg(f^n)=2^n$.
\Part{(b)}
$f$ is $N(\Gcal_{2,2})$-conjugate to a map of the form
\[
  [Y^2-Z^2,XY,eXZ]\quad\text{with $e\in K^*$.}
\]
The $N(\Gcal_{2,2})$-conjugacy class of the map~$f$ is uniquely
determined by the unordered pair~$\{e,e^{-1}\}$. The automorphism
group of~$f$ is
\[
  \Aut(f) = \begin{cases}
    \Gcal_{2,2} &\text{if $e\ne \pm1$,} \\
    \Scal_3\Gcal_{2,2} &\text{if $e=-1$,} \\
    \GG_m\rtimes C_2&\text{if $e=1$.} \\
    \end{cases}
\]
Further, if $e^2$ is not an odd-order root of unity, then
$\deg f^n=n+1$, while if~$e^{4k+2}=1$, then $f^{4k+2}=[X,Y,Z]$.
\Part{(c)}
$f$ is $N(\Gcal_{2,2})$-conjugate to the map $[YZ,XZ,XY]$.  The
automorphism group of~$f$ is $\Aut(f)=\Scal_3\Gcal_{2,2}\cong\Scal_4$.
Further,~$f^2=[X,Y,Z]$.
\end{parts}
\end{proposition}
\begin{proof}
Let
\[
  \a = \SmallMatrix{1&0&0\\0&-1&0\\0&0&1\\} \quad\text{and}\quad
  \b= \SmallMatrix{1&0&0\\0&1&0\\0&0&-1\\},
\]
so $\Gcal_{2,2}=\<\a,\b\>\cong C_2^2$ is the group that we assume
is contained in~$\Aut(f)$. The following table describes the action
of~$\a$ and~$\b$ on quadratic monomials that might appear in~$f$.  An
entry~$(i,j)$ in the table means that~$\a$ multiplies the monomial
by~$(-1)^i$ and that~$\b$ multiplies the monomial by~$(-1)^j$.
{\small
\[
  \begin{array}{|c||c|c|c|c|c|c|} \hline
                        & X^2   & Y^2   & Z^2   & XY    & XZ    & YZ \\ \hline\hline
  \text{$X$-coordinate} & (0,0) & (0,0) & (0,0) & (1,0) & (0,1) & (1,1) \\ \hline
  \text{$Y$-coordinate} & (1,0) & (1,0) & (1,0) & (0,0) & (1,1) & (0,1) \\ \hline
  \text{$Z$-coordinate} & (0,1) & (0,1) & (0,1) & (1,1) & (0,0) & (1,0) \\ \hline
  \end{array}
\]
}%
There are thus four families of $\Gcal_{2,2}$-invariant maps, indexed by
pairs $(i,j)$ modulo~$2$ (or equivalently, by characters~$\Gcal_{2,2}\to\CC^*$),
\begin{align*}
  f_{0,0} &= [aX^2+bY^2+cZ^2,dXY,eXZ], \\
  f_{1,0} &= [aXY,bX^2+cY^2+dZ^2,eYZ], \\
  f_{0,1} &= [aXZ,bYZ,cX^2+dY^2+eZ^2], \\
  f_{1,1} &= [aYZ,bXZ,cXY].
\end{align*}
Lemma~\ref{lemma:subgpnlzr}(b) tells us that
the normalizer~$N(\Gcal_{2,2})$ of~$\Gcal_{2,2}$ contains all of the permutation
matrices. In particular, the permutation $\pi(X,Y,Z)=[Y,Z,X]$ is in~$N(\Gcal_{2,2})$,
and applying~$\pi$ and~$\pi^2$ to~$f_{0,0}$ yields
\begin{align*}
  f_{0,0}^\pi &= [eXY,cX^2+aY^2+bZ^2,dYZ], \\
  f_{0,0}^{\pi^2} &= [dXZ,eYZ,bX^2+cY^2+aZ^2].
\end{align*}
Hence the families~$f_{1,0}$ and~$f_{0,1}$
are~$N(\Gcal_{2,2})$-conjugate to the family~$f_{0,0}$.

We next observe that a map in the~$f_{1,1}$ family is dominant if and
only if~$abc\ne0$. And under this assumption, conjugation by
$[uX,vY,wZ]\in\Dcal\subset N(\Gcal_{2,2})$ with $u^4=a^{-1}bc$,
$v^4=ab^{-1}c$, and $w^4=abc^{-1}$ transforms~$f_{1,1}$ into the
map~$[YZ,XZ,XY]$, i.e., the family~$f_{1,1}$ with $abc\ne0$ consists
of a single $N(\Gcal_{2,2})$-conjugacy class. This gives the map in~(c),
which by abuse of notation we continue to denote by~$f_{1,1}$.
The critical locus of~$f_{1,1}$ is the union of the coordinate axes,
\[
  \Crit(f_{1,1}) = \{X=0\} \cup \{Y=0\} \cup \{Z=0\}.
\]
Any~$\f\in\Aut(f_{1,1})$ thus leaves the union of the coordinate axes
invariant, from which we conclude that~$\f$ has the form
\[
  \f = \pi\circ\s
  \quad\text{for some $\pi\in\Scal_3$
    and some $\s=\SmallMatrix{1&0&0\\0&\b&0\\0&0&\g\\}\in\Dcal$.}
\]
But one easily checks that~$\Scal_3\subset\Aut(f_{1,1})$, so it
suffices to determine which diagonal matrices~$\s\in\Dcal$ are
in~$\Aut(f_{1,1})$.  Letting~$\s=[X,\b Y,\g Z]$, we find that
\[
  f_{1,1}^\s = [\b\g YZ, \b^{-1}\g XZ, \b\g^{-1} XY],
\]
so
\[
  \s\in\Aut(g) \Longleftrightarrow \b\g = \b^{-1}\g = \b\g^{-1}
  \Longleftrightarrow \b^2=\g^2=1 \Longleftrightarrow \s\in\Gcal_{2,2}.
\]
This completes the proof that
\[
  \Aut(f_{1,1}) = \Scal_3\Gcal_{2,2} \cong S_4,
\]
which also completes the proof of~(c).

We now concentrate on maps in the family~$f_{0,0}$, and to ease
notation, we drop the subscript and simply write
\[
  f = [aX^2+bY^2+cZ^2,dXY,eXZ].
\]
It is clear that~$f$ is dominant if and only if~$de\ne0$ and at least
one of~$a,b,c$ is non-zero. Further, if $b=c=0$, then $f=[aX,dY,eZ]$
has degree~$1$, so we may assume that one of~$b$ and~$c$ is
non-zero. And since the involution~$\s(X,Y,Z)=[X,Z,Y]\in
\Scal_3\subset N(\Gcal_{2,2})$ has the effect
\[
  f^\s = [aX^2+cY^2+bZ^2,eXY,dXZ]
\]
of switching the roles of~$b$ and~$c$, after another~$N(\Gcal_{2,2})$
conjugacy we may assume that $cde\ne0$.  Using this assumption and
Table~\ref{table:effectofL}, the action of the one-parameter
subgroup~$L_{k,\ell}(t)$ on~$f$ is
\[
  \mu^{\Ocal(1)}(f,L_{k,\ell})
  = \begin{cases}
    \max\{ -k, k-2\ell, 3k+2\ell\} & \text{if $b\ne0$,} \\
    \max\{ -k, 3k+2\ell\} & \text{if $b=0$.} \\
  \end{cases}
\]
Thus if $b=0$, then $\mu^{\Ocal(1)}(f,L_{1,-2})=-1$, so~$f$
is~$\Dcal$-unstable.

On the other hand, if $b\ne0$, then the identity
\[
  4(-k) +(k-2\ell) + (3k+2\ell) = 0
\]
shows that $\mu^{\Ocal(1)}(f,L_{k,\ell})>0$ for all $(k,\ell)\ne(0,0)$,
so~$f$ is $\Dcal$-stable.

We are reduced to studying~$f$ with $bcde\ne0$.  We conjugate by
a diagonal map~$\d=[uX,vY,wZ]\in\Dcal\subset N(\Gcal_{2,2})$ to obtain
\[
  f^\d = [u^2aX^2+v^2bY^2+w^2cZ^2,u^2dXY,u^2eXZ].
\]
Thus taking $v^2=b^{-1}$ and $w^2=-c^{-1}$, we may assume that $b=1$
and $c=1$.  Further, if $a\ne0$, then we may take~$u^2=a^{-1}$ to
reduce to maps with~$a=1$, while if $a=0$, then we may take
$u^2=d^{-1}$ to reduce to maps satisfying $d=1$.  It thus suffices to
anaylze the maps in the following two families:
\begin{align*}
  f &= [X^2+Y^2-Z^2,dXY,eXZ]\quad\text{with $d,e\in K^*$,} \\
  f &= [Y^2-Z^2,XY,eXZ]\quad\text{with $e\in K^*$.}
\end{align*}  

\par\noindent\framebox{$\boldsymbol{ f = [X^2+Y^2-Z^2,dXY,eXZ],\quad de\ne0}$}
\par
The indeterminacy and critical loci of~$f$ are
\[
  I(f) = \bigl\{[0,1,\pm1]\bigr\}
  \quad\text{and}\quad
  \Crit(f) = \{X=0\} \cup \{X^2-Y^2+Z^2=0\}.
\]
We consider two maps $f=[X^2+Y^2-Z^2,dXY,eXZ]$ and
$f'=[X^2+Y^2-Z^2,d'XY,e'XZ]$ and compute
\[
  \Hom(f,f') := \bigl\{\f\in\PGL_3(K) : f^\f = f' \bigr\}.
\]
Note that by taking~$f'=f$, we will obtain~$\Aut(f)$.

Every $\f\in\Hom(f,f')$ stabilizes the line~$\{X=0\}$ and
either fixes or permutes the two point~$[0,1,\pm1]$.
Thus~$\f$ has the form
\[
  \f = \SmallMatrix{1&0&0\\ \g&\a&\b\\ \d&\pm\b&\pm\a\\}
  \quad\text{with $\a^2\ne\b^2$,}
\]
where choosing the plus sign fixes~$[0,1,\pm1]$ and choosing the minus
sign swaps them.

We compare the second and third coordinates of~$f\circ\f$ and $\f\circ f'$,
\begin{align*}
  f\circ\f &= [*,\, \g d X^2 + \a d X Y + \b d X Z,
                     \d e X^2 \pm \b e X Y \pm \a e X Z],\\
  \f\circ f' &= [*,\, \g X^2 + \g Y^2 - \g Z^2 + \a d' X Y + \b e' X Z, \\
    &\omit\hfill$\displaystyle \d X^2 + \d Y^2 - \d Z^2 \pm \b d' X Y \pm \a e' X Z]$.
\end{align*}  
Since the second and third coordinates of~$f\circ\f$ have no~$Y^2$ term,
we conclude that $\g=\d=0$. Under this assumption, we find that
\begin{align*}
  f\circ\f &= [X^2+(\a^2-\b^2)(Y^2-Z^2), X (\a d Y + \b d Z), \pm X(\b e Y + \a e Z)], \\
  \f\circ f' &= [X^2+Y^2-Z^2,    X(\a d' Y + \b e' Z),  \pm X(\b d' Y + \a e' Z)].
\end{align*}
Hence $\f\in\Hom(f,f')$ if and only if
\[
  \a^2-\b^2=1
  \quad\text{and}\quad
  \a(d-d') = \b(d-e') = \b(e-d') = \a(e-e') = 0.
\]
This leads to three cases.

\par\noindent\framebox{$\boldsymbol{\a\b\ne0}$}
\par
Then we must have $d=e=d'=e'$, i.e., $f'=f$ and $d=e$.  The
automorphism group of these maps is the set of all matrices of the
form $\SmallMatrix{1&0&0\\0&\a&\b\\0&\pm\b&\pm\a\\}$ satisfying
$\a^2-\b^2=1$, which is easily seen to be isomorphic to~$\GG_m\rtimes C_2$.

\par\noindent\framebox{$\boldsymbol{ \b=0}$, $\boldsymbol{ \a=\pm1}$}
\par
Then $d'=d$ and $e'=e$, i.e., $f'=f$, and we obtain exactly four possible maps~$\f$,
namely the four maps in~$\Gcal_{2,2}$ that we already know are in~$\Aut(f)$.

\par\noindent\framebox{$\boldsymbol{\a=0}$, $\boldsymbol{ \b=\pm1}$}
\par
Then $d'=e$ and $e'=d$, so the maps
$[X^2+Y^2-Z^2,dXY,eXZ]$ and $[X^2+Y^2-Z^2,eXY,dXZ]$ are $N(\Gcal_{2,2})$-conjugate via
the permutation $[X,Y,Z]\to[X,Z,Y]$.

It remains to prove that $\deg(f^n)=2^n$, a task that we postpone to
Section~\ref{section:dyntopdeg}, where we  study the
degree sequences of all of the maps in this paper.
This completes our analysis for maps of the form~$f=[X^2+Y^2-Z^2,dXY,eXZ]$.

\par\noindent\framebox{$\boldsymbol{ f = [Y^2-Z^2,XY,eXZ],\quad e\ne0}$}  
\par
The iterates of~$f$ are explicitly described later in
Proposition~\ref{proposition:fY2Z2XYeXZ}, which in particular gives
the stated results for~$\deg(f^n)$.
To make our computation of~$\Aut(f)$ easier, we instead work with a
$\PGL_3$-conjugate of~$f$. Thus we let
\[
    \l(X,Y,Z) = [2X, Y - Z, Y + Z]
\]
and define
\[
  g(X,Y,Z) := f^\l(X,Y,Z) = \bigl[YZ, X(AY+BZ), X(BY+AZ) \bigr]
\]
with
\[
  A = \frac{1+e}{2}
  \quad\text{and}\quad
  B = \frac{1-e}{2}.
\]
We note that $A^2-B^2=e$.  The advantage of~$g$ over~$f$ is the fact
that the critical locus of~$g$ is the union of the coordinate axes,
\[
  \Crit(g) = \{X=0\} \cup \{Y=0\} \cup \{Z=0\}.
\]
As usual, we let~$g'$ be another map of this form, with~$e'$ in place
of~$e$, and we let $\f\in\Hom(g,g')$. Then~$\f$ leaves the union of
the coordinate axes invariant, from which we conclude that~$\f$ has
the form
\[
  \f = \pi\circ\s
  \quad\text{for some $\pi\in\Scal_3$
    and some $\s=\SmallMatrix{1&0&0\\0&\b&0\\0&0&\g\\}\in\Dcal$.}
\]
For each of the six elements of~$\Scal_3$ we need to compute the
effect of~$\pi\circ\s$ on~$g$. We let
\[
  S(X,Y,Z) = [X,Z,Y] \quad\text{and}\quad T(X,Y,Z)=[Y,Z,X]
\]
be generators for~$\Scal_3$. Our task is simplified by the observation
that~$S\in\Aut(g)$ and~$S\in\Aut(g')$, so it suffices to
take~$\pi\in\{I,T,T^2\}$. For each of these choices we compute the action
on~$g$,  
\begin{align*}
  g^{\s} &= \left[YZ,
    X\left(\frac{1}{\b\g} AY + \frac{1}{\b^2} BZ\right),
    X\left(\frac{1}{\g^2} BY + \frac{1}{\b\g} AZ\right) \right], \\
  g^{T\s} &= [ A XY + \g B YZ,\, *,\, *], \\
  g^{T^2\s} &= [ A XZ + \b B YZ,\, *,\, *].
\end{align*}

We consider three cases:
\par\noindent\framebox{$\boldsymbol{ e\ne\pm1,\; AB\ne0}$}
In this case the fact that~$g'$ has no~$XY$ or~$XZ$ in its first coordinate
rules out~$\pi=T$ or~$\pi=T^2$.

On the other hand, for~$\pi=I$ we have
\[
  \s\in\Hom(g,g')  \quad\Longleftrightarrow\quad
  \b\g = A/A'\quad\text{and}\quad \b^2=\g^2=B/B'.
\]
In particular, this can occur only if
\begin{align*}
  0 &=  (\b\g)^2 - \b^2\g^2 
  =  \left(\frac{A}{A'}\right)^2 - \left(\frac{B}{B'}\right)^2 
  = \frac{(AB')^2 - (A'B)^2}{(A'B')^2} \\
  &= \frac{1}{{(A'B')^2}} \left(\left(\frac{1+e}{2}\cdot\frac{1-e'}{2}\right)^2
      - \left(\frac{1+e'}{2}\cdot\frac{1-e}{2}\right)^2 \right)\\
  &= \frac{(e-e')(1-ee')}{(A'B')^2}.
\end{align*}
If $e'=e$, i.e., if $g'=g$, then we find that $\s\in\Aut(g)$ if and
only if $\b\g=\b^2=\g^2=1$, so if and only if $\b=\g=\pm1$.  This
gives two elements of~$\Aut(g)$, and composing with~$S$ gives two
additional elements. These elements form the copy of~$C_2^2$ that we
already know exists in~$\Aut(g)$. Further, if $e'=e^{-1}$, then
$A=eA'$ and $B=-eB'$, so we find that~$\s\in\Hom(g,g')$ if we take
$\b=\sqrt{-e}$ and $\g=-\b$. 

To recapitulate, we have shown that if $e\ne\pm1$, then
\[
  \Aut(g) = \left\< \SmallMatrix{1&0&0\\ 0&-1&0\\ 0&0&-1\\},
  \SmallMatrix{1&0&0\\ 0&0&-1&\\ 0&-1&0\\} \right\> \cong C_2^2,
\]
and that~$g'$ is $\PGL_3(K)$-conjugate to~$g$ if and only if
$e'\in\{e,e^{-1}\}$. Undoing the conjugation by~$\l$, we find that
if~$e\ne\pm1$, then $\Aut(f)=\Gcal_{2,2}$, and that~$f$ is
$N(\Gcal_{2,2})$-conjugate to~$f'$ if and only if $e'\in\{e,e^{-1}\}$.
 
\par\noindent\framebox{$\boldsymbol{ e= 1,\; A=1,\; B=0}$}
In this case the map~$g$ is simply
\[
  g = [YZ, XY,XZ].
\]
It satisfies~$g^2=[X,Y,Z]$, and~$\Aut(g)$ contains a copy of~$\GG_m$
in the form of all maps $[X,tY,t^{-1}Z]$, and it contains~$S$, so
$\GG_m\rtimes C_2\subseteq\Aut(g)$.  Since it is not needed for the
proof of our main theorem, we leave for the reader the proof that this
inclusion is an equality.

\par\noindent\framebox{$\boldsymbol{ e= -1,\; A=0,\; B=1}$}  
In this case the map~$g$ has the simple form
\[
  g = [YZ,XZ,XY],
\]
and we observe that~$g$ is the map that we already analyzed in~(c).
In particular,~$\Aut(g)=\Scal_3\Gcal_{2,2}\cong S_4$.  This proves
that~$\Aut(f)\cong S_4$, in fact one can easily check that~$\l$
normalizes the group~$\Scal_3\Gcal_{2,2}$, so
$\Aut(f)=\Scal_3\Gcal_{2,2}$.
\end{proof}

We next give an explicit formula for the iterates of
the family of maps in Proposition~\ref{proposition:classifyC2C2}(b).

\begin{proposition}
\label{proposition:fY2Z2XYeXZ}
Let~$e\in K^*$, and for $k\ge0$, let
\[
  U_k(Y,Z) = Y^2 - e^{2k} Z^2.
\]
Then the iterates of the map~$f = [Y^2-Z^2,XY,eXZ]$ are
given by the formulas
\begin{multline*}
  f^n(X,Y,Z) \\*
  = \begin{cases}
  [XU_1U_3\cdots U_{n-1},  YU_0U_2\cdots U_{n-2}, e^nZU_0U_2\cdots U_{n-2}] \\
       \omit\hfill\text{if $n$ is even,} \\
  [U_0U_2\cdots U_{n-1},XYU_1U_3\cdots U_{n-2},e^nXZU_1U_3\cdots U_{n-2}] \\
       \omit\hfill\text{if $n$ is odd.} \\
  \end{cases}
\end{multline*}
\begin{parts}
\Part{(a)}
If $e^{2k}\ne1$ for all odd integers~$k$, then
\[
  \deg (f^n) = n+ 1\quad\text{for all $n\ge0$.}
\]
\Part{(b)}
If $e^{2k}=1$ for some odd integer~$k$, then
\[
  f^{2k}(X,Y,Z)=[X,Y,Z].
\]
\end{parts}
\end{proposition}
\begin{proof}
We note that $f=\bigl[U_0(Y,Z),XY,eXZ]$.  The proof of the formulas
for~$f^n$ is an easy induction on~$n$, using the identity
\begin{equation}
  \label{eqn:UkUk1}
  U_k(WY,eWZ) = W^2 U_{k+1}(Y,Z).
\end{equation}
This allows us to compute
\begin{align*}
  f^{2k+2}(X,Y,Z)
  &= f^{2k+1}\bigl(U_0(Y,Z),XY,eXZ\bigr)
    \quad\text{definition of $f$,} \\
  &= \bigl[ (U_0U_2\cdots U_{2k})(XY,eXZ), \\*
    &\omit\hfill$  U_0(Y,Z)\cdot XY\cdot(U_1U_3\cdots U_{2k-1})(XY,eXZ),$ \\*
    &\omit\hfill$  e^{2k+1}\cdot U_0(Y,Z)\cdot eXZ\cdot(U_1U_3\cdots U_{2k-1})(XY,eXZ) \bigr]$ \\*
    &\omit\hfill induction hypothesis, \\
 &=\bigl[  X^{2k+2}(U_1U_3\cdots U_{2k+1})(Y,Z), \\*
     &\omit\hfill$ U_0(Y,Z)\cdot X^{2k+1}\cdot Y (U_2U_4\cdots U_{2k})(Y,Z),$ \\*
     &\omit\hfill$ e^{2k+2}\cdot U_0(Y,Z)\cdot X^{2k+1}\cdot Z (U_2U_4\cdots U_{2k})(Y,Z) \bigr]$ \\*
     &\omit\hfill\text{using \eqref{eqn:UkUk1}.}
\end{align*}  
Canceling $X^{2k+1}$ gives the desired formula. The computation of~$f^{2k+1}$ using
the formula for~$f^{2k}$ is similar. This completes the proof of the formulas for~$f^n$.
\par
Since $\deg U_k(Y,Z)=2$, we see immediately from the formulas for~$f^n$
that $\deg(f^n)\le n+1$, with equality if and only if the coordinate
functions have no common factor. Since~$e\ne0$, we see that a common factor occurs
if and only if some odd index~$U_{2\ell+1}(Y,Z)$ has a factor in common with some even
index~$U_{2m}(Y,Z)$. But
\begin{align*}
  \Resultant\bigl(U_{2\ell+1}(Y,Z),U_{2m}(Y,Z)\bigr) 
  &= \Resultant(Y^2 - e^{4\ell+2} Z^2, Y^2 - e^{4m} Z^2) \\*
  &= e^{8m} (e^{2(2\ell-2m+1)} - 1)^2.
\end{align*}
Hence if~$e^2$ is not an odd-order root of unity, then there is no cancelation and
$\deg(f^n)=n+1$. Finally, if $e^2$ is an odd order root of unity, say $e^{4\ell+2}=1$,
then for all $k\ge0$ we have
\[
  U_{k+2\ell+1}(Y,Z)
  = Y^2 - e^{2(k+2\ell+1)} Z^2
  = Y^2 - e^{2k} Z^2
  = U_k(Y,Z).
\]
This allows us to switch even index~$U_k$'s with odd index~$U_k$'s. In
particular, using this identity in the formula for~$f^{4\ell+2}$, we
find that all of the~$U_k$ factors cancel,
leaving~$f^{4\ell+2}=[X,Y,Z]$.
\end{proof}

\section{Maps with an Automorphism of Order $4$}
\label{section:autC4}
In this section we classify maps in~$\Rat_2^2$ that admit an
automorphism of order~$4$.

\begin{proposition}
\label{proposition:autc4types}
Let $K$ be an algebraically closed field of characteristic~$0$, let
$\Gcal_4$ be the group described in Theorem~$\ref{theorem:main1}$, and
let $f\in\Rat_2^2(\Gcal_4)^\ss$ be a dominant map of degree~$2$ with
finite automorphism group.  Then~$f$ is~$N(\Gcal_4)$-stable and one of
the following holds\textup:
\begin{parts}
\Part{(a)}
$f$ is $N(\Gcal_4)$-conjugate to a map of the form
\[
  f_{a,e} := [aX^2+Z^2,XY,Y^2+eXZ]\quad \text{with $a,e\in K$.}
\]
The automorphism group of~$f_{a,e}$ is given by
\[
\Aut(f_{a,e}) = \Gcal_4.
\]
Two maps~$f_{a,e}$ and~$f_{a',e'}$ are $N(\Gcal_4)$-conjugate if and only $(a,e)=(a',e')$.
\Part{(b)}
$f$ is $N(\Gcal_4)$-conjugate to a map of the form\footnote{See
  Proposition~\ref{proposition:c4f2autgm} for explicit formulas
  for the iterates of~$f_c$ and a detailed description of its
  geometry. In particular, although~$\Aut(f)$ is finite for $c\ne-1$, it
  turns out that $\Aut(f_c^2)$ always contains a copy of~$\GG_m$.}
\[
  f_c := [YZ,X^2+cZ^2,XY]\quad\text{with~$c\in K\setminus\{-1\}$.}
\]
The automorphism group of~$f_c$ is given by\footnote{In the
  excluded case $c=-1$, we have
  \text{$\Aut(f_{-1})\supseteq\GG_m\rtimes C_2$}.}
\begin{align*}
  \Aut(f_c) &= \Gcal_4 \quad\text{if $c\ne\pm1$,} \\
  \Aut(f_c) &\cong S_4 \quad\text{if $c=1$.}
\end{align*}
Two maps~$f_c$ and~$f_{c'}$ are $N(\Gcal_4)$-conjugate if and only
if $cc'=1$.
\end{parts}
\end{proposition}

\begin{proposition}
\label{proposition:C4notpossible}
Let~$K$ be an algebraically closed field of characteristic~$0$, and
let $f\in\Rat_2^2$ satisfy
\[
  \Aut(f) \supseteq \left\< \SmallMatrix{1&0&0\\0&i&0\\0&0&1\\} \right\>.
\]
Then~$f$ is~$\Dcal$-unstable.
\end{proposition}

\begin{proof}[Proof of Proposition~$\ref{proposition:autc4types}$]
Let $\z=i$ to be a primitive 4'th root of unity, so the map $[X,iY,-Z]$
corresponds to the matrix~$\t_2$ defined in
Section~\ref{section:diagstability}. Table~\ref{table:effectoftm}
with $m=2$ and entries reduced modulo~$4$ is  
{\small
\[
\begin{array}{|c||c|c|c|c|c|c|} \hline
  & X^2 & Y^2 & Z^2 & XY & XZ & YZ \\ \hline\hline
  {\text{$X$-coord}}
  & 0 & 2 & 0 & 1 & 2 & 3 \\ \hline
  {\text{$Y$-coord}}
  & 3 & 1 & 3 & 0 & 1 & 2 \\ \hline
  {\text{$Z$-coord}}
  & 2 & 0 & 2 & 3 & 0 & 1 \\ \hline
\end{array}
\]
}%
Hence the assumption that~$\t_2\in\Aut(f)$ leads to the following four families of maps:
\begin{align*}
  f_{2,0} &:= [aX^2+bZ^2, cXY, dY^2+eXZ], \\
  f_{2,1} &:= [aXY,bY^2+cXZ,dYZ], \\
  f_{2,2} &:= [aY^2+bXZ,cYZ,dX^2+eZ^2], \\
  f_{2,3} &:= [aYZ,bX^2+cZ^2,dXY].
\end{align*}
Conjugation by the permutation $\pi(X,Y,Z)=[Z,Y,X]\in N(\Gcal_4)$ gives 
\[
  f_{2,0}^\pi
  = [aX^2+bZ^2, cXY, dY^2+eXZ]^\pi
  = [dY^2+eXZ,cYZ,bX^2+aZ^2],
\]
so~$\pi$ identifies the families~$f_{2,0}$ and~$f_{2,2}$. (One also
easily checks that~$\pi$ stabilizes each of the families~$f_{2,1}$
and~$f_{2,3}$.) Further, the family~$f_{2,1}$ has infinite
automorphism group,
\[
  \Aut(f_{2,1})
  \supseteq \left\{ \SmallMatrix{1&0&0\\ 0&t&0\\ 0&0&t^2\\} : t\in\GG_m \right\} 
  \cong \GG_m.
\]
It remains to consider the families~$f_{2,0}$ and~$f_{2,3}$.

In order for~$f_{2,0}$ to be dominant, we need~$c\ne0$.
Table~\ref{table:effectofL} tells us that
\begin{equation}
  \label{eqn:f20mu}
  \mu^{\Ocal(1)}(f_{2,0},L_{k,\ell})
  \le \max\{-k, 3k+2\ell,  -k-3\ell \},
\end{equation}
with equality if $bd\ne0$. Since
\begin{align*}
  b = 0 &\Longrightarrow   \mu^{\Ocal(1)}(f_{2,0},L_{k,\ell})
  \le \max\{-k,  -k-3\ell \} \xrightarrow{(k,\ell)=(1,0)} -1, \\
  d = 0 &\Longrightarrow   \mu^{\Ocal(1)}(f_{2,0},L_{k,\ell})
  \le \max\{-k,  3k+2\ell \} \xrightarrow{(k,\ell)=(1,-2)} -1, 
\end{align*}
our assumption that~$f$ is~$N(\Gcal_4)$-semistable tells us
that we also have $bd\ne0$.  We may thus conjugate~$f_{2,0}$ by
\[
  \text{$\s=[uX,vY,wZ]$ with $(u^6,v^{12},w^3)=(bc^{-3},b^{-1}c^3d^{-6},b^{-1})$,}
\]
which with appropriate choice of roots puts~$f_{2,0}$ into the form
$f_{2,0}=[aX^2+Z^2,XY,Y^2+eXZ]$.  And since~\eqref{eqn:f20mu} is an
equality, the identity
\[
  14\cdot(-k) + 6\cdot(3k+2\ell) + 4\cdot(-k-3\ell) = 0
\]
shows that $\mu^{\Ocal(1)}(f_{2,0},L_{k,\ell})>0$ for all
$(k,\ell)\ne(0,0)$.  This completes the proof that~$f_{2,0}$ of this form
are~$N(\Gcal_4)$-stable.

In order for~$f_{2,3}$ to be dominant, we need~$ad\ne0$ and at least
one of~$b,c$ non-zero. Since $f_{2,3}^\pi=[dYZ,cX^2+bY^2,aXY]$ has the
effect of switching~$b$ and~$c$ (as well as switching~$a$ and~$d$), we
may assume without loss of generality that~$b\ne0$. Then conjugation
by~$[uX,vY,wZ]$ with
$(u^8,v^8,w^8)=(ab^{-2}c^{-1},a^{-1}b^2d^{-3},a^{-3}b^{-2}d^3)$
puts~$f_{2,3}$ in the form $f_{2,3}=[YZ,X^2+cZ^2,XY]$.
Table~\ref{table:effectofL} tells us that
\[
  \mu^{\Ocal(1)}(f_{2,3},L_{k,\ell})
  = \begin{cases}
    \max\{ 2k, -2k+\ell, 2k+3\ell, -2k-2\ell \} &\text{if $c\ne0$,} \\
    \max\{ 2k, -2k+\ell, -2k-2\ell \} &\text{if $c=0$.} \\
    \end{cases}
\]
In both cases the identity
\[
  3\cdot(2k) + 2(-2k+\ell) + (-2k-2\ell) = 0
\]
shows that $\mu^{\Ocal(1)}(f_{2,3},L_{k,\ell})>0$ for all
$(k,\ell)\ne(0,0)$.  This completes the proof that~$f_{2,3}$ of this
form are~$N(\Gcal_4)$-stable.

\par\vspace{10pt}\noindent
\textbf{Computation of $\Aut(f_{2,0})$ for $f_{2,0}=[aX^2+Z^2,XY,Y^2+eXZ]$}

To ease notation, we are going to drop the subscript on~$f$. 
We consider two maps
\[
  f=[aX^2+Z^2,XY,Y^2+eXZ] \quad\text{and}\quad f'=[a'X^2+Z^2,XY,Y^2+e'XZ]
\]
and compute
\[
  \Hom(f,f') := \bigl\{\f\in\PGL_3(K) : f^\f = f' \bigr\}.
\]

The critical locus of~$f$ is
\[
  \Crit(f) = \{ae X^3 - e X Z^2 + 2 Y^2 Z = 0 \},
\]
and the indeterminacy locus of~$f$ is
\[
  I(f) = \begin{cases}
    \emptyset &\text{if $ae\ne0$,} \\
    \bigl\{[1,0,0]\bigr\} &\text{if $a=0$,} \\
    \bigl\{[1,0,\pm\sqrt{-a}]\bigr\} &\text{if $e=0$ and $a\ne0$.} \\
  \end{cases}  
\]
In particular, if $ae\ne0$, then~$f$ is a morphism.

We assume that~$\Hom(f,f')\ne\emptyset$, and we let~$\f\in\Hom(f,f')$.

We begin by computing~$\Hom(f,f')$ in the generic case~$ae\ne0$.
Since~$\f$ sends~$I(f)$ to~$I(f')$, it follows that also~$a'e'\ne0$.
The critical locus of~$f$ is an irreducible cubic curve. Indeed, setting
\[
  E : ae X^3 - e X Z^2 + 2 Y^2 Z = 0 
  \quad\text{and}\quad \Ocal=[0,1,0],
\]
we see that~$\Ocal$ is a flex point of the cubic, so~$(E,\Ocal)$ is an
elliptic curve with group law specified by the usual rule that
distinct points~$P,Q,R\in E$ sum to~$\Ocal$ if and only if~$P,Q,R$ are
colinear.  Further, the elliptic curve~$(E,\Ocal)$ has~CM by~$\ZZ[i]$,
so $\Aut(E,\Ocal)\cong\ZZ[i]^*=\{\pm1,\pm i\}$, and since the four
maps in~$\Gcal_4$ induce automorphisms of~$(E,\Ocal)$, we see that
\[
  \Aut(E,\Ocal) = \Gcal_4.
\]
And similarly for the elliptic curve~$(E',\Ocal')$ associated to~$f'$.

We next observe that there are exactly four isomorphisms
from $(E,\Ocal)$ to~$(E',\Ocal')$, since if~$\psi_1$ and~$\psi_2$ are
any two such isomorphisms, then
$\psi_2^{-1}\circ\psi_1\in\Aut(E,\Ocal)=\Gcal_4$. Explicitly, if we
fix~$u,v\in K$ satisfying $u^4=a'/a$ and $v^2=e/e'$, then the four
elements of $\Isom\bigl((E,\Ocal),(E',\Ocal')\bigr)$ are
\begin{align*}
  \psi_0(X,Y,Z) &= [u^2X,uvY,Z], &
  \psi_1(X,Y,Z) &= [u^2X,iuvY,-Z], \\
  \psi_2(X,Y,Z) &= [u^2X,-uvY,Z], &
  \psi_3(X,Y,Z) &= [u^2X,-iuvY,-Z].
\end{align*}
Since these are diagonal maps in~$\PGL_3$, we also note that
\begin{equation}
  \label{eqn:psiaeqapsi}
  \psi\circ\a = \a\circ\psi
  \quad\text{for all $\a\in\Gcal_4$ and all
    $\psi\in\Isom\bigl((E,\Ocal),(E',\Ocal')\bigr)$.}
\end{equation}

The map~$\f\in\Hom(f,f')$ sends~$\Crit(f)$ to~$\Crit(f')$,
so~$\f(E)=E'$. In other words,~$\f|_E$ induces an isomorphism of
genus~$1$ curves $E\to E'$. (There is, however, no \emph{a priori}
reason that~$\f$ needs to send~$\Ocal$ to~$\Ocal'$.)  Standard
properties of elliptic curves~\cite[III.4.7]{MR2514094} tell us that
there is an isogeny $\psi:(E,\Ocal)\to(E',\Ocal')$ and a point~$P_0\in
E$ so that
\[
  \f(P) = \psi(P+P_0) = \psi\circ T_{P_0}(P) \quad\text{for all $P\in E$,}
\]
where $T_{P_0}:E\to E$ denotes the translation-by-$P_0$ map.
Further, since~$\f$ is invertible on all of~$\PP^2$ and since
translation by~$P_0$ is invertible on~$E$, we see that~$\psi$ is
bijective, hence~$\psi\in\Isom\bigl((E,\Ocal),(E',\Ocal')\bigr)$ is
one of the fours maps listed earlier.

We next exploit the fact that~$\f\in\PGL_3(K)$ maps lines to
lines. Thus if~$P,Q,R\in E$ are distinct points satisfying
$P+Q+R=\Ocal$, then~$\f(P),\f(Q),\f(R)$ are also colinear, so
$\f(P)+\f(Q)+\f(R)=\Ocal'$. Using the fact that~$\psi$ is a group
isomorphism, we compute
\begin{multline*}
  \Ocal' = \f(P)+\f(Q)+\f(R) = \psi(P+P_0)+\psi(Q+P_0)+\psi(R+P_0) \\
  = \psi(P+Q+R)+3\psi(P_0) = \psi(\Ocal)+3\psi(P_0)=3\psi(P_0).
\end{multline*}
Hence~$\psi(P_0)$ is a 3-torsion point of~$E'$, and since~$\psi$ is a
group isomorphism, we conclude that~$P_0$ is a 3-torsion point of~$E$.
  
We claim that~$P_0=\Ocal$. To prove this claim, we assume
that~$P_0\ne\Ocal$ and derive a contradiction.  For an
arbitrary~$\a\in\Gcal_4\subseteq\Aut(f)$, we note that the
composition~$\a\circ\f^{-1}\circ\a^{-1}\circ\f$ is in~$\Aut(f)$. On the
other hand, we can write this composition explicitly as
\begin{align*}
  \a\circ\f^{-1}\circ\a^{-1}\circ\f
  &= \a \circ (\psi\circ T_{P_0})^{-1} \circ \a^{-1} \circ (\psi\circ T_{P_0}) \\
  &= \a \circ T_{-P_0} \circ \psi^{-1} \circ \a^{-1} \circ \psi\circ T_{P_0} \\  
  &= \a \circ T_{-P_0}  \circ \a^{-1} \circ \psi^{-1} \circ \psi\circ T_{P_0}
      \quad\text{from \eqref{eqn:psiaeqapsi},} \\
  &= T_{-\a(P_0)}  \circ  T_{P_0} \\
  &= T_{P_0-\a(P_0)}.
\end{align*}
Our assumption that~$P_0$ is a non-trivial $3$-torsion point implies
that $P_0\ne\a(P_0)$ for all $\a\in\Gcal_4\setminus\{1\}$, since for
such~$\a$, the kernel of~$\a-1$ consists of $2$-torsion points.
Hence the set
\[
  \bigl\{ P_0-\a(P_0) : \a\in\Gcal_4 \bigr\}
\]
contains four distinct elements of~$E[3]$; in particular, it contains
generators of~$E[3]$. We saw above that all of the
translations~$T_{P_0-\a(P_0)}$ are in~$\Aut(f)$, so using the fact
that~$\Aut(f)$ is a group, we have proven that
\[
  \Aut(f) \supset \bigl\{ T_Q : Q \in E[3] \bigr\}.
\]
Thus~$\Aut(f)$ contains a subgroup of type~$C_3\times C_3$. This and
the fact that~$f$ is a degree~$2$ morphism contradicts
Proposition~\ref{proposition:fCpCp},
which concludes the proof that~$P_0=\Ocal$.

We now know that every~$\f\in\Hom(f,f')$ has the form~$\f=\psi$ for
some~$\psi\in\Isom\bigl((E,\Ocal),(E',\Ocal')\bigr)$,
i.e.,~$\Hom(f,f')$ consists of the four maps listed earlier, so there
is an integer~$m$ such that
\[
  \f(X,Y,Z) = [u^2X,i^m uvY,(-1)^mZ],
\]
where we recall that~$u$ and~$v$ satisfy $u^4=a'/a$ and $v^2=e/e'$.
We compute
\begin{align*}
  f^\f = \f^{-1} f\circ \f(X,Y,Z) 
  &= [u^2 aX^2 + u^{-2} Z^2, u^2 XY,  u^2 v^2 Y^2 +  u^2 eXZ] \\
  &= \left[a' X^2 +  Z^2, \frac{a'}{a} XY,  \frac{a'e}{ae'} (Y^2 + e' XZ)\right].
\end{align*}
Comparing this to $f'=[a'X^2+Z^2,XY,Y^2+e'XZ]$, we see that $f^\f=f'$
if and only if $a=a'$ and $e=e'$, i.e., if and only if~$f'=f$.  This
completes the proof that
\[
\Hom(f,f') = \begin{cases}
  \emptyset &\text{if $f\ne f'$,}\\
  \Gcal_4 &\text{if $f=f'$,}\\
\end{cases}
\]
which completes proof of~(a) in the case that $ae\ne0$.

We next consider the case of maps with $a=0$, i.e., maps of the form
\[
  f = [Z^2,XY,Y^2+eXZ].
\]
These maps satisfy
\[
  I(f) = \bigl\{[1,0,0]\bigr\}
  \quad\text{and}\quad
  \Crit(f) = \{Z=0\} \cup \{eXZ=2Y^2\}.
\]
Letting~$f'=[Z^2,XY,Y^2+e'XZ]$ be another such map, we conclude that
any~$\f\in\Hom(f,f')$ fixes the point~$[1,0,0]$ and stabilizes the
line~$Z=0$. (This is true regardless of whether~$e=0$ and/or $e'=0$,
since~$\f$ preserves the multiplicities of components of~$\Crit(f)$.)
Thus~$\f$ has the form
$\f=\SmallMatrix{1&\a&\b\\ 0&\g&\d\\ 0&0&\e\\}$. Equating
\begin{align*}
  f\circ\f
  &= [ \e^2 Z^2, \g XY + \d XZ + \a\g Y^2 + (\a\d+\b\g) YZ + \b\d Z^2, \, *], \\
  \f\circ f'
  &= [Z^2 + \a X Y + \b Y^2 + \b e' X Z, \g X Y + \d Y^2 + \d e' X Z, \, * ],
\end{align*}
we see from the $XY$ and $Y^2$ terms in the first coordinate that $\a=\b=0$,
and then the $Y^2$ term in the second coordinate gives $\d=0$.
Hence~$\f$ is a diagonal matrix, and we have
\[
  [Z^2,XY,Y^2+e'XZ]
  = f' = f^\f
  = [\e^2 Z^2, XY, \g^2 \e^{-1} Y^2 + e XZ].
\]
Therefore~$\f\in\Hom(f,f')$ if and only if $e=e'$ and $\e^2=1$ and
$\g^2=\e$.  So if $f=f'$, then we get the four maps in~$\Gcal_4$, and if
$f\ne f'$, then~$f$ and~$f'$ are not conjugate.

It remains to consider the case of maps with $a\ne0$ and $e=0$, i.e.,
maps of the form
\[
  f = [aX^2 + Z^2, XY, Y^2]\quad\text{with $a\ne0$.}
\]
These maps satisfy
\[
  \Crit(f) = \{Z=0\}\cup\{Y^2=0\},
\]
i.e., the critical locus of~$f$ consists of two lines, one with
multiplicity~$1$ and one with multiplicity~$2$.  Letting~$f'=[a'X^2 +
  Z^2, XY, Y^2]$ be another such map, we conclude that
any~$\f\in\Hom(f,f')$ stabilizes the lines~$\{Z=0\}$ and $\{Y=0\}$, so
$\f$ has the form $\f=\SmallMatrix{\a&\b&\g\\ 0&\d&0\\ 0&0&1}$.
Equating the middle coordinates of
\begin{align*}
  f\circ\f
  &= [*\,, \a\d XY + \b\d Y^2 + \g\d YZ, \,*], \\
  \f\circ f'
  &= [*\,, \d XY,\, *],  
\end{align*}  
and using the fact that~$\d\ne0$ (since~$\f$ is invertible), we see
that $\b=\g=0$.  Hence~$\f$ is a diagonal matrix, and we have
\[
  [a' X^2 + Z^2, XY, Y^2]
  = f' = f^\f
  = [\a^2 a X^2 +  Z^2, \a^2 XY, \a\d^2 Y^2].
\]
Therefore~$\f\in\Hom(f,f')$ if and only if $a=a'$ and $\a^2=1$ and
$\d^2=\a^{-1}$.  So if $f=f'$, then we get the four maps in~$\Gcal_4$, and if
$f\ne f'$, then~$f$ and~$f'$ are not conjugate.

\par\vspace{10pt}\noindent
\textbf{Computation of $\Aut(f_{2,3})$ for $f_{2,3}=[YZ,X^2+cZ^2,XY]$}

To ease notation, we again drop the subscript on~$f$. We
consider two such maps
\[
  \text{$f=[YZ,X^2+cZ^2,XY]$ \quad\text{and}\quad $f'=[YZ,X^2+c'Z^2,XY]$,}
\]
and we compute $\Hom(f,f')$, where we suppose
that~$\Hom(f,f')\ne\emptyset$. Let $\f\in\Hom(f,f')$, and let
$\g=\sqrt{-c}$ and $\g'=\sqrt{-c'}$.

The indeterminacy locus of~$f$ consists of three points and the
critical locus of~$f$ consists of three lines (with multiplicity
if~$c=0$),
\begin{align*}
  I(f) &= \bigl\{ [0,1,0], [\pm\g,0,1] \bigr\}, \\
  \Crit(f) &= \{Y=0\} \cup \{X=\g Z\} \cup \{X = -\g Z\}.
\end{align*}
The orbit portrait of~$\Crit(f)$ is
\[
  \{X=\pm\g Z\} \xrightarrow{\;f\;} [1,0,\pm\g] \in \{Y=0\}
   \xrightarrow{\;f\;} [0,1,0] \in I(f),
\]
and similarly for~$f'$.  Suppose first that $c\ne\pm1$. Then
$[1,0,\pm\g]\notin I(f)$, so the fact that the map $\f\in\Hom(f,f')$
sends the orbit portrait of~$\Crit(f)$ to the orbit portrait
of~$\Crit(f')$ implies that~$\f$ fixes the line~$\{Y=0\}$ and the
point~$[0,1,0]$.  This means that~$\f$ has the form
\[
  \f=\SmallMatrix{s&0&t\\ 0&1&0\\ u&0&v\\}\in\PGL_3(K).
\]
We are assuming that~$f^\f=f'$.  We start by comparing the first and
third coordinates of~$f\circ\f$ and~$\f\circ f'$,
\begin{align*}
  f\circ\f &= [uXY+vYZ,\, *\,, sXY+tYZ],\\
  \f\circ f' &= [tXY+sYZ,\, *\,, vXY+uYZ].
\end{align*}
Thus there is an~$\e\in K^*$ such that $(u,v,s,t)=(\e t,\e s,\e v,\e
u)$, and this in turn implies that $u=\e t=\e^2u$ and $v=\e s=\e^2
v$. The invertibility of~$\f$ implies that~$u$ and~$v$ are not
both~$0$, so~$\e=\pm1$. Further, setting $u=\e t$ and $v=\e s$,
the middle coordinates of~$f\circ\f$ and~$\f\circ f'$ look like
\begin{align*}
  f\circ\f &= [*,\, (s^2+ct^2)X^2 + 2st(1+c)XZ + (t^2+cs^2)Z^2,\, *], \\
  \f\circ f' &= [*,\, X^2+c'Y^2,\, *].  
\end{align*}
Hence we must have
\[
  s^2 + ct^2 = \e,\quad st(1+c)=0,\quad t^2+cs^2 = \e c'.
\]
This leads to two cases (since we are assuming for the present that $c\ne-1$):
\begin{align*}
  s=0 &\quad\Longrightarrow\quad \e c^{-1}=t^2=\e c' \quad\Longrightarrow\quad cc'=1, \\
  t=0 &\quad\Longrightarrow\quad s^2=\e \quad\Longrightarrow\quad c=c'.
\end{align*}

We see that if~$f'\ne f$, i.e., $c'\ne c$, then we must have
$c'=c^{-1}$ and $s=v=0$, and in this case we find that the permutation
$[Z,Y,Z]\in N(\Gcal_4)$ is in~$\Hom(f,f')$.

If $c'=c$, i.e., we are computing~$\Aut(f)$, then our assumption that
$c^2\ne1$ means that we must have $t=u=0$ and $s^2=\e=\pm1$. This proves
that
\[
  \Aut(f) = \left\{ \SmallMatrix{s&0&0\\ 0&1&0\\ 0&0&\e s} \in \PGL_3(K): 
   \e=\pm1~\text{and}~s=\pm\sqrt{\e} \right\} = \Gcal_4.
\]

Next we consider the case that~$c=1$. This gives the map that is
labeled~$f_{4.2}$ in Example~\ref{example:notinj}.  It is shown in
that example that~$f_{4.2}$ is $\PGL_3$-conjugate to the map
$f_{6.1}:=[YZ,XZ,XY]$. We proved in
Proposition~\ref{proposition:classifyC2C2}(c) that
$\Aut(f_{6.1})=\Scal_3\Gcal_{2,2}\cong S_4$, and hence we find that
$\Aut(f_{4.2})\cong S_4$. Explicitly,~$\Aut(f_{4.2})$ is the subgroup
of~$\PGL_3$ given by conjugating $\Scal_3\Gcal_{2,2}$ by the inverse
of the map~$\b$ given in Example~\ref{example:notinj}.
  
Finally, if $c=-1$, then a similar calculation shows that~$s$ and~$t$
need only satisfy the single relation~$s^2-t^2=\e$, so
\[
  \Aut(f) \supseteq
  \left\{ \SmallMatrix{s&0&t\\ 0&1&0\\ \e t&0&\e s} \in \PGL_3(K):
     \begin{array}{@{}l@{}}
       \e=\pm1~\text{and}\\ s^2-t^2=\e\\ \end{array}
     \right\} \cong \GG_m\rtimes C_2.
\]
\end{proof}

\begin{proof}[Proof of Proposition~$\ref{proposition:C4notpossible}$]
Taking $\z=i$ to be a primitive 4'th root of unity, we see that the
map~$[X,iY,Z]$ is~$\t_0$. Then Table~\ref{table:effectoftm} with
$m=0$ and entries reduced modulo~$4$ is
{\small
\[
\begin{array}{|c||c|c|c|c|c|c|} \hline
  & X^2 & Y^2 & Z^2 & XY & XZ & YZ \\ \hline\hline
  {\text{$X$-coord}}
  & 0 & 2 & 0 & 1 & 0 & 1 \\ \hline
  {\text{$Y$-coord}}
  & 3 & 1 & 3 & 0 & 3 & 0 \\ \hline
  {\text{$Z$-coord}}
  & 0 & 2 & 0 & 1 & 0 & 1 \\ \hline
\end{array}
\]
}%
Hence~$\t_2\in\Aut(f)$ leads to the following four families of maps:
\begin{align*}
  f_{0,0} &:= [aX^2+bZ^2+cXZ,dXY+eYZ,fX^2+gZ^2+hXZ], \\
  f_{0,1} &:= [aXY+bYZ,cY^2,dXY+eYZ], \\
  f_{0,2} &:= [aY^,0,bY^2], \\
  f_{0,3} &:= [0,aX^2+bZ^2+cXZ].
\end{align*}
For each family we use Table~\ref{table:effectofL} to compute
\begin{align*}
  \mu^{\Ocal(1)}(f_{0,0},L_{k,\ell})
  & \le \max\{-k,3k+2\ell,k+\ell,-3k-\ell\} \\
  &\omit\hfill$\xrightarrow{(k,\ell)=(1,-2)} -1$, \\
  \mu^{\Ocal(1)}(f_{0,1},L_{k,\ell})
  & \le \max\{-\ell,2k,-\ell,-2k-2\ell,-k\}
  \xrightarrow{(k,\ell)=(0,1)} -1,\\
  \mu^{\Ocal(1)}(f_{0,2},L_{k,\ell})
  & \le \max\{k-2\ell,-k-3\ell\}
  \xrightarrow{(k,\ell)=(0,1)} -3,\\
  \mu^{\Ocal(1)}(f_{0,3},L_{k,\ell})
  & \le \max\{-3k-\ell,k+\ell,-k\}
  \xrightarrow{(k,\ell)=(1,-2)} -1.
\end{align*}
This shows that all of these maps are $\Dcal$-unstable.
\end{proof}

We now investigate more closely one of the families of maps appearing
in Proposition~\ref{proposition:autc4types}.

\begin{proposition}
\label{proposition:c4f2autgm}
Let $f=[YZ,X^2+cZ^2,XY]$ be the map from
Proposition~\textup{\ref{proposition:autc4types}(b)}.  Let $R(X,Z)=X^2+cZ^2$
and $S(X,Z)=cX^2+Z^2$. Then the iterates of~$f$ are given by the
explicit formulas
\begin{align*}
  f^{2k}(X,Y,Z) &= \bigl[R(X,Y)^k X, S(X,Y)^k Y, R(X,Y)^k Z\bigr], \\
  f^{2k+1}(X,Y,Z) &= \bigl[S(X,Y)^k YZ, R(X,Y)^{k+1}, S(X,Y)^k XY\bigr].
\end{align*}
In particular, if $c\ne\pm1$, then $\deg(f^n)=2n+1$, while if
$c=\pm1$, then $f^{2k}=[X,(-1)^kY,Z]$.

Let $p(X,Y,Z)=[X,Z]$. Then there is a commutative diagram
\[
\begin{CD}
  \PP^2 @>f>> \PP^2 \\
  @VVpV @VVpV \\
  \PP^1 @>[U,V]\to[V,U]>> \PP^1\\
\end{CD}
\]
The second iterate $f^2=[X^3+cXZ^2,cX^2Y+YZ^2,X^2Z+cZ^3]$ has infinte
automorphism group,
\[
  \Aut(f^2) \supset \left\{\SmallMatrix{1&0&0\\0&t&0\\0&0&1\\} : t\in\GG_m\right\}.
\]
\end{proposition}
\begin{proof}
We need to prove that
\begin{align*}
  f^{2k}(X,Y,Z) &= \bigl[R^k X, S^k Y, R^k Z\bigr], \\
  f^{2k+1}(X,Y,Z) &= \bigl[S^k YZ, R^{k+1}, S^k XY\bigr].
\end{align*}
The proof is by induction on~$k$. The formulas are visibly correct for
$k=0$. Assuming that the formula for~$f^{2k}$ is correct, we compute
\begin{align*}
  f^{2k+1}(X,Y,Z) &= f\bigl(f^{2k}(X,Y,Z)\bigr) \\
  &= f(R^kX,S^kY,R^kZ) \\
  &= [R^kS^kYZ, R^{2k}X^2+cR^{2k}Z^2, R^kS^k XY] \\
  &= [S^k YZ, R^{k+1}, S^k XY], 
\end{align*}  
which shows that the formula for~$f^{2k+1}$ is correct.
Similarly, assuming that the formula for~$f^{2k+1}$ is correct, we compute
\begin{align*}
  f^{2k+2}(X,Y,Z) &= f\bigl(f^{2k+1}(X,Y,Z)\bigr) \\
  &= f(S^k YZ, R^{k+1}, S^k XY) \\
  &= [R^{k+1}S^kXY, S^{2k}Y^2Z^2+cS^{2k}X^2Y^2, R^{k+1}S^kYZ ] \\
  &= [R^{k+1}X, S^{k+1}Y , R^{k+1}Z ],
\end{align*}
which shows that the formula for~$f^{2k+2}$ is correct. This completes the proof
of the formulas for the iterates of~$f$.

We now ask when the coordinates of~$f^n$ have a common factor. It is clear that
neither~$X$ nor~$Y$ nor~$Z$ is a common factor, so any non-trivial common factor must
be a common factor of~$R$ and~$S$. Since
\[
  \Resultant(R,S)=\Resultant(X^2+cZ^2,cX^2+Z^2) = (c^2-1)^2,
\]
there is no common factor if~$c\ne\pm1$. Hence if~$c\ne\pm1$, then~$\deg(f^n)=n+1$.
On the other hand, if~$c=\pm1$, then $S=\pm R$, so $f^{2k}=[X,(-1)^kY,Z]$.

Finally, the commutativity of the diagram and verification that the
indicated matrices are in~$\Aut(f^2)$ are trivial calculations. This
completes the proof of Proposition~\ref{proposition:c4f2autgm}.
\end{proof}

\section{Maps with an Automorphism of Order $3$}
\label{section:autC3}
In this section we classify maps in~$\Rat_2^2$ that admit an
automorphism of order~$3$. The classification that we give in
Table~\ref{table:autord3} arises naturally during the proof, but we
note that later in Section~\ref{section:endproof} we will use somewhat
different normal forms in order to create the families as described in
Table~\ref{table:maintable}.

\begin{proposition}
\label{proposition:autc3types}
Let $K$ be an algebraically closed field of characteristic~$0$, let
$\Gcal_3$ be the group described in Theorem~$\ref{theorem:main1}$, and
let $f\in\Rat_2^2(\Gcal_3)^\ss$ be a dominant map of degree~$2$ with
finite automorphism group.  Further define a subgroup
$\Gcal_{3,2}\subset\PGL_3(K)$ by
\[
  \Gcal_{3,2} := \left\<\SmallMatrix{1&0&0\\0&\z&0\\0&0&\smash[t]{\z^2}\\},
  \SmallMatrix{1&0&0\\0&0&1\\0&1&0\\} \right\> \cong S_3.
\]
Then~$f$ is $N(\Gcal_3)$-conjugate to one of the maps in
Table~$\ref{table:autord3}$, where the penultimate column indicates if
the given maps are morphisms and the last column describes when two
maps of the given form are $N(\Gcal_3)$-conjugate to one
another.\footnote{In some cases  we have given only the
  isomorphism class of~$\Aut(f)$. But during the proof of the
  proposition, we give an explicit description of~$\Aut(f)$ as a
  subgroup of~$\PGL_3(K)$.}
Further, maps of Type~$\boldsymbol C_3(n)$ and~$\boldsymbol C_3(n')$
for $n\ne n'$ are not~$N(\Gcal_3)$-conjugate.
\end{proposition}

\begin{proposition}
\label{proposition:C3notpossible}
Let~$K$ be an algebraically closed field of characteristic~$0$,
let~$\z$ be a primitive cube root of unity, and let $f\in\Rat_2^2$
be a dominant rational map satisfying
\[
  \Aut(f) \supseteq \left\< \SmallMatrix{1&0&0\\0&\z&0\\0&0&1\\} \right\>.
\]
Then~$f$ is~$\Dcal$-unstable.
\end{proposition}

\begin{table}
{\small
\[
\begin{array}{|c|c|c|c|c|c|} \hline
  \hspace{.5em}f\hspace{.5em} & \text{Coeffs} & \Aut(f) & \text{Mor?} & f' \sim f \\ \hline\hline
  \multicolumn{5}{|l|}{\boldsymbol{C_3(1)}:\;f_b = [X^2+bYZ,Z^2,Y^2],\quad b\in K}  \\ \cline{2-5}
  &  & \Gcal_{3,2} & \text{Yes} & b'=b \\ \hline
  \multicolumn{5}{|l|}{\boldsymbol{C_3(2)}:\;f_{a,b} = [aX^2+YZ,XY,Y^2+gXZ],\quad a,g\in K~\text{not both $0$}}  \\ \cline{2-5}
  &  & \Gcal_3 & \text{No} & - \\ \hline\hline
  \multicolumn{5}{|l|}{\boldsymbol{C_3(3)}:\;f_b = [bYZ,Z^2+XY,Y^2],\quad b\in K^*}  \\ \cline{2-5}
  &   & \Gcal_3  & \text{No} & - \\ \hline\hline
  \multicolumn{5}{|l|}{\boldsymbol{C_3(4)}:\;f_{b,g} = [bYZ,Z^2+XY,Y^2+gXZ],\quad b,g\in K^*}  \\ \cline{2-5}
  & g\ne1 & \Gcal_3 & \text{No} & (b',g')=(bg,g^{-1}) \\
  & g = 1 & \Gcal_{3,2} & \text{No} & - \\ \hline\hline
  \multicolumn{5}{|l|}{\boldsymbol{C_3(5)}:\;f_{a,b} = [aX^2-aYZ,Z^2-XY,bY^2-bXZ],\quad a,b\in K^*}  \\ \cline{2-5}
  & a\ne1~\text{and}~b\ne 1 & \cong C_3 & \text{No} & (a',b')=(b,a) \\
  & \text{exactly one of $a,b=1$}&\cong C_3\rtimes C_2  & \text{No} &  (a',b')=(b,a) \\
  & a=b=1 &\cong S_4  & \text{No} & - \\   \hline\hline
  \multicolumn{5}{|l|}{\boldsymbol{C_3(6)}:\;f_{a,b} := [aX^2+bYZ,Z^2+XY,Y^2],\quad a,b\in K^*}  \\ \cline{2-5}
  &  & \Gcal_3 & \text{Yes} & - \\ \hline\hline
  \multicolumn{5}{|l|}{\boldsymbol{C_3(7)}:\;f_{b,d,g} = [X^2+bYZ,Z^2+dXY,Y^2+gXZ],} \\
  \multicolumn{5}{|r|}{b,d,g\in K^*,\; bdg\ne-1,8 }  \\ \cline{2-5}
  & g\ne d& \Gcal_3 & \text{Yes} & -\\
  & g = d&  \Gcal_{3,2} & \text{Yes} & -\\   \hline\hline
  \multicolumn{5}{|l|}{\boldsymbol{C_3(8)}:\;f_{c,e} = [X^2+2YZ,cZ^2+2cXY,eY^2+2eXZ],\quad c,e\in K^*}  \\ \cline{2-5}
  & (c,e)\ne(\z_3,\z_3^2)~\text{and}~(\z_3^2,\z_3) & \Gcal_3 & \text{Yes} & (c',e')=(e,c) \\
  & (c,e)=(\z_3,\z_3^2)~\text{or}~(\z_3^2,\z_3) & \cong C_7\rtimes C_3 & \text{Yes} & (c',e')=(e,c) \\ \hline
\end{array}
\]
}
\caption{Dominant degree $2$ maps $f\in\Rat_2^2(\Gcal_3)^\ss$}
\label{table:autord3}
\end{table}

\begin{proof}[Proof of Proposition $\ref{proposition:autc3types}$]
During the proof we will frequently use the fact
that~$N(\Gcal_3)=\Scal_3\Dcal$; see Lemma~\ref{lemma:subgpnlzr}.
The map $[X,\z Y,\z^2 Z]$ generating~$\Gcal_3$ is the map defined by the matrix~$\t_2$ in
Section~\ref{section:diagstability}. Using Table~\ref{table:effectoftm}
with $m=2$ and entries reduced modulo~$3$, we find that
{\small
\[
\begin{array}{|c||c|c|c|c|c|c|} \hline
  & X^2 & Y^2 & Z^2 & XY & XZ & YZ \\ \hline\hline
  {\text{$X$-coord}}
  & 0 & 2 & 1 & 1 & 2 & 0 \\ \hline
  {\text{$Y$-coord}}
  & 2 & 1 & 0 & 0 & 1 & 2 \\ \hline
  {\text{$Z$-coord}}
  & 1 & 0 & 2 & 2 & 0 & 1 \\ \hline
\end{array}
\]
}%
Hence assuming that~$\t_2\in\Aut(f)$ leads to the following three families of maps:
\begin{align*}
  f &:= [aX^2+bYZ,cZ^2+dXY,eY^2+gXZ], \\
  f' &:= [aZ^2+bXY,cY^2+dXZ,eX^2+gYZ], \\
  f'' &:= [aY^2+bXZ,cX^2+dXZ,eZ^2+gXY].
\end{align*}

Conjugating by the cyclic permutation~$\pi(X,Y,Z)=[Y,Z,X]\in
N(\Gcal_3)$ gives
\begin{align*}
  f^\pi &= [eZ^2+gXY,aY^2+bXZ,cX^2+dYZ],\\
  f^{\pi^2} &= [cY^2+dXZ,eX^2+gYZ,aZ^2+bXY],
\end{align*}
which shows that our three families are $N(\Gcal_3)$-conjugates. It
thus suffices to analyze one of them, so we concentrate on
\begin{equation}
  \label{eqn:fXYZaX2bYZ}
  f(X,Y,Z) = [aX^2+bYZ,cZ^2+dXY,eY^2+gXZ].
\end{equation}

Using Table~\ref{table:effectofL}, we see that
\[
\mu^{\Ocal(1)}(f,L_{k,\ell})
= \max\bigl\{
\overbrace{-k}^{a,d,g\ne0},
\overbrace{2k}^{b\ne0},
\overbrace{2k+3\ell}^{c\ne0},
\overbrace{-k-3\ell}^{e\ne0}
\bigr\},
\]
where $-k$ appears in the max if one or more of $a,d,g$ is non-zero.

In Table~\ref{table:ssscondC3}, maps marked as being semi-stable are not
stable.  Also, in the column marked~$a,d,g$, the symbol~$\ne0$ means
that at least one of~$a,d,g$ is non-zero, while~$0$ means that all
three values are~$0$.

\begin{table}[ht]
\[
\begin{array}{|c||c|c|c|c|c|c|}\hline
  \text{Case} & a,d,g & b & c & e & \text{stability} \\ \hline\hline
  1 & \ne0 & * & \ne0 & \ne0 & \text{stable} \\ \hline
  2 & \ne0 & \ne0 & \ne0 & 0 & \text{semi-stable} \\ \hline
  3 & \ne0 & \ne0 & 0 & \ne0 & \text{semi-stable} \\ \hline
  4 & \ne0 & \ne0 & 0 & 0 & \text{semi-stable} \\ \hline  
  5 & 0  & * & * & * & \text{unstable} \\ \hline
\end{array}
\]
\caption{Semi-stability and stability conditions}
\label{table:ssscondC3}
\end{table}

To justify our assertion that the maps in Case~1 are stable, we use the identity
\[
  (-k) + (2k+3\ell) + (-k-3\ell) = 0,
\]
which shows that
\[
  \mu^{\Ocal(1)}(f,L_{k,\ell})
  =\max\{-k,2k+3\ell,-k-3\ell\}>0
  \quad\text{for all $(k,\ell)\ne(0,0)$.}
\]
For Cases~2,~3 and~4 it is clear that
$\mu^{\Ocal(1)}\ge0$ due to the~$-k$ and~$2k$ terms in the max, but
taking $k=0$ and $\ell=\pm1$ gives a non-zero $(k,\ell)$
pair with $\mu^{\Ocal(1)}(f,L_{k,\ell})=0$. Hence these cases give maps that
are semi-stable, but not stable. Finally, in Case~5 we have
\[
  \mu^{\Ocal(1)}(f,L_{k,\ell})
  \le \max\{2k,2k+3\ell,-k-3\ell\}
  \xrightarrow{(k,\ell)=(-2,1)} -1,
\]
which shows that these maps are unstable.  (We also remark that the
Case~5 map $f = [bYZ,cZ^2,eY^2]$ is not dominant, since its image is
contained in the conic $b^2YZ=ceX^2$.)

Conjugating~\eqref{eqn:fXYZaX2bYZ} by~$\s=[uX,vY,wZ]\in N(\Gcal_3)$
yields the twist
\begin{equation}
  \label{eqn:f3twist}
  f^\s = \left[u aX^2+\frac{vw}{u} bYZ,\frac{w^2}{uv} cZ^2+u dXY,
    \frac{v^2}{uw} eY^2+u gXZ\right].
\end{equation}

If $e=0$ and $c\ne0$, we can use the permutation $[X,Z,Y]\in
N(\Gcal_3)$ that swaps $c$ and $e$.  The remainder of the proof is a
case-by-case analysis that depends on properties of the coefficients.

\par\noindent\framebox{$\boldsymbol{c=e=0}$}
\par
The map
\[
  f(X,Y,Z) = [aX^2+bYZ,dXY,gXZ]
\]
has the property that $[X,tY,t^{-1}Z]\in\Aut(f)$ for every~$t$, so~$\Aut(f)$
contains a copy of~$\GG_m$.

\par\noindent\framebox{$\boldsymbol{c=0}$ and $\boldsymbol{e\ne 0}$}
\par
We have
\[
  f(X,Y,Z) = [aX^2+bYZ,dXY,eY^2+gXZ].
\]
The dominance of $f$ implies that $d\ne0$,
and the semi-stability of~$f$ implies that~$b\ne0$.
Using the fact that $bde\ne0$, we see from the twisting
formula~\eqref{eqn:f3twist} that an appropriate twist
lets us take $b=d=e=1$.
So
\[
  f(X,Y,Z) = [aX^2+YZ,XY,Y^2+gXZ].
\]
The indeterminacy locus is
\[
  I(f)=\begin{cases}
  \bigl\{[0,0,1]\bigr\}&\text{if $a\ne0$,} \\
  \bigl\{[0,0,1],[1,0,0]\bigr\}&\text{if $a=0$.} \\
  \end{cases}
\]

Suppose that $\f\in\Hom(f,f')$ fixes $[0,0,1]$, which is forced if~$a\ne0$,
and is one of two possibilities if $a=0$.
Thus $\f=\SmallMatrix{\a&\b&0\\ \g&\d&0\\ \l&\mu&1\\}$.
Comparing
\begin{align*}
  \f\circ f' &= [*,\, \g(a'X^2+YZ) + \d XY,\,*],\\
  f\circ\f &= [*,\, (\a X+\b Y)(\g X + \d Y),\, *],
\end{align*}
we see that $\g=0$ because $f\circ\f$ has no $YZ$ term, and that
$\b\d=0$ because $\f\circ f'$ has no $Y^2$ term.  But $\g=\d=0$
contradicts the invertibility of~$\f$, so $\b=0$.  Hence
$\f=\SmallMatrix{\a&0&0\\ 0&\d&0\\ \l&\mu&1\\}$. Next we look at
the first coordinates,
\begin{align*}
  \f\circ f' &= [\a(a'X^2+YZ),\,*,\,*],\\
  f\circ\f &= [a(\a X)^2 + (\d Y)(\l X + \mu Y + Z),\,*,\,*].
\end{align*}
Since $\d\ne0$, the lack of an $XY$ term gives $\l=0$ and the lack of
a $Y^2$ term gives $\mu=0$. Hence~$\f$ is diagonal. Then
\[
  f' = f^\f = [\a^2 aX^2 + \d YZ, \a^2 XY, \a\d^2 Y^2 + \a^2 g XZ].
\]
Normalizing on the $XY$ term, this formula holds if and only if
\[
  (a,\a^{-2}\d,\a^{-1}\d^2,g) = (a',1,1,g').
\]
So if and only if $f'=f$ and $\d=\a^2$ and $\d^2=\a$. Hence~$\a$ is
a cube root of unity and $\d=\a^2$, which gives the copy of~$\Gcal_3$
that we already know is in~$\Aut(f)$.

Next suppose that $a=0$ and that $\f\in\Hom(f,f')$ swaps $[0,0,1]$ and
$[1,0,0]$. Then $f=[YZ,XY,Y^2+gXZ]$ and
$\f=\SmallMatrix{0&\a&\b\\0&1&0\\\g&\d&0\\}$, and we have
\[
\f\circ f' = [*,\, XY,\, *],
\quad
f\circ\f = [*,\, (\a Y+\b Z)Y,\, *].
\]
This gives a contradiction, so even in the case that~$a=0$, we obtain
no new elements. 

\par\noindent\framebox{$\boldsymbol{ce\ne0}$}
\par
We see from the twisting
formula~\eqref{eqn:f3twist} that an appropriate twist
lets us take $c=e=1$, so
\[
  f(X,Y,Z) = [aX^2+bYZ,Z^2+dXY,Y^2+gXZ].
\]
Further, the semi-stablity of~$f$ tells us that at least one
of~$a,d,g$ is non-zero. Further, the permutation~$\pi=[X,Z,Y]\in N(\Gcal_3)$
conjugates~$f$ to
\[
  f^\pi = [aX^2+bYZ,Z^2+gXY,Y^2+dXZ],
\]
i.e., it swaps~$d$ and~$g$. This gives two subcases: (1) $d\ne0$; (2) $d=g=0$.

\par\noindent\framebox{$\boldsymbol{ce\ne0}$ and $\boldsymbol{d=g=0}$}
\par
Semi-stablity of~$f$  tells us that $a\ne0$, and then
a twist~\eqref{eqn:f3twist} lets us set~$a=1$, so
\[
  f(X,Y,Z) = [X^2+bYZ,Z^2,Y^2].
\]
We observe that~$f$ is a morphism with critical set
\[
  \Crit(f) = \{XYZ=0\}.
\]
We also observe that the permutation $\pi=[X,Z,Y]\in\Aut(f)$.  Let
$\f\in\Hom(f,f')$, so~$\f$ permutes the three lines in~$\Crit(f)$.
If~$\f$ swaps the lines~$Y=0$ and~$Z=0$, then~$\pi\f$ fixes them, so
it suffices to analyze the maps~$\f$ that fix the lines~$Y=0$
and~$Z=0$ and the maps~$\f$ that satisfy
$\{X=0\}\to\{Y=0\}\to\{Z=0\}\to\{X=0\}$.

The maps fixing the lines~$Y=0$ and~$Z=0$ have the
form~$\f=\SmallMatrix{\a&\b&\g\\0&\d&0\\0&0&1\\}$.
Then
\begin{align*}
  \f\circ f' &= [ \a(X^2+b'YZ) + \b Z^2 + Y^2, \d Z^2, Y^2], \\
  f\circ \f  &= [ (\a X+\b Y+\g Z)^2 + b\d YZ, Z^2, \d^2 Y^2 ].
\end{align*}
Looking at the $XY$ and $XZ$ terms in the first coordinate tells
us that $\a\b=\a\g=0$. The invertibility of~$\f$ implies that~$\a\ne0$,
so $\b=\g=0$, i.e.,~$\f$ is diagonal. Then
\[
  f' = f^\f = [\a\d X^2 + \a^{-1}\d^2 bYZ, Z^2, \d^3 Y^2],
\]
so we need~$\d^3=\a\d=1$ and $b' = \a^{-1}\d^2 b$. But the first
conditions imply that $\a^{-1}\d^2=\d^3=1$, so~$b'=b$, and we recover
the three maps in~$\Gcal_3$ that we already knew were
in~$\Aut(f)$. Composing with~$\pi$ gives a copy of~$\Gcal_{3,2}\cong
S_3$ sitting in~$\Aut(f)$.

Next suppose that~$\f$ cyclically permutes the lines~$XYZ=0$ as
described earlier.  Then
$\f =\SmallMatrix{0&0&\a\\ \b&0&0\\ 0&\g&0\\}$, and
\[
  f^\f = [\g^3 Y^2, \b^3 X^2, \, * ].
\]
This cannot possibly equal~$f'=[X^2+bYZ,\,*,\,*]$. 

\par\noindent\framebox{$\boldsymbol{dce\ne0}$}
\par
In this case we can twist using~\eqref{eqn:f3twist} to make $d=1$, so
\[
  f(X,Y,Z) = [aX^2+bYZ,Z^2+XY,Y^2+gXZ],
\]
and the assumed dominance of~$f$ tells us that~$a$ and~$b$ are not
both~$0$.

We first determinate when~$f$ is a morphism. Suppose that $[x,y,z]\in I(f)$.
Then
\begin{equation}
  \label{eqn:ax2byz}
  ax^2 = -byz,\quad z^2=-xy,\quad y^2=-gxz,
\end{equation}
and multiplying these three equations yields $a(xyz)^2=-bg(xyz)^2$.
So either $xyz=0$ or $a=-bg$. We start with the former and observe:
\begin{align*}
  x = 0 &\Longrightarrow z=0 \Longrightarrow y=0 \quad \rightarrow\leftarrow, \\
  y = 0 & \Longrightarrow  z=0  \Longrightarrow ax^2=0  \Longrightarrow a=0, \\
  z = 0 & \Longrightarrow y=ax^2=0  \Longrightarrow a=0.
\end{align*}
Hence~$I(f)\cap\{XYZ=0\}$ is empty if $a\ne0$, and it contains the
point $\bigl\{[1,0,0]\bigr\}$ if~$a=0$.

Next suppose that $xyz\ne0$ and $a=-bg$.  If $a=0$, then necessarily
$g=0$ (since we cannot have~$a$ and~$b$ both~$0$). But then~$y=0$,
which contradicts the case that we are studying. On the other hand,
if~$g\ne0$, then the solutions to~\eqref{eqn:ax2byz} with $xyz\ne0$
are the points of the form $[1,-\g^2,\g]$ with $\g^3=-g$.  This
completes the proof that
\[
  I(f) = \begin{cases}
    \emptyset &\text{if $a\ne0$ and $a\ne-bg$,} \\
    \bigl\{[1,0,0]\bigr\} &\text{if $a=0$,} \\
    \bigl\{[1,-\g^2,\g] : \g^3=-g\bigr\} &\text{if $a=-bg\ne0$.}
  \end{cases}
\]
We also compute the critical locus
\[
  \Crit(f) = \{ ag X^3 + b Y^3 + bg Z^3 - (4a+bg) XYZ = 0 \}.
\]

\par\noindent\framebox{$\boldsymbol{dce\ne0}$ \textbf{and} $\boldsymbol{a=0}$}
\par
We note that since~$a=0$, we must have~$b\ne0$. The map~$f$ has the form
\[
  f(X,Y,Z) = [bYZ,Z^2+XY,Y^2+gXZ],
\]
and its indeterminacy locus is a single point,
$I(f)=\bigl\{[1,0,0]\bigr\}$.
Further, the critical locus of~$f$ is the cubic curve
\[
  \Crit(f) = \{ Y^3 + g Z^3 - g XYZ = 0 \}.
\]
If $g\ne0$, then~$\Crit(f)$ is a nodal cubic with node~$[1,0,0]$,
while if~$g=0$, then~$\Crit(f)$ is the triple line~$Y^3=0$. We
consider these cases separately.

\par\noindent\framebox{$\boldsymbol{dce\ne0}$ \textbf{and} $\boldsymbol{a=g=0}$}
\par
We are now in the case that $f(X,Y,Z) = [bYZ,Z^2+XY,Y^2]$,
$I(f)=\bigl\{[1,0,0]\bigr\}$, and $\Crit(f)=\{Y=0\}$.
Any~$\f\in\Hom(f,f')$ maps~$I(f)$ and~$\Crit(f)$ to~$I(f')$
and~$\Crit(f')$, so must have the form
$\f=\SmallMatrix{1&\mu&\l\\0&\a&0\\0&\g&\d\\}$. Comparing the first
coordinates of
\begin{align*}
  \f\circ f' &= \bigl [b'YZ+\mu(Z^2+XY)+\l Y^2,\, *,\, * \bigr], \\
  f\circ\f  &= \bigl[ b(\a Y)(\g Y+\d Z), \, *,\, * \bigr],
\end{align*}
the lack of a $Z^2$ term in the latter forces~$\mu=0$.  We also see
that $b'=\a\d b$ and $\l=b\d$.
Setting~$\mu=0$ gives
\begin{align*}
  \f\circ f' &= \bigl [*,\, \a(Z^2+XY),\, * \bigr], \\
  f\circ\f  &= \bigl[ *,\, (\g Y+\d Z)^2 + (X+\l Z)(\a Y),\,* \bigr].
\end{align*}
The lack of a $Y^2$ term in the former forces $\g=0$, and thus
\begin{align*}
  \f\circ f' &= \bigl [b'YZ+\l Y^2, \a Z^2+\a XY, \d Y^2 \bigr], \\
  f\circ\f  &= \bigl[ \a\d b Y Z, \d^2 Z^2 + \a XY + \a \l Y Z, \a^2 Y^2 \bigr].
\end{align*}
The lack of a $Y^2$ term in the first coordinate forces $\l=0$,
so~$\f$ is diagonal, and then comparing the remaining terms, we see
that
\[
  \f\circ f'=f\circ\f
  \quad\Longleftrightarrow\quad
  [b',\a,\a,\d] = [\a\d b,\d^2,\a,\a^2].
\]
Hence $\f\circ f'=f\circ\f$ if and only if $\a^3=\d^3=1$ and $\d=\a^2$
and $b'=b$, which completes the proof (in this case) that~$b$ is
a~$N(\Gcal_3)$ invariant and that~$\Aut(f)\cong C_3$.

\par\noindent\framebox{$\boldsymbol{dceg\ne0}$ \textbf{and} $\boldsymbol{a=0}$}
\par
We are working with maps of the form $f(X,Y,Z) =
[bYZ,Z^2+XY,Y^2+gXZ]$.  As in the previous case, we have
$I(f)=\bigl\{[1,0,0]\bigr\}$, but now the critical locus is a nodal
cubic curve,
\[
  C : Y^3+gZ^3-gXYZ=0, 
\]
with node at~$[1,0,0]$. Let $\f\in\Hom(f,f')$. Then~$\f$
fixes~$[1,0,0]$, and also induces an isomorphism of~$C\to C'$.  In
particular, the two tangent lines at the nodes form a
$\f$-invariant set, so~$\f$ either leaves each of the lines $YZ=0$
invariant, or it swaps them.

Suppose first that it leaves them invariant. Then~$\f$ has the form
$\f=\SmallMatrix{\a&\b&\g\\0&\d&0\\0&0&1\\}$ and we find that
\begin{align*}
  \f\circ f' &=  [ *,\, *,\, Y^2 + g'XZ], \\
  f\circ\f  &= \bigl[ *,\, *,\,  \d^2 Y^2 + g (\a X+\b Y+\g Z) Z \bigr].
\end{align*}
The lack of~$YZ$ and~$Z^2$ monomials and the assumption that~$g\ne0$ gives
$\b=\g=0$. Hence~$\f$ is diagonal, which yields
\[
  f^\f = [\a^{-1} \d b Y Z, \d^{-1} Z^2 + \a XY, \d^2 Y^2 + \a g X Z].
\]
Therefore
\[
f'=f^\f \quad\Longleftrightarrow\quad
[ \a^{-1}\d b, \d^{-1}, \a, \d^2, \a g] =[b', 1, 1, 1, g'].
\]
The middle three coordinates give $\d^{-1}=\a=\d^2$, which is
equivalent to $\a^3=\d^3=1$ and $\d=\a^2$. Then the other coordinates
force $b'=b$ and $g'=g$, so we obtain only the three maps in~$\Aut(f)$
that we already knew.

Next we consider the case that~$\f$ swaps the nodal tangent lines,
which means that~$\f$ has the form
$\f=\SmallMatrix{\a&\b&\g\\0&0&\d\\0&1&0\\}$. Then
\begin{align*}
  \f\circ f' &=  [ *,\, *,\, Z^2+XY], \\
  f\circ\f  &= \bigl[ *,\, *,\, (\d Z)^2 + g(\a X+\b Y+\g Z)Y \bigr].
\end{align*}
The lack of~$Y^2$ and~$YZ$ monomials (and the fact that $g\ne0$) tells
us that~$\b=\g=0$. Then
\[
  f^\f = [\a^{-1}\d bYZ, \d^2 Z^2 + \a g XY, \d^{-1} Y^2 + \a XZ],
\]
so $f'=f^\f$ if and only if
\[
  [\a^{-1}\d b, \d^2, \a g, \d^{-1}, \a] = [b', 1, 1, 1, g'].
\]
A bit of algebra shows that this last equality is equivalent to the
following four conditions:
\[
  g'=g^{-1},\quad b'=bg,\quad \d^3=1,\quad \a=(\d g)^{-1}.
\]
So first we find that~$f'\ne f$ is $N(\Gcal_3)$-conjugate to~$f$ if
and only if $(b',g')=(bg,g^{-1})$.  And second, we find that~$\Aut(f)$
has elements of this form if and only if $g=1$, in which case
taking~$\d^3=1$ and $\a= g\d^2$ gives three additional elements,
making~$\Aut(f)$ isomorphic to~$C_3\rtimes C_2$.

\par\noindent\framebox{$\boldsymbol{adce\ne0}$ \textbf{and} $\boldsymbol{a=-bg}$}
\par
In this case our maps look like
\[
  f(X,Y,Z) = [-bg X^2+bYZ,Z^2+XY,Y^2+gXZ].
\]
To make our computation notationally less cumbersome, we twist by
$[uX,vY,wZ]$ with $w=g^{-1/3}$, $v=1$, and $u=-w^2$. Then~$f$ has the form
\[
  f(X,Y,Z) = [bg  X^2 - bg  YZ,  Z^2 -  XY,  g Y^2 - g XZ].
\]
The indeterminacy locus consists of three points,
\[
  I(f) = \bigl\{ [1,\rho,\rho^2] : \rho\in\bfmu_3 \bigr\}.
\]

We make another change of variables to move the points in~$I(f)$ to
the standard basis vectors. Thus we let~$\z$ be a primitive cube root
of unity and $U=\SmallMatrix{1&1&1\\ 1&\z&\z^2\\ 1&\z^2&\z\\}$. Then
$f^U$ has indeterminacy locus
\[
  I(f^U) = \bigl\{[1,0,0],[0,1,0],[0,0,1]\bigr\}.
\]
To ease notation, we let $F=f^U$.  Letting $\pi=\z-1$, a short
calculation shows that~$F$ has the form
\begin{multline*}
  F(X,Y,Z) = [
  A XY + B XZ + C YZ, 
  B XY + C XZ + A YZ,\\*
  C XY + A XZ + B YZ],
\end{multline*}
where
\begin{align*}
  A &= \pi b g - (2\pi+3) g + (\pi+3), \\*
  B &= \pi b g + (\pi+3) g - (2\pi+3), \\*
  C &= \pi b g + \pi g + \pi.
\end{align*}
This system of linear equations relating~$(A,B,C)$
to~$(bg,g,1)$ has determinant~$27$, so~$A,B,C$ are not all~$0$.
Further, some linear algebra yields
\begin{align*}
  A=B &\quad\Longleftrightarrow\quad g=1, \\*
  A=B=0 &\quad\Longleftrightarrow\quad b=g=1.
\end{align*}

Momentarily writing~$F=F_{A,B,C}$ to indicate the dependence on the
coefficients, we observe that conjugation by the permutations
\[
  \s(X,Y,Z):=[Z,X,Y]\quad\text{and}\quad \t(X,Y,Z):=[Y,X,Z]
\]
has the effect
\[
  F_{A,B,C}^\s=F_{A,B,C}\quad\text{and}\quad F_{A,B,C}^\t=F_{B,A,C}.
\]
Thus $\s\in\Aut(F_{A,B,C})$ and $\t\in\Hom(F_{A,B,C},F_{B,A,C})$. In
particular, if~$A=B$, then $\t\in\Aut(F_{A,B,C})$.

Let $\f\in\Hom(F,F')$. Then~$\f$ permutes the points in~$I(F)=I(F')$.
Suppose first that~$\f$ fixes the three points in~$I(F)$,
so~$\f=\SmallMatrix{1&0&0\\0&\b&0\\0&0&\g\\}$ is a diagonal
matrix. Then
\begin{multline*}
  F^\f = [
    \b A XY + \g B XZ + \b\g C YZ,
    B XY + \b^{-1}\g C XZ + \g A YZ, \\*
    \b\g^{-1} C XY + A XZ + \b B YZ],
\end{multline*}
so $F'=F^\f$ if and only if
\begin{multline*}
  [A',A',A',B',B',B',C',C',C'] \\* =
  [A,\b A,\g A, B, \b B, \g B, \b\g C, \b^{-1}\g C, \b\g^{-1} C].
\end{multline*}
We first observe that
\begin{align*}
  A\ne0 &\quad\Longrightarrow\quad [A',A',A']=[A,\b A,\g A] \quad\Longrightarrow\quad \b=\g=1,\\
  B\ne0 &\quad\Longrightarrow\quad [B',B',B']=[B,\b B,\g B] \quad\Longrightarrow\quad \b=\g=1.
\end{align*}
Hence if either~$A$ or~$B$ is non-zero, then $\b=\g=1$ and~$\f$ is the
identity matrix and~$F=F'$.

On the other hand, if~$A=B=0$, then we have the single map
$[YZ,XZ,XY]$, and we already computed in
Proposition~\ref{proposition:classifyC2C2}(c) that
\[
  \Aut([YZ,XZ,XY])=\Scal_3\Gcal_{2,2}\cong S_4.
\]
  
Using the fact that $\s\in\Aut(F_{A,B,C})$ together with the fact that
$\t\in\Aut(F_{A,B,C})$ if and only if $A=B$, we have proven that
\begin{align}
  \label{eqn:FABCsC3}
  \bigl\{\f\in\Aut(F_{A,B,C}) &:   \text{$\f$ leaves $I(F_{A,B,C})$ invariant} \bigr\}  \notag\\*
  &\supseteq \begin{cases}
  \left\<\s  \right\> \cong C_3
  &\text{if $A\ne B$,} \\
  \left\<\s,\t  \right\> = \Scal_3 
  &\text{if $A=B\ne0$,} \\
  \Scal_3\Gcal_{2,2} \cong S_4.
  &\text{if $A=B=0$.} \\
\end{cases}
\end{align}
And further, if $F_{A',B',C'}\ne F_{A,B,C}$, then
\begin{align}
    \label{eqn:homFABC}
  \bigl\{\f\in\Hom(F_{A,B,C},&F_{A',B',C'}):   \text{$\f$ leaves $I(F_{A,B,C})$ invariant} \bigr\}  \notag\\*
    &\supseteq \begin{cases}
    \emptyset&\text{if  $(A',B')\ne(B,A)$,}\\
    \{\t,\t\s,\t\s^2\}&\text{if  $(A',B')=(B,A)$.}\\
  \end{cases}
\end{align}

Next suppose that~$\f\in\Hom(F,F')$ induces a cyclic permutation of
the three points in~$I(F)$. Then~$\s^i\f$ fixes~$I(F)$ for some
$i\in\{1,2\}$, and we also know that~$\s\in\Aut(F)$,
so~$\s^i\f\in\Hom(F,F')$. It follows that we get no new elements
of~$\Aut(F)$ beyond those already described in~\eqref{eqn:FABCsC3},
and we get no new possibilities for~$\Hom(F,F')$.

Finally, suppose that~$\f$ induces a transposition on the
set~$I(F)$. Then for an appropriate choice of~$i\in\{0,1,2\}$, the
map~$\t\s^i\f$ fixes~$I(F)$.  Hence~$\t\s^i\f$ is one of the maps
described by~\eqref{eqn:FABCsC3} (if $F'=F$) or by~\eqref{eqn:homFABC}
(if $F'=F$), and in all cases we see that~$\f$ is already included in
the list of maps in~\eqref{eqn:FABCsC3} or~\eqref{eqn:homFABC}.

\par\noindent\framebox{$\boldsymbol{adce\ne0}$ \textbf{and} $\boldsymbol{a\ne-bg}$}
\par
Our map looks like
\[
  f(X,Y,Z) = [aX^2+bYZ,Z^2+XY,Y^2+gXZ].
\]
It is a morphism, and its critical locus is the cubic curve
\[
  g X^3 + b Y^3 + b g Z^3 - (4 a + b g) X Y Z = 0.
\]
We consider various subcases depending on whether $b$ and/or $g$ vanishes.

\par\noindent\framebox{$\boldsymbol{adce\ne0}$ \textbf{and} $\boldsymbol{b=0}$}
\par
Our map looks like
\[
  f(X,Y,Z) = [aX^2,Z^2+XY,Y^2+gXZ].
\]
It is a morphism, and its critical locus is the reducible cubic curve
\[
  g X^3 -  4 a X Y Z = 0.
\]
If $g\ne0$, then $\Crit(f)$ is the union of the line $X=0$ and a conic,
while if $g=0$, then $\Crit(f)$ is the union of the three lines $XYZ=0$.

Suppose first that~$\f\in\Hom(f,f')$ fixes the line~$X=0$ (which is
necessary if $g\ne0$). Then~$\f$ has the form
$\f=\SmallMatrix{1&0&0\\ \l&\a&\b\\ \mu&\g&\d\\}$.  We note that
$\f\circ f'$ has no~$YZ$ monomials, while
\[
  f\circ\f = [\cdots,\, 2\g\d YZ + \cdots, 2\a\b YZ + \cdots].
\]
Hence $\a\b=\g\d=0$. The non-singularity of~$\f$ precludes some
possibilities, so either $\a=\d=0$ or $\b=\g=0$.

Suppose first that $\b=\g=0$. Then
\begin{align*}
  \f\circ f' &= [\cdots,\, \text{no $XZ$ monomial},\,\text{no $XY$ monomial} ], \\
  f\circ\f &= [\cdots,\, 2\mu\d XZ+\cdots,\, 2\l\a XY+\cdots].
\end{align*}
The non-singularity of~$\f$ forces $\a\d\ne0$, so we find that $\mu=\l=0$, i.e.,
the matrix~$\f$ is diagonal. Then
\[
   f^\f = [a X^2, \a^{-1}\d^2 Z^2 + XY, \a^2\d^{-1} + g XZ],
\]
so $f'=f^\f$ if and only if
\[
  [a',1,1,1,g'] = [a,\a^{-1}\d^2,1,\a^2\d^{-1},g].
\]
This occurs if and only if $(a',g')=(a,g)$ and $\a^3=\d^3=1$ and
$\d=\a^2$.  So we find only the three elements of~$\Aut(f)$ that we
already had.

Next suppose that $\a=\d=0$. Then
\begin{align*}
  \f\circ f' &= [\cdots,\, \text{no $XY$ monomial},\,\text{no $XZ$ monomial} ], \\
  f\circ\f &= [\cdots,\, 2\mu\g XY+\cdots,\, 2\l\b XZ+\cdots].
\end{align*}
The non-singularity of~$\f$ tells us that $\b\g\ne0$, so $\l=\mu=0$.
Then
\[
  f^\f = [aX^2, \g^{-1}\b^2 Z^2 + g XY, \b^{-1}\g^2 Y^2+ XZ].
\]
Hence $f'=f^\f$ if and only if
\[
  [a',1,1,1,g'] = [a,\g^{-1}\b^2,g,\b^{-1}\g^2,1].
\]
This forces $g'=g$ and $a=a'g$ and $a'=ag'$, which combine to give
$g=g'=\pm1$.  Further $\b^3=\g^3$ and $\g^2=g \b$. Hence
$\b^6=\g^6=(g\b)^3=g\b^3$, so $\b^3=g$. 

Thus if $g=1$, then $f'=f$ and $\Aut(f)$ contains three additional
elements corresponding to taking~$\b\in\bfmu_3$ and $\g=\b^2$, while
if $g=-1$, then we find that the maps $f_{a,-1}$ and $f_{-a,-1}$ are
$\PGL_3(K)$-conjugate.

To recapitulate, if $g\ne0$, then
\[
\Aut(f) = \begin{cases}
  \left\<\SmallMatrix{1&0&0\\0&\z&0\\0&0&\z^2} 
  \right\>=\Gcal_3 &\text{if $g\ne 1$,} \\[2\jot]
  \left\<\SmallMatrix{1&0&0\\0&\z&0\\0&0&\z^2},
  \SmallMatrix{1&0&0\\0&0&\z\\0&\z^2&0\\}
  \right\>\cong C_3\rtimes C_2 &\text{if $g= 1$,} \\
\end{cases}
\]
Further, distinct maps $f_{a,g}$ and $f_{a',g'}$ are $N(\Gcal_3)$-conjugate
if and only if $g=g'=-1$ and $a'=-a$.

If $g=0$, then $\Aut(f)$ contains the above maps, but we must also consider
the possibility that~$\f\in\Aut(f)$ non-trivially permutes the three
lines $XYZ=0$ in~$\Crit(f)$, which is the following case.

\par\noindent\framebox{$\boldsymbol{adce\ne0}$ \textbf{and} $\boldsymbol{b=g=0}$}
\par
Our map looks like
\[
  f(X,Y,Z) = [aX^2,Z^2+XY,Y^2],
\]
and we are looking for maps $\f\in\Hom(f,f')$ that induce a
non-trivial permutation of the lines~$XYZ=0$. Such a~$\f$ has the form
$\f=\psi \pi$ with $\pi\in\Scal_3$ and~$\psi$ diagonal. The
map~$f^\psi$ has the form $F_{A,B,C}:=[AX^2,BZ^2+CXY,DY^2]$ for some
non-zero $A,B,C,D$. We claim that a non-trivial
permutation~$\pi\in\Scal_3$ cannot take a map of the form~$F_{A,B,C}$
to another map of the same form. Lacking a clever argument, we simply compute
the effect of each permutation:
\begin{align*}
  \pi &= [Y,X,Z] &
  F_{A,B,C}^\pi &= [BZ^2+CXY,AY^2,DX^2], \\ 
  \pi &= [Z,Y,X] &
  F_{A,B,C}^\pi &=  [DY^2,BX^2+CZY,AZ^2], \\ 
  \pi &= [X,Z,Y] &
  F_{A,B,C}^\pi &=  [AX^2,DZ^2,BY^2+CXZ], \\ 
  \pi &= [Y,Z,X] &
  F_{A,B,C}^\pi &=  [DZ^2,AY^2,BX^2+CYZ], \\ 
  \pi &= [Z,X,Y] &
  F_{A,B,C}^\pi &=  [BY^2+CXZ,DX^2,AZ^2]. 
\end{align*}
This completes the proof that we obtain no new maps if $b=g=0$.

\par\noindent\framebox{$\boldsymbol{abdce\ne0}$ \textbf{and} $\boldsymbol{g=0}$}
\par
Our map looks like
\[
  f(X,Y,Z) = [aX^2+bYZ,Z^2+XY,Y^2].
\]
Its critical locus is the singular irreducible cubic curve
\[
   b Y^3 + b g Z^3 - 4 a  X Y Z = 0
\]
having a node at $[1,0,0]$, and the two tangent lines at the node
are $YZ=0$.  Hence any~$\f\in\Hom(f,f')$ must either fix or swap the
two lines~$YZ=0$ (which will also force it to fix their intersection
point~$[1,0,0]$).

Suppose first that it fixes the nodal tangent lines.  Then~$\f$ has
the form $\f=\SmallMatrix{1&\l&\nu\\0&\b&0\\0&0&\g\\}$, and we find
that
\begin{align*}
  \f\circ f' &= [\cdots,\, \b Z^2 + \b XY,\, \cdots], \\
  f\circ\f &= [\cdots,\, \g^2 Z^2 + \b XY + \b\l Y^2 + \b\nu YZ,\, \cdots].
\end{align*}
The absence of~$Y^2$ and~$YZ$ terms, together with the invertibility of~$\f$ (which implies
that~$\b\ne0$) forces $\l=\nu=0$, i.e.,~$\f$ is diagonal. Then
\[
  f^\f = [a X^2 + \b \g b YZ, \b^{-1}\g^2 Z^2 + XY, \b^2\g^{-1} Y^2],
\]
so $f'=f^\f$ if and only if
\[
  [a',b',1,1,1] = [a, \b \g b, \b^{-1}\g^2, 1, \b^2\g^{-1}],
\]
so if and only if $a'=a$ and $b'=\b\g b$ and
$\b^{-1}\g^2=\b^2\g^{-1}=1$. The last condition is equivalent to
$\b^3=1$ and $\g=\b^2$, so in particular~$\b\g=1$. Hence $f'=f^\f$ if
and only if~$f'=f$ and~$\f$ is one of the three maps that we already
knew was in~$\Aut(f)$.

Next suppose that~$\f$ swaps the nodal tangent lines.  Then~$\f$ has
the form $\f=\SmallMatrix{1&\l&\nu\\0&0&\b\\0&\g&0\\}$, and we find
that
\begin{align*}
  \f\circ f' &= [\cdots,\, \cdots,\, \g Z^2 + \g XY], \\
  f\circ\f &= [\cdots,\, \cdots,\, \b^2 Z^2].
\end{align*}
The lack of an~$XY$ monomial forces~$\g=0$, contradicting the
invertibility of~$\f$. Hence there are no~$\f\in\Hom(f,f')$ that swap
the lines~$XY=0$.

\par\noindent\framebox{$\boldsymbol{abcdeg\ne0}$ \textbf{and} $\boldsymbol{ace\ne-bdg}$}
\par
We have resumed using the general form
\[
  f(X,Y,Z) = [aX^2+bYZ,cZ^2+dXY,eY^2+gXZ],
\]
so the earlier dehomogenized condition $a\ne-bg$ for~$f$ to be a
morphism becomes $ace\ne-bdg$.

It is clear that every diagonal map~$\f\in\PGL_3$ preserves the form
of~$f$, as does the transposition~$[X,Z,Y]$. We are going to prove
that the full set of elements of~$\PGL_3$ that preserves this general
form is the the group generated by diagonal maps and this
transposition, except for one exceptional case

The map~$f$ is a morphism, and its critical locus is the cubic
curve
\begin{equation}
 \label{eqn:critfbdg4ace}
  \Crit(f): adg X^3 + bde Y^3 + bcg Z^3 - (bdg+4ace) XYZ = 0.
\end{equation}
Conjugating by~$\psi:=[uX,vY,wZ]$ gives
\[
  f^\psi = [uaX^2 + u^{-1}vwbYZ, v^{-1}w^2cZ^2 + udXY, v^2w^{-1}eY^2 + ugXZ]
\]
with critical locus    
\[
  \Crit(f^\psi) :  u^3 adg X^3 + v^3 bde Y^3 + w^3 bcg Z^3 - uvw (bdg+4ace) XYZ = 0.
\]
So taking $u^3=(adg)^{-1}$, $v^3=(bde)^{-1}$, and $w^3=(bcg)^{-1}$, the critical
locus becomes
\[
  \Crit(f^\psi) = \{ X^3 + Y^3 + Z^3 - \D XYZ = 0\},
\]
where $\D=(ace)^{-1/3}(bdg)^{-2/3} (bdg+4ace)$.  A short computation
shows that the cubic curve $\Crit(f^\psi)$ is non-singular if and only
if $\D^3\ne27$. (As we will see later, if~$\D^3=27$, then
$\Crit(f^\psi)$ is a union of three lines.)  Using the formula
for~$\D$, we observe that
\begin{align*}
  \D^3 - 27 &= (bdg+4ace)^3 - 27(ace)(bdg)^2 \\
  &=  (bdg+ace)(bdg-8ace)^2.
\end{align*}
We have ruled out $bdg=-ace$, so we are reduced to two cases, which we
consider in turn.

\par\noindent\framebox{$\boldsymbol{abcdeg\ne0}$ \textbf{and}
  $\boldsymbol{bdg\ne-ace}$ \textbf{and} $\boldsymbol{bdg\ne 8ace}$}
\par
The Hessian of $X^3 + Y^3 + Z^3 - \D XYZ$ is
\[
   \det\SmallMatrix{6X&-\D Z&-\D Y\\-\D Z&6Y&-\D X\\-\D Y&-\D X&6Z\\}
  =  -6 \D ^2 (X^3+Y^3+Z^3) + 2(108-\D ^3)XYZ.
\]
The flex points of the smooth cubic curve $\Crit(f^\psi)$ are thus the
roots of $(27-\D^3)XYZ=0$. Our assumptions imply that~$\D^3\ne27$, so
the flex points are the nine points with $XYZ=0$, i.e., the points
\[
  P_i := \begin{cases}
    [1,-\z^i,0],&\text{for $i=0,1,2$,}\\
    [0,1,-\z^i],&\text{for $i=3,4,5$,}\\
    [-\z^i,0,1],&\text{for $i=6,7,8$,}\\
  \end{cases}
\]
where we recall that~$\z$ is a primitive cube root of unity.

To ease notation, we let $F=f^\psi$, and we write~$F$ as
\begin{equation}
  \label{eqn:FAX2BYZ}
  F(X,Y,Z) = [AX^2+BYZ,CZ^2+DXY,EY^2+GXZ],
\end{equation}
where~$A,\ldots,G$ are monomials in fractional powers of~$a,\ldots,g$.
More precisely, tracking through their dependence on~$a,\ldots,g$, they
satisfy the multiplicative relations
\[
  AD = BC,\quad AG=BE,\quad BCG = 1,
\]
and if~$a,\ldots,g$ are generic, these are the only such relations
that they satisfy.\footnote{These formulas will explain why, when we
  compute the maps $\f\in\PGL_3$ that preserve the
  form~\eqref{eqn:FAX2BYZ}, we don't get all diagonal matrices.  More
  precisely, the effect of the diagonal matrix with entries~$\a,\b,\g$
  is to multiply $AB^{-1}C^{-1}D$ by~$(\a/\g)^3$ and to multiply $BCG$
  by~$\g^3$, so~$\a$ and~$\g$ must be cube roots of~$1$.}
We also let
\[
  \G = \{\text{flex points of $\Crit(F)$}\}= \{P_0,P_1,\ldots,P_8\}.
\]

Suppose that $\f\in\Hom(F,F')$. Then~$\f$ permutes the
nine points in~$\G$, but not entirely independently.  The line through
any two points in~$\G$ contains a unique third point of~$\G$, so
the~$\f$-images of two points in~$\G$ determines the~$\f$-image of a
third point.  So if we choose three non-colinear points, for
example~$P_0,P_1,P_3$, then the map~$\f$ is uniquely determined by the
images of these three points, where those images must be chosen from
among the non-colinear triples in~$\G$. Unfortunately, there are a
large number of possibilities.
 
For each triple of indices~$(i,j,k)$ such that $P_i,P_j,P_k$ are not
co-linear, we let~$\f\in\PGL_3$ be a general map satisfying
\[
  \f(P_0)=P_i,\quad \f(P_1)=P_j,\quad \f(P_3)=P_k.
\]
Thus~$\f$ has the form
\[
  \f = \SmallMatrix{*&*&*&\\ *&*&*&\\ *&*&*&\\}
       \SmallMatrix{1&0&0\\0&\b&0\\0&0&\g}
       \SmallMatrix{1&1&0\\ -1&-\z&1\\ 0&0&-1\\}^{-1},
\]
where the matrix with~$*$ entries is the matrix whose columns are the
point~$P_i,P_j,P_k$, and where~$\b,\g\in K^*$ are arbitrary. (If
necessary, we may write~$\f_{i,j,k,\b,\g}$ to indicate the dependence of~$\f$
on the various parameters.)

We compute~$F^\f$ and pick out the coefficients of the~$12$ monomials
that do not appear in~$F'$. Each of those coefficients is a linear
combination of~$A,\ldots,G$, with coefficients that are polynomials
in~$\b$ and~$\g$. We accumulate this data in the form
\begin{equation}
  \label{eqn:coefMABG}
  \SmallMatrix{
    \text{coeff of $Y^2$ in $X$-coord of $F^\f$} \\
    \text{coeff of $Z^2$ in $X$-coord of $F^\f$} \\
    \text{coeff of $XY$ in $X$-coord of $F^\f$} \\
    \text{coeff of $XZ$ in $X$-coord of $F^\f$} \\
    \text{coeff of $X^2$ in $Y$-coord of $F^\f$} \\
    \text{coeff of $Y^2$ in $Y$-coord of $F^\f$} \\
    \text{coeff of $XZ$ in $Y$-coord of $F^\f$} \\
    \text{coeff of $YZ$ in $Y$-coord of $F^\f$} \\
    \text{coeff of $X^2$ in $Z$-coord of $F^\f$} \\
    \text{coeff of $Z^2$ in $Z$-coord of $F^\f$} \\
    \text{coeff of $XY$ in $Z$-coord of $F^\f$} \\
    \text{coeff of $YZ$ in $Z$-coord of $F^\f$} \\
  }
  = M_{i,j,k}
  \SmallMatrix{
    A\\B\\C\\D\\E\\G\\
  },
\end{equation}
where $M_{i,j,k}$ is a~$12$-by-$6$ matrix whose entries are
polynomials in the ring $\QQ(\z)[\b,\g]$.  (At times we may
write~$M_{i,j,k}(\b,\g)$ to indicate the dependence
of~$M_{i,j,k}$ on~$\b$ and~$\g$.)  The fact that~$F^\f$ is not allowed
to have any of the indicated monomials implies
that~\eqref{eqn:coefMABG} is the zero vector, and then our assumption
that~$A,\ldots,G$ are all non-zero implies that the matrix~$M_{i,j,k}$
has rank at most~$5$. (Indeed, it implies the far stronger statement
that the column null space of~$M_{i,j,k}$ contains a vector whose
coordinates are all non-zero.)

Unfortunately, there are~$432$ valid~$i,j,k$ triples, and even
exploiting various symmetries, there are too many cases to check by
hand. So we give a computer assisted proof via the following algorithm.

\par\vspace{5pt}\noindent\textbf{Step 1}:
For each valid choice of~$i,j,k$, we computed the determinants of
various of the $6$-by-$6$ minors of~$M_{i,j,k}$ and set them equal
to~$0$. This gave many simultaneous equations for the two
unknowns~$\b$ and~$\g$. We used resultants on pairs of equations to
eliminate~$\g$, and then took the $\gcd$ of pairs of equations with
respect to~$\b$. Taking the square-free part, we obtained a separable
polynomial~$\Pi_{i,j,k}(\b)$ satisfying:
\[
  \rank M_{i,j,k}(\b,\g)\le 5~\text{for some $\b,\g\in K^*$}
  \quad\Longrightarrow\quad \Pi_{i,j,k}(\b)=0.
\]
The output from our program showed that
\[
  \Pi_{i,j,k}(\b)  \Bigm| (\b^6-1)\left(\b+\frac12\right).
\]
Indeed, the roots of~$\Pi_{i,j,k}(\b)$ are $6$'th roots of unity
except in the six cases $M_{6,8,0}$, $M_{7,8,0}$, $M_{6,8,1}$,
$M_{7,8,1}$, $M_{6,8,2}$, $M_{7,8,2}$, for which~$\Pi_{i,j,k}(\b)$
also had~$\b=-\frac12$ as a root.

\par\vspace{5pt}\noindent\textbf{Step 2}:
Loop through all valid~$i,j,k$ and all~$\b_0\in\bfmu_6\cup\{-\frac12\}$.
We let
\[
  r=r_{i,j,k}(\b_0,\g):=\rank M_{i,j,k}(\b_0,\g)
\]
denote the rank of $M_{i,j,k}(\b_0,\g)$ over the function
field~$K(\g)$, i.e., where $\g$ is an indeterminate.
We also write $r_{i,j,k}(\b_0,\g_0)$ for the rank of
the matrix when we set $\g=\g_0$.

\par\vspace{5pt}\noindent\textbf{Step 2.1}:
If $r_{i,j,k}(\b_0,\g)\le5$, compute the null space of
$M_{i,j,k}(\b_0,\g)$ over the function field~$\CC(\g)$.  We found that
there are~$144$ choices of~$(i,j,k,\b_0)$ for which
$r_{i,j,k}(\b_0,\g)\le5$, and in every case, every vector in
$\operatorname{Null}\bigl(M_{i,j,k}(\b_0,\g)\bigr)$ has at least one
coordinate equal to~$0$.

\begin{example}
\label{example:01k11}
We illustrate with an example. Let $k\in\{3,4,5\}$. Then
$r_{0,1,k}(1,\g)=3$, i.e., the matrix $M_{0,1,k}(1,\g)$ has rank~$3$
over the function field~$\CC(\g)$.  One then checks that every vector
\[
  [A,\ldots,G]\in\operatorname{Null}_{\CC(\g)}\bigl(M_{0,1,k}(1,\g)\bigr)
\]
has $B=E=0$. However, if we further specialize by setting~$\g=1$, then
$M_{0,1,k}(1,1)$ is the $0$-matrix. The associated elements
of~$\PGL_3$ are the diagonal matrices
\[
  \f_{0,1,3,1,1}=\SmallMatrix{1&0&0\\0&1&0\\0&0&1\\},\quad
  \f_{0,1,4,1,1}=\SmallMatrix{1&0&0\\0&1&0\\0&0&\z\\},\quad
  \f_{0,1,5,1,1}=\SmallMatrix{1&0&0\\0&1&0\\0&0&\z^2\\}.
\]
\end{example}


\par\vspace{5pt}\noindent\textbf{Step 2.2}:
Let $M'_{i,j,k}(\b_0,\g)$ be a $12$-by-$r$ matrix whose column span
over the function field~$\CC(\g)$ is the same as the column span
of $M_{i,j,k}(\b_0,\g)$. (In most cases we have $r=6$ and $M'=M$,
but as noted in Step~2.1, there are~144 cases with~$r\le5$.)

We computed the determinants of various $r$-by-$r$ minors
of~$M'_{i,j,k}(\b_0,\g)$ and then computed their~$\gcd$.  We found
exactly~$72$ values of~$(i,j,k,\b_0)$ such that there exists
some~$\g_0$ with
\[
  r_{i,j,k}(\b_0,\g_0) < r_{i,j,k}(\b_0,\g),
\]
i.e., such that the generic rank over~$\CC(\g)$ is strictly larger
than the specialized rank over~$\CC$ for some~$\g=\g_0\in\CC$.  More precisely,
there are~$36$ cases with generic rank~$r_{i,j,k}(\b_0,\g)=6$ and~$18$
cases each with~$r_{i,j,k}(\b_0,\g)=5$ and~$3$.  Further, and most
importantly, we found in all~$72$ cases that~$\b_0^3=1$
and~$\g_0^6=1$. The complete set of~$(i,j,k,\b_0)$ is given
in Table~\ref{table:Mijkb0g}.

\begin{table}[t]  
{\small\[
\begin{array}{|c|c|c|} \hline
M & r & \b_0 \\ \hline\hline  
M_{0,1,3} & 3 & 1\\ \hline 
M_{0,1,4} & 3 & 1\\ \hline 
M_{0,1,5} & 3 & 1\\ \hline 
M_{2,0,3} & 3 & 1\\ \hline 
M_{2,0,4} & 3 & 1\\ \hline 
M_{2,0,5} & 3 & 1\\ \hline 
M_{1,2,3} & 3 & 1\\ \hline 
M_{1,2,4} & 3 & 1\\ \hline 
M_{1,2,5} & 3 & 1\\ \hline 
M_{7,6,3} & 3 & \z^2\\ \hline 
M_{6,8,3} & 3 & \z^2\\ \hline 
M_{8,7,3} & 3 & \z^2\\ \hline 
M_{7,6,4} & 3 & \z^2\\ \hline 
M_{6,8,4} & 3 & \z^2\\ \hline 
M_{8,7,4} & 3 & \z^2\\ \hline 
M_{7,6,5} & 3 & \z^2\\ \hline 
M_{6,8,5} & 3 & \z^2\\ \hline 
M_{8,7,5} & 3 & \z^2\\ \hline
\end{array}
\quad
\begin{array}{|c|c|c|} \hline
M & r & \b_0 \\ \hline\hline  
M_{1,0,3} & 5 & \z\\ \hline 
M_{1,0,4} & 5 & \z\\ \hline 
M_{1,0,5} & 5 & \z\\ \hline 
M_{0,2,3} & 5 & \z\\ \hline 
M_{0,2,4} & 5 & \z\\ \hline 
M_{0,2,5} & 5 & \z\\ \hline 
M_{2,1,3} & 5 & \z\\ \hline 
M_{2,1,4} & 5 & \z\\ \hline 
M_{2,1,5} & 5 & \z\\ \hline 
M_{6,7,3} & 5 & \z^2\\ \hline 
M_{8,6,3} & 5 & \z^2\\ \hline 
M_{7,8,3} & 5 & \z^2\\ \hline 
M_{6,7,4} & 5 & \z^2\\ \hline 
M_{8,6,4} & 5 & \z^2\\ \hline 
M_{7,8,4} & 5 & \z^2\\ \hline 
M_{6,7,5} & 5 & \z^2\\ \hline 
M_{8,6,5} & 5 & \z^2\\ \hline 
M_{7,8,5} & 5 & \z^2\\ \hline 
\end{array}
\quad
\begin{array}{|c|c|c|} \hline
M & r & \b_0 \\ \hline\hline  
M_{0,1,6} & 6 & \z\\ \hline 
M_{0,1,7} & 6 & \z\\ \hline 
M_{0,1,8} & 6 & \z\\ \hline 
M_{2,0,6} & 6 & \z\\ \hline 
M_{2,0,7} & 6 & \z\\ \hline 
M_{2,0,8} & 6 & \z\\ \hline 
M_{1,2,6} & 6 & \z\\ \hline 
M_{1,2,7} & 6 & \z\\ \hline 
M_{1,2,8} & 6 & \z\\ \hline 
M_{1,0,6} & 6 & \z^2\\ \hline 
M_{1,0,7} & 6 & \z^2\\ \hline 
M_{1,0,8} & 6 & \z^2\\ \hline 
M_{0,2,6} & 6 & \z^2\\ \hline 
M_{0,2,7} & 6 & \z^2\\ \hline 
M_{0,2,8} & 6 & \z^2\\ \hline 
M_{2,1,6} & 6 & \z^2\\ \hline 
M_{2,1,7} & 6 & \z^2\\ \hline 
M_{2,1,8} & 6 & \z^2\\ \hline 
\end{array}
\quad
\begin{array}{|c|c|c|} \hline
M & r & \b_0 \\ \hline\hline  
M_{8,7,1} & 6 & 1\\ \hline 
M_{7,8,1} & 6 & 1\\ \hline 
M_{9,7,1} & 6 & 1\\ \hline 
M_{7,9,1} & 6 & 1\\ \hline 
M_{9,8,1} & 6 & 1\\ \hline 
M_{8,9,1} & 6 & 1\\ \hline 
M_{8,7,2} & 6 & 1\\ \hline 
M_{7,8,2} & 6 & 1\\ \hline 
M_{9,7,2} & 6 & 1\\ \hline 
M_{7,9,2} & 6 & 1\\ \hline 
M_{9,8,2} & 6 & 1\\ \hline 
M_{8,9,2} & 6 & 1\\ \hline 
M_{8,7,2} & 6 & 1\\ \hline 
M_{7,8,2} & 6 & 1\\ \hline 
M_{9,7,2} & 6 & 1\\ \hline 
M_{7,9,2} & 6 & 1\\ \hline 
M_{9,8,2} & 6 & 1\\ \hline 
M_{8,9,2} & 6 & 1\\ \hline 
\end{array}
\]}
\caption{Matrices $M_{i,j,k}(\b_0,\g)$ with generic rank $r$ such that there exists a $\g_0$
  such that $M_{i,j,k}(\b_0,\g_0)$ has rank strictly smaller than $r$}
\label{table:Mijkb0g}
\end{table}

\par\vspace{5pt}\noindent\textbf{Step 3}: It remains to
compute the null space of $M_{i,j,k}(\b_0,\g_0)$ as~$\b_0$
and~$\g_0$ range over $\b_0\in\bfmu_3$ and $\g_0\in\bfmu_6$. For each
$(i,j,k,\b_0,\g_0)$ such that~$\f_{i,j,k}$ is invertible, we check
whether the null space of $M_{i,j,k}(\b_0,\g_0)$ contains a vector
$(A,B,C,D,E,G)$ whose coordinates are all non-zero. It turns out that
in every such case the matrix $M_{i,j,k}(\b_0,\g_0)$ is
identically~$0$. Table~\ref{table:specialM} lists the values
of~$(i,j,k,\b_0,\g_0)$, together with the associated
map~$\f\in\PGL_3$. We observe that the~$\f$ in
Table~\ref{table:specialM} consist of the~$9$ diagonal maps
satisfying~$\f^3=1$, together with the conjugation of these~$9$ maps
by the transposition~$[X,Z,Y]$.

\begin{table}[t]
{\small\[
\begin{array}{|c|c|c|c|} \hline
  M & \b_0 & \g_0 & \phi \\ \hline\hline
    M_{1,2,4} = 0 & 1 & 1 & \SmallMatrix{ 1 & 0 & 0 \\0 & 1 & 0 \\0 & 0 & 1 \\ } \\ \hline
    M_{1,2,5} = 0 & 1 & 1 & \SmallMatrix{ 1 & 0 & 0 \\0 & 1 & 0 \\0 & 0 & \smash[t]{\z^2} \\ } \\ \hline
    M_{1,2,6} = 0 & 1 & 1 & \SmallMatrix{ 1 & 0 & 0 \\0 & 1 & 0 \\0 & 0 & \z \\ } \\ \hline
    M_{3,1,4} = 0 & 1 & \z & \SmallMatrix{ 1 & 0 & 0 \\0 & \z & 0 \\0 & 0 & \z \\ } \\ \hline
    M_{3,1,5} = 0 & 1 & \z & \SmallMatrix{ 1 & 0 & 0 \\0 & \z & 0 \\0 & 0 & 1 \\ } \\ \hline
    M_{3,1,6} = 0 & 1 & \z & \SmallMatrix{ 1 & 0 & 0 \\0 & \z & 0 \\0 & 0 & \smash[t]{\z^2} \\ } \\ \hline
    M_{2,3,4} = 0 & 1 & \z^2 & \SmallMatrix{ 1 & 0 & 0 \\0 & \smash[t]{\z^2} & 0 \\0 & 0 & \smash[t]{\z^2} \\ } \\ \hline
    M_{2,3,5} = 0 & 1 & \z^2 & \SmallMatrix{ 1 & 0 & 0 \\0 & \smash[t]{\z^2} & 0 \\0 & 0 & \z \\ } \\ \hline
    M_{2,3,6} = 0 & 1 & \z^2 & \SmallMatrix{ 1 & 0 & 0 \\0 & \smash[t]{\z^2} & 0 \\0 & 0 & 1 \\ } \\ \hline
\end{array}\quad
\begin{array}{|c|c|c|c|} \hline
  M & \b_0 & \g_0 & \phi \\ \hline\hline
    M_{8,7,4} = 0 & \z^2 & 1 & \SmallMatrix{ 1 & 0 & 0 \\0 & 0 & \z \\0 & \z & 0 \\ } \\ \hline
    M_{7,9,4} = 0 & \z^2 & 1 & \SmallMatrix{ 1 & 0 & 0 \\0 & 0 & 1 \\0 & 1 & 0 \\ } \\ \hline
    M_{9,8,4} = 0 & \z^2 & 1 & \SmallMatrix{ 1 & 0 & 0 \\0 & 0 & \smash[t]{\z^2} \\0 & \smash[t]{\z^2} & 0 \\ } \\ \hline
    M_{8,7,5} = 0 & \z^2 & \z & \SmallMatrix{ 1 & 0 & 0 \\0 & 0 & \smash[t]{\z^2} \\0 & \z & 0 \\ } \\ \hline
    M_{7,9,5} = 0 & \z^2 & \z & \SmallMatrix{ 1 & 0 & 0 \\0 & 0 & \z \\0 & 1 & 0 \\ } \\ \hline
    M_{9,8,5} = 0 & \z^2 & \z & \SmallMatrix{ 1 & 0 & 0 \\0 & 0 & 1 \\0 & \smash[t]{\z^2} & 0 \\ } \\ \hline
    M_{8,7,6} = 0 & \z^2 & \z^2 & \SmallMatrix{ 1 & 0 & 0 \\0 & 0 & 1 \\0 & \z & 0 \\ } \\ \hline
    M_{7,9,6} = 0 & \z^2 & \z^2 & \SmallMatrix{ 1 & 0 & 0 \\0 & 0 & \smash[t]{\z^2} \\0 & 1 & 0 \\ } \\ \hline
    M_{9,8,6} = 0 & \z^2 & \z^2 & \SmallMatrix{ 1 & 0 & 0 \\0 & 0 & \z \\0 & \smash[t]{\z^2} & 0 \\ } \\ \hline
\end{array}
\]}
\caption{Matrices $M_{i,j,k}(\b_0,\g_0)$ with $\b_0,\g_0\in\CC^*$ such
  that $\operatorname{Null}\bigl(M_{i,j,k}(\b_0,\g_0)$ contains a
  vector with all non-zero coordinates, with the associated
  $\f\in\PGL_3$}
\label{table:specialM}
\end{table}

This long calculation completes the proof that under our assumptions
that~$abcdeg\ne0$ and $ace\ne-bdg$, maps of the form
\begin{equation}
  \label{eqn:fax2byz}
  f(X,Y,Z) = [aX^2+bYZ,cZ^2+dXY,eY^2+gXZ]
\end{equation}
satisfy
\[
  \Hom(f,f') \subseteq \Dcal \cup \pi\Dcal,
\]
where~$\Dcal\subset\PGL_3$ is the group of diagonal matrices and~$\pi$
is the transposition~$\pi=[X,Z,Y]$.  We  normalize maps of the
form~\eqref{eqn:fax2byz} by conjugating by~$[uX,vY,wZ]$ with
$u=a^{-1}$, $v^3=c^{-1}e^{-2}$, and $w^3=c^{-2}e^{-1}$. This puts~$f$ into the
normalized form
\[
  f_{b,d,g}(X,Y,Z) = [X^2+bYZ,Z^2+dXY,Y^2+gXZ]
\]
that depends on three non-zero parameters~$b,d,g$ satisfying
$bdg\ne-1$.

We consider first conjugation by a diagonal map $\f=[ X,\b Y,\g Z]$. Then
\[
  f_{b,d,g}^\f(X,Y,Z) =
  [ X^2 + \b \g bYZ, \b^{-1}\g^2 Z^2 +  dXY, \b^2\g^{-1} Y^2 +  gXZ].
\]
The map~$f_{b,d,g}^\f$ is thus in normalized form if and only if
\[
  [ \b \g, \b^{-1}\g^2, \b^2\g^{-1}, 1 , d, g] = [1,1,1,1,d,g].
\]
This is true if and only if $\b\in\bfmu_3$ and $\g=\b^2$, i.e., if 
the map~$\f$ is in the subgroup of order~$3$ that we already know is
contained in~$\Aut(f)$.

Next we conjugate by the composition of a diagonal map and the
transposition $[X,Z,Y]$, i.e., by a map of the form~$\f=[X,\b Z,\g Y]$. Then
\[
  f_{b,d,g}^\f(X,Y,Z) =
  [ X^2 + \b \g bYZ, \b^{2}\g^{-1} Z^2 +  g XY, \b^{-1}\g^{2} Y^2 +   d XZ].
\]
The map~$f_{b,d,g}^\f$ is thus in normalized form if and only if
\[
  [ \b \g, \b^2\g^{-1}, \b^{-1}\g^2, 1, g, d ] = [1,1,1,1,d,g].
\]
The first four coordinates force $\b\in\bfmu_3$ and $\g=\b^2$, as
expected.  The last three coordinates force $d=g$, so we find
that~$f_{b,d,g}$ and~$f_{b,g,d}$ are $\PGL_3$-conjugates, and if
$d=g$, then $\Aut(f_{b,d,g})$ contains the transposition~$[X,Z,Y]$, so
is isomorphic to~$C_3\rtimes C_2$.

\par\noindent\framebox{$\boldsymbol{abcdeg\ne0}$ \textbf{and}  $\boldsymbol{bdg=8ace}$}
\par
We want to make a change of variables so that $b=2a$, $d=2c$ and $g=2e$.
Conjugating by $[uX,vY,wZ]$ with
\[
  u=(adg)^{-1/3},\quad v=(bde)^{-1/3},\quad w=(bcg)^{-1/3}
\]
puts~$f$ into this form, i.e.,
\begin{equation}
  \label{eqn:fXYZaX22YZ}
  f(X,Y,Z) = \bigl[ a(X^2+2YZ), c(Z^2+2XY), e(Y^2+2XZ) \bigr].
\end{equation}
N.B. This only works because of our assumption that $bdg=8ace$. So
we are reduced to studying~$f=f_{a,c,e}$ in this normalized form.
The critical locus is
\[
  \Crit(f_{a,c,e}) = \{X^3+Y^2+Z^3-3XYZ=0\}.
\]
The critical locus decomposes as a union of three lines via the
factorization
\[
  X^3+Y^2+Z^3-3XYZ = 
  (  X + Y+ Z)(  X+\z  Y+\z^2  Z)(  X+\z^2  Y + \z  Z).
\]
The pairwise intersections of these three lines in $\Crit(f_{a,c,e})$
gives the following set of three points,
\[
  \bigl\{[1, 1, 1],  [1,\z , \z^2 ],  [1,\z^2 , \z ]\bigr\}.
\]
Every $\f\in\Hom(f,f')$ stabilizes this set, so has the form
\begin{align*}
    \f&=
    \SmallMatrix{1&1&1\\ 1&\z &\z^2 \\ 1 & \z^2  & \z  \\ }
    \SmallMatrix{1&0&0\\ 0&\b&0\\ 0&0&\g} \pi
    \SmallMatrix{1&1&1\\ 1&\z^2 &\z \\ 1 & \z  & \z^2  \\ } \in \PGL_3
\end{align*}
for some $\b,\g\in K^*$ and some permutation $\pi\in\Scal_3$. 

Just as in the previous case, for each $\pi\in\Scal_3$ we compute
$f^\f$ and pick out the coefficients of the~$12$ monomials that do not
appear in~$f'$. Each of those coefficients is a linear combination
of~$a,c,e$, with coefficients that are polynomials in~$\b$
and~$\g$. Just as in~\eqref{eqn:coefMABG}, we accumulate this data in
the form
\begin{equation} 
  \label{eqn:coefMace}
  \SmallMatrix{
    \text{coeff of $Y^2$ in $X$-coord of $f^\f$} \\
    \text{coeff of $Z^2$ in $X$-coord of $f^\f$} \\
    \vdots \\
    \text{coeff of $XY$ in $Z$-coord of $f^\f$} \\
    \text{coeff of $YZ$ in $Z$-coord of $f^\f$} \\
  }
  = M_\pi
  \SmallMatrix{a\\c\\e\\},
\end{equation}
where~$M_\pi$ is a $12$-by-$3$ matrix whose entries are polynomials in
the ring $\QQ(\z)[\b,\g]$.  (At times we may write~$M_\pi(\b,\g)$ to
indicate the dependence of~$M_\pi$ on~$\b$ and~$\g$.)  The fact
that~$f^\f$ is not allowed to have any of the indicated monomials
implies that~\eqref{eqn:coefMace} is the zero vector, and then our
assumption that~$a,c,e$ do not vanish implies that the matrix~$M_\pi$
has rank at most~$2$. (Indeed, it implies the far stronger statement
that the column null space of~$M_\pi$ contains a vector whose
coordinates are all non-zero.)

There are only~$6$ choices for~$\pi\in\Scal$, which we compute in
turn.  For each~$\pi$ we computed the determinants of the $3$-by-$3$
minors of~$M_\pi$. Of these~$220$ minors, exactly~$160$ have non-zero
determinant, and aside from factors of the form~$c\b^i\g^j$ with
$c\in\ZZ$, these~$160$ non-zero determinants yield exactly~$8$
distinct polynomials in~$\QQ(\z_3)[\b,\g]$. Taking pairwise resultants to
eliminate~$\b$ (respectively~$\g$) and then pairwise gcds, we find
that a necessary condition for $\rank M_\pi(\b,\g)\le2$ is that~$\b$
and~$\g$ satisfy
\[
  \b^7 = \g^7 = 1.
\]
For each such pair~$(\b,\g)$, we compute the null space
of $M_\pi(\b,\g)$ and use it to find all maps~$f$ of the
form~\eqref{eqn:fXYZaX22YZ} and all~$\f$ such that~$f^\f$ has the same
form.

For example, taking $\b=\g=1$ gives $M_\pi(1,1)=0$, so we obtain
maps~$\f$ that are allowed for every~$f$. More precisely, the
elements~$\pi\in\Scal_3$ of order dividing~$3$ give the three maps
\[
  [X,Y,Z], [X,\z Y,\z^2 Z], [X,\z^2 Y, \z Z] \in \Aut(f)
\]
that we already know are in $\Aut(f)$, while the three
elements~$\pi\in\Scal_3$ of order~$2$ give maps~$\f$ satisfying
$f^\f_{a,c,e}=f_{a,e,c}$.

More interesting are the cases for which $\b=\z_7$ is a primitive
$7$'th root of unity. Writing~$\z_3$ instead of~$\z$ for our chosen
primitive cube root of unity, we find that $\b=\z_7$ is possible in
exactly the following situations:
\begin{align*}
  \g=\z_7^3 &\quad\text{and}\quad [a,c,e] = [1,\z_3,\z_3^2], \\
  \g=\z_7^5 &\quad\text{and}\quad [a,c,e] = [1,\z_3^2,\z_3].
\end{align*}     
Since we know how to swap~$c$ and~$e$, it suffices to consider the first case,
for which we find that~$\Aut(f)$ contains the following element of order~$7$:
\[
  \def\zz{\z^{\vphantom1}}
  \SmallMatrix{\z_7^3 + \zz_7 + 1 & \zz_3 \z_7^3 - \z_3^2 \zz_7 + 1 & \z_3^2 \z_7^3 + \zz_3 \zz_7 + 1\\
    \z_3^2 \z_7^3 + \zz_3 \zz_7 + 1 & \z_7^3 + \zz_7 + 1 & \zz_3 \z_7^3 - \z_3^2 \zz_7 + 1\\
    \zz_3 \z_7^3 - \z_3^2 \zz_7 + 1 & \z_3^2 \z_7^3 + \zz_3 \zz_7 + 1 & \z_7^3 + \zz_7 + 1}
  \in \Aut\bigl(f_{1,\zz_3,\z_3^2}\bigr).
\]
Of course, $\Aut(f)$ also contains the powers of this map, and
composing with one of the transpositions in~$\Scal_3$ gives a map in
$\Hom\bigl(f_{1,\z^{\vphantom1}_3,\z_3^2},f_{1,\z_3^2,\z^{\vphantom1}_3}\bigr)$.

We note that these two maps~$f$ for which~$\Aut(f)$ contains an element of
order~$7$ are both conjugate to the map~$[Z^2,X^2,Y^2]$ that we
studied in Propositions~\ref{proposition:disteps}
and~\ref{proposition:C5C7etc}. Indeed, one finds that
$\f = \SmallMatrix{1&\z_3^2&\z_3\\ 1&\z_3&\z_3^2\\ 1&1&1\\ }$ gives
\[
  [Z^2,X^2,Y^2]^\f = [X^2+2YZ,\z_3 (Z^2+2 XY), \z_3^2( Y^2 + 2 XZ)].
\]
So we can refer to Proposition~\ref{proposition:C5C7etc} for the fact that
these maps have automorphism group exactly equal to~$C_7\rtimes C_3$.

It remains to check that maps of Types~$C_3(n)$ and~$C_3(n')$ are not
$N(\Gcal_3)$-conjugate for $n\ne n'$, nor indeed are
they~$\PGL_3$-conjugate.  To do this, we note that the only case in
which the indeterminacy and critical loci for~$C_3(n)$ and~$C_3(n')$
have the same geometry is
\[
  C_3(2) \stackrel{?}{=} C_3(3), \quad I(f)=\text{1 point},\quad\Crit(f)=\text{triple line}.
\]
Any $\f\in\Hom(f,f')$ preserves the geometry of these loci, so we ask if the maps
\[
  f = [aX^2+YZ,XY,Y^2]\quad\text{and}\quad f'=[bYZ,Z^2+XY,Y^2]
\]
can be conjugate to one another.  Since $\Crit(f)=\Crit(f')$ is the
line $Y=0$ (with multiplicity~$3$), we must have
$\f\bigl(\{Y=0\}\bigr)=\{Y=0\}$. But we observe that
\begin{align*}
  f\bigl(\Crit(f)\bigr) &= f\bigl(\{Y=0\}\bigr) = \bigl\{[1,0,0]\bigr\} \in \Crit(f),\\
  f'\bigl(\Crit(f')\bigr) &= f'\bigl(\{Y=0\}\bigr) = \bigl\{[0,1,0]\bigr\} \notin \Crit(f').
\end{align*}
Hence if there were a map~$\f\in\Hom(f,f')$, then we would find that
\begin{align*}
  f'\bigl(\Crit(f')\bigr)  &= f^\f\bigl(\Crit(f^\f)\bigr)  = f^\f\bigl(\Crit(f)^\f\bigr) \\
  &= \Bigl(f\bigl(\Crit(f)\bigr)\Bigr)^\f  \in \Crit(f)^\f  = \Crit(f^\f)  = \Crit(f').
\end{align*}
This contradiction shows that~$f$ and~$f'$ are not~$\PGL_3(K)$-conjugate.
\end{proof}

\begin{proof}[Proof of Proposition $\ref{proposition:C4notpossible}$]
The map $[X,\z Y,Z]$ is the map~$\t_0$ defined in
Section~\ref{section:diagstability}. Setting $m=0$ in Table~\ref{table:effectoftm}
and reducing the entries modulo~$3$ yields
{\small
\[
\begin{array}{|c||c|c|c|c|c|c|} \hline
  & X^2 & Y^2 & Z^2 & XY & XZ & YZ \\ \hline\hline
  {\text{$X$-coord}}
  & 0 & 2 & 0 & 1 & 0 & 1 \\ \hline
  {\text{$Y$-coord}}
  & 2 & 1 & 2 & 0 & 2 & 0 \\ \hline
  {\text{$Z$-coord}}
  & 0 & 2 & 0 & 1 & 0 & 1 \\ \hline 
\end{array}
\]
}%
Hence assuming that~$\t_0\in\Aut(f)$ leads to the following three families of maps:
\begin{align*}
  f_{0,0} &:= [aX^2+bZ^2+cXZ,dXY+eYZ,gZ^2+hXZ], \\
  f_{0,1} &:= [aXY+bYZ,cY^2,dXY+eYZ], \\
  f_{0,2} &:= [aY^2,bX^2+cZ^2+dXZ,eY^2].
\end{align*}


For the first two maps we use Table~\ref{table:effectofL} to compute
\begin{align*}
  \mu^{\Ocal(1)}(f_{0,0},L_{k,\ell})
  &\le \bigl\{ -k, 3k+2\ell, k+\ell \}
  \xrightarrow{(k,\ell)=(1,-2)} -1, \\
  \mu^{\Ocal(1)}(f_{0,1},L_{k,\ell})
  &\le \bigl\{ -\ell, -2k, -2k-2\ell \}
  \xrightarrow{(k,\ell)=(1,1)} -1.
\end{align*}
Hence~$f_{0,0}$ and~$f_{0,1}$ are $\Dcal$-unstable. (The latter also
has degree~$1$, of course.)  And finally, we see that the map
$f_{0,2}$ is not dominant, since its image is contained in the line
$\{eX=aZ\}$.  (Or, if $a=e=0$, then $f_{0,2}(\PP^2)=[0,1,0]$.) 
\end{proof}

\section{Proof of Theorems~\ref{theorem:main1} and~\ref{theorem:main2}
    and of Corollary~\ref{corollary:listofaut}}
\label{section:endproof}

In this and the next section, we use our accumulated results to prove
Theorems~\ref{theorem:main1} and~\ref{theorem:main2} and
Corollary~\ref{corollary:listofaut}. We recall that the assumptions
for both the theorem and corollary are that
$f:\PP^2\dashrightarrow\PP^2$ is a dominant rational map of degree~$2$
lying in the semi-stable locus of~$\Rat_2^2$ and such that
$\infty>\#\Aut(f)\ge3$.

We remark that the computation of the indeterminacy and critical loci
of maps in the various families as described in
Table~\ref{table:maintable} is an elementary, albeit tedious,
calculation, so in some cases we have omitted the details.

We start with the assumption that~$\Aut(f)$ is finite and of order at
least~$3$. Let~$G$ be a finite group of order at least~$3$.  Then
either $p\mid\#G$ for some odd prime $p$, or else $G$ is a 2-group of
order at least~$4$. In the former case, Cauchy's theorem says that $G$
contains an element of order $p$, while in the latter case, the strong form
of the first Sylow theorem~\cite[Theroem~2.12.1]{MR1375019} says that
$G$ contains a subgroup of order~$4$, hence contains a copy of either
$C_2^2$ or $C_4$.

Suppose that there is an odd prime~$p$ with $p\mid \Aut(f)$, and let
$G\subset\Aut(f)$ be a subgroup of
order~$p$. Lemma~\ref{lemma:GinPGL3}(a) tells us that
after~$\PGL_3(K)$-conjugation, we may assume that
\[
  G = \left\< \SmallMatrix{1&0&0\\0&\z&0\\0&0&\z^m}\right\>
\]
for some primitive $p$'th root of unity~$\z$ and some integer~$m$.
For future reference, we also remark that when $m=1$, we may instead
use the map $\SmallMatrix{1&0&0\\0&\z&0\\0&0&1}$, which is
$\PGL_3$-conjugate to $\SmallMatrix{1&0&0\\0&\z&0\\0&0&\z}$.

We consider the case that $p\ge5$.  Then
Proposition~\ref{proposition:disteps} and the semi-stability
assumption tell us that~$f$ has one of the following three forms:
\[
  [aX^2+YZ,bXY,cXZ],\quad [YZ,X^2,Y^2],\quad [Z^2,X^2,Y^2].
\]
Proposition~\ref{proposition:C5C7etc} says that maps~$f$ in the first
family have a copy of~$\GG_m$ in~$\Aut(f)$, while the second map
satisfies $\Aut(f)\cong C_5$, and the third map satisfies
$\Aut(f)\cong C_7\rtimes C_3$. (A conjugate of this last case also
appears in the compendium of maps for which~$\Aut(f)$ contains an
element of order~$3$.)  This completes the proof of the automorphism
parts of Theorems~\ref{theorem:main1} and~\ref{theorem:main2} and of
Corollary~\ref{corollary:listofaut} in the case that~$\#\Aut(f)$ is
divisible by a prime $p\ge5$.

Similarly, if $p=3$, i.e., if $3\mid\#\Aut(f)$, then the automorphism
parts of Theorems~\ref{theorem:main1} and~\ref{theorem:main2} and of
Corollary~\ref{corollary:listofaut} can be deduced from
Proposition~\ref{proposition:autc3types}, although we need to do
some work to put the maps into the indicated forms. (Using the above
remark and Proposition~\ref{proposition:autc3types}(b), we need only
consider the case~$m=2$.)

Thus let $f=[aX^2+bYZ,cZ^2+dXY,eY^2+gXZ]$.  If $ace\ne0$, then the
transformation
\[
  \psi(X,Y,Z) = \bigl[ (ce)^{2/3}X, ac^{1/3}Y, ae^{1/3}Z \bigr]
\]
yields
\[
  f^\psi(X,Y,Z) 
  = [ X^2 + ab(ce)^{-1} YZ, Z^2 + a^{-1} d XY, Y^2 + a^{-1}g XZ ].
\]
In other words, if $ace\ne0$, we can find a normal form for~$f$ with
$a=c=e=1$.  We apply this transformation to the maps of Type~$C_3(1)$,
$C_3(5)$, $C_3(6)$, $C_3(7)$, and $C_3(8)$ in
Table~\ref{table:autord3}, which serves to give them a more uniform
description. Types~$C_3(1)$ and~$C_3(7)$ already have this form.

For~$C_3(5)$ we find that
\[
  f^\psi(X,Y,Z) = [X^2 - a^2b^{-1}YZ, Z^2 - a^{-1}XY, Y^2 - a^{-1}b XZ].
\]
Thus
\[
  f^\psi(X,Y,Z) = [X^2 + B YZ, Z^2 + D XY, Y^2 + G XZ]
  \quad\text{with $BDG= -1$.}
\]
Further, the conditions that~$a$ and/or $b$ equal~$1$ become
\begin{align*}
  a=b=1 &\Longleftrightarrow D=G= -1, &
  a=1,\;b\ne1 &\Longleftrightarrow D= -1\ne G, \\
  a\ne1,\;b=1 &\Longleftrightarrow D=G\ne -1, &
  a\ne1,\;b\ne1 &\Longleftrightarrow D\ne -1,\;D\ne G
\end{align*}

For $C_3(6)$ we find that
\[
  f^\psi(X,Y,Z) = [X^2 + ab YZ, Z^2 + a^{-1} XY, Y^2 ].
\]
So with a slight relabeling, Type~$C_3(6)$ becomes simply
\[
  f^\psi(X,Y,Z) = [X^2 + b YZ, Z^2 + d XY, Y^2 ],
\]
i.e., it's~$C_3(7)$ with $g=0$.

And for $C_3(8)$, we find that
\[
  f^\psi(X,Y,Z) = [ X^2 + 2(ce)^{-1} YZ, Z^2 +  2c XY, Y^2 + 2e XZ ].
\]
In other words, we obtain the~$C_3(7)$ form with $bdg=8$. Further, the
only cases with $\Aut(f)\ne C_3$ are $c$ a primitive cube root
of unity and $e=c^2$, with these two cases being conjugate. In
particular, $ce=1$.

We next consider Types~$C_3(2)$,~$C_3(3)$, and~$C_3(4)$. We claim that
they may all be put into the form
$f_{a,c,g}:=[aX^2+YZ,cZ^2+XY,Y^2+gXZ]$ with $(a,b,g)\ne(0,0,0)$ and
one or more of~$a,c,g$ equal to~$0$.  For Type~$C_3(2)$, we already
have $f=f_{a,0,g}$.  For Types~$C_3(3)$ and~$C_3(4)$, which have the
form~$f=[bYZ,Z^2+XY,Y^2+gXZ]$ with $b\ne0$, the transformation
$\psi=[b^{2/3}X,b^{1/3}Y,Z]$ yields
$f^\psi=[YZ,b^{-1}Z^2+XY,Y^2+gXZ]$, so we get maps~$f_{0,c,g}$
with~$c\ne0$.  We also observe that the $\PGL_3$-conjugacies for
Type~$C_3(4)$ become the maps $\f=[c^{1/3}X,c^{2/3}g^{1/3}Y,g^{2/3}Z]$
which have the effect $f_{0,c,g}^\f=f_{0,c/g,1/g}$.

We next consider the maps such that~$\Aut(f)$ contains an element of
order~$4$. The description of these maps in
Proposition~\ref{proposition:autc4types} is already in the form that we want.

Finally we consider the maps such that~$\Aut(f)$ contains a subgroup
of type~$C_2^2$. To fit these maps, which are described in
Proposition~\ref{proposition:classifyC2C2}, into a single family, we
apply the transformation $\psi=[X/\sqrt{d},Y,Z]$ to the map
$f=[X^2+Y^2-Z^2,dXY,eXZ]$. This gives
$f^\psi=[d^{-1}X^2+Y^2-Z^2,XY,ed^{-1}XZ]$, so these maps have the form
$f_{a,e}=[aX^2+Y^2-Z^2,XY,eXZ]$. Then one family in
Proposition~\ref{proposition:classifyC2C2} is~$f_{a,e}$ with $a\ne0$
and the other family is~$f_{0,e}$.

We conclude this section with the one part of
Corollary~\ref{corollary:listofaut} that is not immediate from the
main theorems.  Corollary~\ref{corollary:listofaut}(b) asserts that
each group listed in~(a) occurs as the full automorphism group of a
semistable map of degree~$2$. Theorems~\ref{theorem:main1}
and~\ref{theorem:main2}, together with appropriately chosen maps from
Table~\ref{table:maintable}, take care of $G\in\{C_3,C_4,C_5,C_2^2\}$.
That same table gives us maps with $\Aut(f)\in\{S_3,S_4,C_7\rtimes C_3\}$.
For these groups it suffices to point out that the identity component
of the normalizers of the associated subgroups of~$\PGL_3$ is in all
cases is equal to the group of diagonal matrices, and hence
semistability for (say)~$\Gcal_3$ or~$\Gcal_7$ gives semistability for
the larger group. It remains to deal with the cases~$G=C_1$
and~$G=C_2$.

Consider the family of maps
\[
  f_{u,v} = [YZ,X^2, u XY + v Y^2].
\]
Then $\Aut(f_{1,0})\cong C_4$ and $\Aut(f_{0,1})\cong C_5$. Hence the generic
member of this family has $\Aut(f_{u,v})=C_1$.

Similarly, consider the family of maps
\[
  f_{u,v,w} = [ u Y^2 + v Z^2, XY, w Y^2 + 2XZ ].
\]
Then $\Aut(f_{0,1,1})\cong C_4$ and $\Aut(f_{1,-1,0})\cong C_2^2$, so
generically~$\Aut(f_{u,v,w})$ is either~$C_1$ or~$C_2$.  Since
$\f=[X,-Y,Z]\in\Aut(f_{u,v,w})$, we conclude that a generic map in
the family satisfies~$\Aut(f_{u,v,w})\cong C_2$.

\section{Computation of Dynamical and Topological Degrees}
\label{section:dyntopdeg}
In this section we compute the dynamical and topological degrees of
the various maps in Table~\ref{table:maintable}.  The following
elementary results will be useful, especially in establishing that a
map is algebraically stable, i.e., satisfies $\l_1(f)=\deg(f)$.
Results of this sort, and much more, appear in~\cite{MR1369137}, but for the
convenience of the reader, we give the short proofs.

\begin{lemma}
\label{lemma:CinCrita}
Let $f:\PP^2\dashrightarrow\PP^2$ be a dominant rational map.
\begin{parts}
\Part{(a)}
If~$f$ is a morphism, then
\[
  \l_1(f)=\deg(f)\quad\text{and}\quad\l_2(f)=\deg(f)^2.
\]
\Part{(b)}
$\l_1(f)<\deg(f)$ if and only if there is a curve $\G\subset\PP^2$
and an integer $n\ge1$ such that $f^n(\G)\subseteq I(f)$.
\Part{(c)}
Let~$\G\subset\PP^2$ be a curve such that~$f(\G)$ consists of a
single point. Then $\G\subseteq\Crit(f)$.
\end{parts}  
\end{lemma}
\begin{proof}
(a)
This is standard and elementary.  
\par\noindent(b)\enspace
Let
\[
  f(X,Y,Z)=[F_1(X,Y,Z),G_1(X,Y,Z),H_1(X,Y,Z)],
\]
and define inductively  
\[
  [F_{n+1},G_{n+1},H_{n+1}] = 
  [F_n(F_1,G_1,H_1),G_n(F_1,G_1,H_1),H_n(F_1,G_1,H_1)].
\]
Then
\[
 f^n(X,Y,Z)=[F_n(X,Y,Z),G_n(X,Y,Z),H_n(X,Y,Z)],
\]
so $\l_1(f)<\deg(f)$ if and only if there exists an~$n$ such that
$F_n,G_n,H_n$ have a non-trivial common factor in~$K[X,Y,Z]$.  Taking
the smallest such~$n$, if~$R(X,Y,Z)$ is the common factor, then the
curve~$\G=\{R=0\}$ satisfies $f^{n-1}(\G)\subset I(f)$.
\par\noindent(c)\enspace This is just the chain rule.  We view~$\G$ as
an abstract curve with an embedding $j:\G\hookrightarrow\PP^2$.
Differentiating the constant function $f\circ j$ gives
$(Df\circ j)\cdot \nabla j = 0$. The fact that~$j$ is non-constant,
i.e.,~$\G$ is a curve, tells us that~$\nabla j(t)\ne\bfzero$ for all
but finitely many $t\in \G$, and hence that $\det
Df\bigl(j(t)\bigr)=0$ for all but finitely many $t\in \G$.  Taking
Zariski closures gives $\G=\Image(j)\subset \{\det Df=0\}=\Crit(f)$.
\end{proof}

The remainder of this section is devoted to computing the dynamical and topological
degrees of the maps in Table~\ref{table:maintable}.

\paragraph{\textbf{Types 1.1,\thinspace1.3,\thinspace1.5,\thinspace1.7.1.8,\thinspace3.4,\thinspace8.1}}
\leavevmode\newline
These maps are morphisms, so Lemma~\ref{lemma:CinCrita}
gives $\l_1(f)=2$ and $\l_2(f)=4$.

\paragraph{\textbf{Types 1.2}
  $\boldsymbol{(C_3)}:\;f=[X^2-YZ,Z^2-XY,Y^2-XZ]$}
\leavevmode\newline
A short calculation shows that $f^2(X,Y,Z)=[X,Y,Z]$, so~$\deg(f^n)$
alternates between~$2$ and~$1$. In particular, $\l_1(f)=\l_2(f)=1$.

\paragraph{\textbf{Types 1.4,\thinspace1.6}
  $\boldsymbol{(\Gcal_3)}:\;f=[X^2+bYZ,Z^2+dXY,Y^2+gXZ],\; bdg=-1$}
\leavevmode\newline
The critical locus has the equation $dg X^3 + bd Y^3+bg Z^3 - 3 XYZ=0$.
Using the assumption that~$bdg=-1$, we find that the critical
locus is the union of the three lines
\[
  L_k = \{(dg)^{1/3} X + \z_3^k (bd)^{1/3}Y + \z_3^{2k} (bg)^{1/3}Z=0\},
  \quad\text{$k=0,1,2$,}
\]
where~$\z_3$ is a primitive cube root of unity. 
We compute
\[
  f(L_k) = \bigl\{ [b^{2/3}, \z_3^k d^{2/3}, \z_3^{2k} g^{2/3}] \bigr\}
  \quad\text{for $k=0,1,2$.}
\]
In particular,~$f\bigl(\Crit(f)\bigr)$ consists of three points.
The indeterminacy locus of~$f$ is
\[
  I(f) = \bigl\{ [ b^{1/3}, g^{1/3}, d^{1/3} ],
  [ b^{1/3}, \z_3 g^{1/3}, \z_3^2 d^{1/3} ],
  [ b^{1/3}, \z_3^2 g^{1/3}, \z_3 d^{1/3} ] \bigr\},
\]
where~$b^{1/3}$ and~$d^{1/3}$ are arbitrary fixed cube roots of~$b$
and~$d$, and then~$g^{1/3}$ is set equal to~$-1/b^{1/3}d^{1/3}$.  Thus
for generic values of~$b,d,g$ (satisfying $bdg=-1$), the orbits of the
three points~$f(L_1),f(L_2),f(L_3)$ will not hit~$I(f)$, so
generically we have~$\l_1(f)=2$. However, we expect that there is a
countable set of  values of~$b,d,g$ such that~$\l_1(f)<2$. To see why, we observe that
for each $n\ge1$, the equation
\[
  f^n(L_0)= f^{n-1}(b^{2/3},d^{2/3},g^{2/3}) =[b^{1/3},g^{1/3},d^{1/3}] \in I(f)
\]
yields two homogeneous polynomial equations for the point
$[b^{1/3},g^{1/3},d^{1/3}]$.  Substituting the inhomogeneous value
$g^{1/3}=1/b^{1/3}d^{1/3}$ yields two inhomogeneous polynomials
equations for~$(b^{1/3},d^{1/3})$, and hence a finite number of values
for~$(b,d,g)$. Varying~$n$ should then yield a countable number of
exceptional values of~$(b,d,g)$ with~$f^n(L_0)\in I(f)$, and thus
with~$\l_1(f)<2$. This applies to Type~1.6, for which~$b,d,g$ satisfy
the single relation~$bdg=-1$. For maps of Type~1.4 with
$(b,d,g)=(b,b^{-1},-1)$, there is only one degree of freedom, so it
seems plausible that for these maps we have~$\l_1(f)=2$.

In order to compute the topological degree, we set $f(X,Y,Z)=[\a,\b,1]$
for generic~$\a,\b$ and solve for~$[X,Y,Z]$. This gives two equations
\[
  X^2 + b YZ = \a(Y^2+gXZ),\quad
  Z^2+dXY = \b(Y^2+gXZ).
\]
We dehomogenize $x=X/Z$ and $y=Y/Z$. Then we can solve the second equation
for~$x$ and substitute into the first equation to find
\[
  ( \b ^2-d^2 \a ) b^2d^2 y^4 + (d^2 b^2 - \b \a ) b d^2 y^3
  + ( \b ^2 - d^2 \a ) b y + (d^2 b^2 - \a \b ) = 0.
\]
The discriminant of this quartic equation is a mess, but part of it looks like
\[
  \Disc(f) = -27   b^8d^4 (\b^6 - b^4d^8)^2
         + \a \cdot \bigl(\text{polynomial in $\ZZ[b,d,\a,\b]$} \bigr).
\]
In particular, since $bdg=-1$, we see that for generic~$\a,\b$, the
quartic has distinct roots. If those roots lead to points not
in~$I(f)$, which we expect to be true for most~$(b,d,g)$ triples,
then~$\l_2(f)=\#f^{-1}(\a,\b)=4$. On the other hand, the general
inequality $\l_2\le\l_1^2$ shows that we should expect $\l_2<4$ for
countably many~$(b,d,g)$.

We illustrate with the extreme case $b=d=g=-1$, which is the map of
Type~1.2. In that case, the discriminant quartic factors (essentially)
as
\[
  (y^3-1) \left(y - \frac{1-\a\b}{\b^2-\a}\right).
\]
The three roots with~$y\in\bfmu_3$ lead to points in~$I(f)$, so we
find that $\#f^{-1}(\a,\b)=1$, which confirms our earlier computation.

\paragraph{\textbf{Type 2.1}
  $\boldsymbol{(\Gcal_3)}:\;f=[YZ,XY,Y^2+gXZ],\; g\ne0$}
\leavevmode\newline
Here $\Crit(f)$ is a line and a conic and~$\#I(f)=2$,
\[
\Crit(f) = \{Y=0\} \cup \{Y^2=gXZ\},\quad
  I(f) = \bigl\{ [0,0,1], [1,0,0] \bigr\}.
\]
We have $f\bigl(\{Y=0\}\bigr)=[0,0,1]\in I(f)$, so~$\l_1(f)<2$.  The
degree sequence of~$f$ is $2,3,5,8,13,21,34,\ldots$, which suggests
that~$\deg(f^n)$ is the $(n+2)$'nd Fibonacci number.

We dehomogenize $Y=1$, so $f(x,z)=\bigl(z/x,(1+gxz)/x\bigr)$.
Setting $f(x,z)=(\a,\b)$ with generic~$\a,\b$, we find
that $z=\a x$ and $\a gx^2-\b x+1=0$. Since $g\ne0$,
we have $\l_2(f)=\#f^{-1}(\a,\b)=2$.

\paragraph{\textbf{Types 2.2,\thinspace2.6}
  $\boldsymbol{(\Gcal_3)}:\;f=[YZ, cZ^2+XY, Y^2+gXZ],\; cg\ne0$}
\leavevmode\newline
The critical locus is a nodal cubic that never gets mapped to a point,
so~$\l_1(f)=2$. Dehomogenizing with respect to~$X$ and setting
$\bigl((cz^2+y)/yz, (y^2+gz)/yz \bigr)=(\a,\b)$ with generic~$\a,\b$ leads
to $y=cz^2/(\a z-1)$ and
$c(c-\a\b)z^3+(c\b-g\a^2)z^2 + 2\a g z - g = 0$.  (Note that $z=0$ is
not a valid value.) Hence $\l_2(f)=3$.

\paragraph{\textbf{Types 2.3}
  $\boldsymbol{(\Gcal_3)}:\;f=[YZ, cZ^2+XY, Y^2],\; c\ne0$}
\leavevmode\newline
Here $I(f)=\bigl\{[1,0,0]\bigr\}$ and $\Crit(f)=3\cdot\{Y=0\}$.
We have
\[
  \{Y=0\} \xrightarrow{\;f\;} [0,1,0]  \xrightarrow{\;f\;} [0,0,1]  \xrightarrow{\;f\;} [0,1,0],
\]
so although the critical locus maps to a point, that point is part of
a 2-cycle, so it never hits~$I(f)$. Hence~$\l_1(f)=2$. Further,
$f^{-1}(X,Y,Z)=]b^2YZ-X^2,b^2Z^2,bXZ]$, so~$f$ is birational and~$\l_2(f)=1$.

\paragraph{\textbf{Type 2.4}
  $\boldsymbol{(\Gcal_3)}:\;f=[aX^2+YZ,XY,Y^2],\; a\ne0$}
\leavevmode\newline
The critical locus is a triple line, $\Crit(f)=3\cdot\{Y=0\}$, and we
have $f\bigl(\{Y=0\}\bigr)=[1,0,0]\in \Fix(f)$. Hence by the usual
argument via Lemma~\ref{lemma:CinCrita}, we find that $\l_1(f)=2$.
Further, $f^{-1}(X,Y,Z)=[YZ,Z^2,XZ-aY^2]$ shows that~$f$ is birational,
so~$\l_2(f)=1$.

\paragraph{\textbf{Type 3.1}
  $\boldsymbol{(\Gcal_4)}:\;f=[Z^2,XY,Y^2],\;a=e=0$}
\leavevmode\newline
The dynamical degree of a monomial map may be computed using the
formula in~\cite{MR2358970}. In affine coordinates the map is
$f(x,y)=(y^{-2},xy^{-1})$ with exponent
matrix~$\SmallMatrix{0&-2\\1&-1}$.  Then~\cite{MR2358970} says
that~$\l_1(f)$ is the spectral radius of the exponent matrix, so
$\l_1(f)=\left|\frac{-1+\sqrt{-7}}{2}\right|=\sqrt2$.  And setting
$f(x,y)=(\a,\b)$ with generic~$\a,\b$, we see that~$f^{-1}(\a,\b)$
consists of the two points~$(\b\g,\g)$ with $\g^2=\a^{-1}$.
Hence~$\l_2(f)=2$. 


\paragraph{\textbf{Type 3.2}
  $\boldsymbol{(\Gcal_4)}:\;f=[Z^2,XY,Y^2+eXZ],\; e\ne0$}
\leavevmode\newline
We have
\[
  I(f) = \bigl\{[1,0,0]\bigr\}
  \quad\text{and}\quad
  \Crit(f) = \{Z=0\} \cup \{eXZ=2Y^2\}.
\]
Both $f\bigl(\{Z=0\}\bigr)$ and $f\bigl(\{eXZ=2Y^2\}\bigr)$ are
curves, so~$\l_1(f)=2$. Next we dehomogenize with $Z=1$.  Then
$f(x,y)=\bigl((y^2+ex)^{-1},xy(y^2+ex)^{-1}\bigr)$.  Setting
$f(x,y)=(\a,\b)$ with generic~$\a,\b$ leads to $x=\a^{-1}\b y^{-1}$
and $\a y^3-y+\b e=0$, so $\l_2(f) = \#f^{-1}(\a,\b) = 3$.

\paragraph{\textbf{Type 3.3}
  $\boldsymbol{(\Gcal_4)}:\;f=[aX^2+Z^2,XY,Y^2],\; a\ne0$}
\leavevmode\newline
We have
\[
  I(f) = \bigl\{[1,0,\pm\sqrt{-a}]\bigr\}
  \quad\text{and}\quad
  \Crit(f) = \{Z=0\} \cup 2\cdot \{Y=0\}.
\]
The orbit of the line $Y=0$ is
\[
  \{Y=0\} \xrightarrow{\;f\;} \bigl\{[1,0,0]\bigr\}
  \xrightarrow{\;f\;} \bigl\{[1,0,0]\bigr\},
\]
so the orbit of~$Y=0$ never lands in~$I(f)$. 
The image the line $Z=0$ is a curve,
\[
  \{Z=0\} \xrightarrow{\;f\;} \{XZ=aY^2\}.
 \]
Hence~$\l_1(f)=2$. Next we dehomogenize with $Z=1$.  Then
$f(x,y)=\bigl((ax^2+1)/y^2,x/y\bigr)$.  Setting
$f(x,y)=(\a,\b)$ with generic~$\a,\b$ leads to $x=\b y$
and $(a\b^2-a)y^2+1=0$. Thus $\l_2(f) = \#f^{-1}(\a,\b) = 2$.

\paragraph{\textbf{Types 4.1,\thinspace4.2,\thinspace4.3,\thinspace4.4}$
  \boldsymbol{(\Gcal_4)}:\;f=[YZ,X^2+cZ^2,XY]$}
\leavevmode\newline
Proposition~\ref{proposition:c4f2autgm} says that~$f$ satisfies
$\deg(f^n)\le2n+1$, so~$\l_1(f)=1$.  Since~$\l_2\le\l_1^2$ in general,
it follows that~$\l_2(f)=1$.  (Alternatvely, it is not hard to write
down the inverse map.)

\paragraph{\textbf{Types 5.1,\thinspace5.2,\thinspace5.3}
  $\boldsymbol{(\Gcal_{2,2})}:\;f=[Y^2-Z^2,XY,eXZ],\; e\ne0$}
\leavevmode\newline
Proposition~\ref{proposition:fY2Z2XYeXZ} tells us that $\deg(f^n)\le n+1$,
so just as in the previous case, we have~$\l_1(f)=\l_2(f)=1$.

\paragraph{\textbf{Types 5.4,\thinspace5.5}$
  \boldsymbol{(\Gcal_{2,2})}:\;f=[aX^2+Y^2-Z^2,XY,eXZ],\; ae\ne0$}
\leavevmode\newline
We have
\[
  I(f) = \bigl\{ [0,1,1], [0,1,-1] \bigr\}
\]
and
\[
  \Crit(f) = \{X=0\} \cup \{aX^2-Y^2+Z^2=0\}.
\]
The line~$X=0$ is sent to a fixed point of~$f$,
\[
  f\bigl(\{X=0\}\bigr) =  [1,0,0] \in \Fix(f),
\]
while the conic $aX^2-Y^2+Z^2=0$ is sent to another conic,
\begin{align*}
  f\bigl(\{aX^2-Y^2&+Z^2=0\}\bigr)\\
  &= \bigl\{ [aX^2+Y^2-Z^2,XY,eXZ] : aX^2-Y^2+Z^2=0 \bigr\} \\
  &= \bigl\{ [2aX^2,XY,eXZ] : aX^2-Y^2+Z^2=0 \bigr\} \\
  &= \bigl\{ [2aX,Y,eZ] : aX^2-Y^2+Z^2=0 \bigr\} \\
  &= \bigl\{ [u,v,w] : a(u/2a)^2 - v^2 + (w/e)^2 = 0 \bigr\} \\
  &= \bigl\{ [X,Y,Z] : X^2/4a - Y^2 + Z^2/e^2 = 0 \bigr\}.
\end{align*}
Lemma~\ref{lemma:CinCrita}(c) tells us that the only curve in~$\PP^2$
that~$f$ maps to a point is the line~$\{X=0\}$. Now suppose that
$\l_1(f)<2$.  Then Lemma~\ref{lemma:CinCrita}(b) says that there is a
curve~$\G\subset\PP^2$ and an~$n\ge1$ such that~$f^n(\G)\subset I(f)$.
In particular,~$f^n(\G)$ is a point, so there is some $0\le m< n$
such that $f^m(\G)$ is a curve and $f^{m+1}(\G)$ is a point. But we
have just shown that this forces $f^m(\G)$ to be the line $X=0$ and
forces $f^{m+1}(\G)$ to be the point~$[1,0,0]$, which is not
in~$I(f)$. This completes the proof that $\l_1(f)=2$.

To compute the topological degree~$\l_2(f)$, we compute the inverse
image of a generic point. To simplify the calculation, we compute
$[x,y,z]\in f^{-1}(\a,\b,e\g)$ with~$\a,\b,\g$ generic. From the last
two coordinates we have $e\g xy = \b exz$, so $x=0$ or $\g y=\b
z$. We cannot have~$x=0$, since~$f(0,Y,Z)=[1,0,0]$. Hence our point
has the form $[x,y,z]=[\b x,\b y,\b z]=[\b x,\b y,\g y]$, so the
homogeneous point~$[x,y]\in\PP^1$ determines~$[x,y,z]$. We next use
the first two coordinates to deduce that
$\b (ax^2+y^2-z^2) = \a d xy$. Multiplying this by~$\b$ and using
$\b z=\g y$ yields
$(\b^2 a x^2+\b^2y^2-\g^2y^2) = \a\b d xy$.
Since~$\a,\b,\g$ are generic, this equation has two solutions
in~$\PP^1$. This shows that $\#f^{-1}(\a,d\b,e\g)=2$, so~$\l_2(f)=2$.

\paragraph{\textbf{Types 2.5}
  $\boldsymbol{(\Gcal_3)}:\;f=[aX^2+YZ,XY,Y^2+gXZ],\; ag\ne0$}
\leavevmode\newline
Here $\Crit(f)$ is a nodal cubic curve and the image of $\Crit(f)$ is
also a curve.  It follows from Lemma~\ref{lemma:CinCrita} that
$\l_1(f)=2$, since no iterate of~$f$ maps a curve to a point, much
less to a point in~$I(f)$. To compute the topological degree, we
dehomogenize by setting $x=X/Y$ and $z=Z/Y$ and solving
$\bigl((ax^2+z)/x,(1+gxz)/x\bigr)=(\a,\b)$ for generic~$\a,\b$. The
first coordinate gives $z=\a x-a x^2$, and substituting into the
second coordinate gives $1+gx(\a x-a x^2)=\b x$. Hence $\l_2(f)=3$.

\paragraph{\textbf{Type 7.1}
  $\boldsymbol{(\Gcal_5)}:\;f=[YZ,X^2,Y^2]$}
\leavevmode\newline
The map~$f$ of Type~$7.1$ is also monomial, but there's an 
easier way to calculate~$\l_1(f)$.  We note that
$f^8=[X^{16},Y^{16},Z^{16}]$, so in particular~$f^8$ is a
morphism. Hence
\[
  \l_1(f)=\l_1(f^8)^{1/8} =\deg(f^8)^{1/8}=16^{1/8}=\sqrt2.
\]
Similarly $\l_2(f)=\l_2(f^8)^{1/8}=256^{1/8}=2$.


\begin{acknowledgement}
We would like to thank Dan Abramo\-vich, Friedrich Knop, and Kimball
Martin for their helpful advice. We would also like to thank the
referees for their comments, corrections, and suggestions, and
especially for pointing out an algebra error in the original proof of
Lemma~\ref{lemma:GinPGL3} that had caused us to miss the conjugacy
class~\eqref{eqn:Lemma10c3}.
\end{acknowledgement}



\appendix

\section{Maps with Infinite Automorphism Group}  
\label{section:quotientmap}
As indicated earlier, we have restricted attention in this paper to
maps whose automorphism group is finite. The reason is because maps
with infinite automorphism group decompose as fiber product maps, as
described by the following general construction.

\begin{definition}
Let $X$ and $Y$ be varieties, and let $f:X\to X$ be a dominant
rational map. We say that \emph{$f$ descends to $Y$} if there are
dominant rational maps $\pi:X\to Y$ and $g:Y\to Y$ such we have a
commutative diagram
\[
\begin{CD}
  X @>f>> X \\
  @VV\pi V @VV\pi V \\
  Y @>g>> Y\\
\end{CD}
\]
\end{definition}

We note that if $f:X\to X$ descends to~$Y$, then analyzing the
dynamics of~$f$ may be reduced, in some sense, to analyzing the
dynamics of $g:Y\to Y$ and the ``twisted'' dynamics on the fibers.  In
particular, if $1\le\dim(Y)<\dim(X)$, then one is reduced to lower
dimensional problems.

We let
\[
  \Aut(X,f):=\{\f\in\Aut(X):\f\circ f=f\circ\f\}
\]
and note that if the quotient of~$X$ by a
subgroup~$\Gcal\subseteq\Aut(X,f)$ is well-defined, then~$f$ descends
to the quotient $X/\Gcal$.  We give two examples.
  
\begin{example}
Let
\[
  f = [YZ,X^2-Z^2,XY]
\]
be the map of Type~4.1 with $c=-1$ in Table~\ref{table:maintable}.
Then~$\Aut(f)$ contains a copy of~$\GG_m$,
\[
  \left\{ \SmallMatrix{s&0&t\\ 0&1&0\\ t&0&s\\ } : s^2-t^2=1 \right\} \subset \Aut(f).
\]
The quotient of~$f$ by this subgroup yields the commutative diagram
\[
  \begin{CD}
    \PP^2 @>[YZ,X^2-Z^2,XY]>> \PP^2 \\
    @VV [X^2-Z^2,Y^2]V  @VV [X^2-Z^2,Y^2]V \\
    \PP^1 @>[U,V]\to[-V,U]>> \PP^1 \\
  \end{CD}
\]
Of course, the map $f$ is not very interesting dynamically, since $f^2=[X,-Y,Z]$.
\end{example}

\begin{example}
Here is a more interesting example. The automorphism group of the map
\[
  f : \PP^2\to\PP^2,\quad f(X,Y,Z)=[X^2+YZ,XY,XZ]
\]
contains a copy of~$\GG_m$ via
\[
  \left\{ \SmallMatrix{1&0&0\\ 0&t&0\\ 0&0&t^{-1}\\ } : t\ne0 \right\} \subset \Aut(f).
\]
The quotient of~$f$ gives the diagram
\[
  \begin{CD}
    \PP^2 @>[X^2+YZ,XY,XZ]>> \PP^2 \\
    @VV [X^2,YZ]V  @VV [X^2,YZ]V \\
    \PP^1 @>[U,V]\to[(U+V)^2,UV]>> \PP^1 \\
  \end{CD}
\]
\end{example}  

\section{Semistability for Subgroups}
\label{section:semistabsubgp}

In this section we give a general result for GIT semistability
relative to subgroups.  We thank Friedrich Knop for explaining how the
following result is a consequence of a general theorem of
Luna~\cite{MR0376704}.


\begin{proposition}
\label{proposition:XGXHNH}
We work over an algebraically closed field~$K$ of characteristic~$0$. Let
\begin{align*}
  X &= \text{a smooth projective variety,} \\
  G &= \text{a reductive group acting linearly on $X$,} \\
  H &= \text{a subgroup of $G$ that is also reductive,}\\
  N(H) &= \text{the normalizer of $H$ in $G$.}
\end{align*}
For $x \in X$, let
\begin{align*}
  \operatorname{Stab}(x) &= \{ g \in G : g \cdot x = x \},\\
  X^H &= \{ x \in X : \text{$\operatorname{Stab}(x)$ contains $H$}\}.
\end{align*}
\begin{parts}
\Part{(a)}
Let $x\in X^H$. Then
\[
  \text{$x$ is $N(H)$-semistable}
  \quad\Longleftrightarrow\quad
  \text{$x$ is $G$-semistable}.
\]  
\Part{(b)}
The map
\[
  (X^H)^\ss/N(H) \longrightarrow X^\ss/G
\]
is a finite map.
\Part{(c)}
Let $\tilde X$ be the affine cone over~$X$, and for~$x\in X^H$, let~$\tilde x\in\tilde X$
be a lift. Thus each~$x\in X^H$ determines a character~$\chi_x$ on~$H$ via
\[
  \chi_x : H \longrightarrow\GG_m,\quad
  h\cdot\tilde x = \chi_x(h)\tilde x.
\]
For each character~$\chi\in\hat H:=\Hom(H,\GG_m)$, let $X^H_\chi:=\{x\in X^H:\chi_x=\chi\}$.
Then~$(X^H)^\ss$ decomposes as a disjoint union
\[
  (X^H)^\ss = \bigcup_{\chi\in\hat H} (X^H_\chi)^\ss.
\]
\Part{(d)}
For $\chi\in\hat H$ and $\pi\in N(H)$, define~$\chi^\pi\in\hat H$ by
$\chi^\pi(h)=\chi(\pi^{-1}h\pi)$. Then there is an isomorphism
\[
  X^H_\chi \longrightarrow X^H_{\chi^\pi},\quad
  x \longmapsto \pi\cdot x.
\]
\end{parts}
\end{proposition}
\begin{proof} (a)
(Knop~\cite{MO243725})
Let $\tilde{X}$ be the affine cone over~$X$, and for~$x\in X^H$, let
$\tilde{x}\in\tilde{X}$ be a lift. By
definition~\cite{mumford:geometricinvarianttheory}, the point~$x$ is
$G$-semistable if the closure of the orbit~$G\tilde{x}$ does not
contain the vertex~$0$, and similarly for~$N:=N(H)$. From this it is
obvious that if~$x$ is~$G$-semistable, then it is also~$N$-semistable.

Conversely, suppose that~$x$ is $N$-semistable. After possibly
replacing~$\tilde{x}$ by a point in the (unique) closed $N$-orbit
$\overline{N\tilde{x}}$, we may assume that $N\tilde{x}$ is closed and
not equal to~$\{0\}$.

The group~$H$ acts on the line~$K\tilde{x}$ by a character~$\chi$, and the
assumption that~$x$ is $N$-semistable implies that~$\chi$ has finite order.
In particular, $\chi(H)$ is a finite subgroup of~$\GG_m$. Let
\[
  \tilde{G} := G\times\chi(H)
  \quad\text{and}\quad
  \tilde{H} := \bigl\{ (h,\chi(h)^{-1}) : h \in H \bigr\}.
\]
Then~$\tilde{G}$ acts on~$\tilde{X}$ and $\tilde{x}$ is fixed by
$\tilde{H}$. Also note that the normalizer~$\tilde{N}$ of~$\tilde{H}$
in~$\tilde{G}$ is of finite index in $N\times\chi(H)$. In particular,
the orbit~$\tilde{N}\tilde{x}$ is closed and not equal to~$\{0\}$.
Now apply~\cite[Corollary~1]{MR0376704}, which says that
\[
  \text{$\tilde{N}\tilde{x}$ closed}
  \quad\Longrightarrow\quad
  \text{$\tilde{G}\tilde{x}$ closed}.
\]
Hence~$0$ is not in the closure of~$\tilde{G}\tilde{x}$, and therefore~$x$
is $G$-semistable.  This proves~(a)
\par\noindent(b)\enspace
This is an immediate consequence of the main result of~\cite{MR0376704},
which in our terminology says that $\tilde{X}^{\tilde{H}}/\tilde{N}\to
\tilde{X}/\tilde{G}$ is finite.
\par\noindent(c)\enspace
The character~$\chi_x$ is well-defined, since if we choose some other
lift~$\tilde{x}'$ of~$x$, then $\tilde{x}'=c\tilde{x}$ for some
$c\ne0$. It remains to show that the union is disjoint.  So suppose
that $x\in X^H_\chi\cap X^H_{\chi'}$. It follows that
$\chi(h)\tilde{x}=\chi'(h)\tilde{x}$ for all $h\in H$, and since
$\tilde{x}\ne\bfzero$, we find that
$\chi(h)=\chi'(h)$. Hence~$\chi=\chi'$.
\par\noindent(d)\enspace
This follows from the calculation
\[
  h\cdot \pi\cdot\tilde{x}
  = \pi\cdot(\pi^{-1}\cdot h\cdot\pi)\cdot\tilde{x}
  = \pi\cdot \chi_x(\pi^{-1}\cdot h\cdot\pi)\tilde{x}
  =  \chi_x(\pi^{-1}\cdot h\cdot\pi)\pi\cdot\tilde{x},
\]
which shows that $\chi_{\pi\cdot x}=\chi_x^\pi$.
\end{proof}

\end{document}